\documentclass{amsart}
\pdfoutput=1
\usepackage{amssymb}
\usepackage{mathrsfs}
\usepackage{stmaryrd}
\usepackage{multirow}
\usepackage{hyperref}
\usepackage[shortlabels]{enumitem}

\usepackage{microtype}

\usepackage{comment}
\usepackage{tikz}
\usepackage{tikz-cd}
\usepackage[all,cmtip]{xy}
\usepackage{todonotes}
\usepackage{graphicx}

\usepackage{subcaption}

\newtheorem{theorem}[equation]{Theorem}
\newtheorem{lemma}[equation]{Lemma}
\newtheorem{proposition}[equation]{Proposition}
\newtheorem{cor}[equation]{Corollary}

\theoremstyle{definition}
\newtheorem{definition}[equation]{Definition}
\newtheorem{remark}[equation]{Remark}

\newtheorem{example}[equation]{Example}
\newtheorem{notation}[equation]{Notation}

\numberwithin{equation}{section}

\makeatletter
\newcommand*\bigcdot{\mathpalette\bigcdot@{.6}}
\newcommand*\bigcdot@[2]{\mathbin{\vcenter{\hbox{\scalebox{#2}{$\m@th#1\bullet$}}}}}
\makeatother

\DeclareMathOperator{\conv}{conv}
\DeclareMathOperator{\sign}{sign} 
\DeclareMathOperator{\lcm}{lcm} 
\DeclareMathOperator{\Der}{Der} 
\DeclareMathOperator{\Hom}{Hom} 
\DeclareMathOperator{\Spec}{Spec} 
 
\DeclareMathOperator{\link}{lk} 
\DeclareMathOperator{\id}{Id} 
\DeclareMathOperator{\coker}{coker} 
\DeclareMathOperator{\Ext}{Ext} 
\DeclareMathOperator{\Cone}{Cone} 
\DeclareMathOperator{\Flag}{Flag} 
\DeclareMathOperator{\Image}{im} 

\renewcommand\AA{\mathbb{A}}

\newcommand{\KK}{\Bbbk}
\newcommand\NN{\mathbb{N}}
\newcommand\RR{\mathbb{R}}
\newcommand\PP{\mathbb{P}}
\newcommand\ZZ{\mathbb{Z}}
\newcommand\cQ{\mathcal{Q}}
\newcommand\old{\mathrm{old}}
\newcommand\new{\mathrm{new}}

\newcommand\mfm{\mathfrak{m}}
\newcommand\wt{\mathbf{wt}}
\newcommand\cG{\mathcal{G}}
\newcommand\cA{\mathcal{A}}
\newcommand\cK{\mathcal{K}}
\newcommand\cR{\mathcal{R}}
\newcommand\cV{\mathcal{V}}
\newcommand\cE{\mathcal{E}}
\newcommand\ba{\mathbf{a}}
\newcommand\bb{\mathbf{b}}
\newcommand\bc{\mathbf{c}}
\newcommand\be{\mathbf{e}}

\newcommand\bu{\mathbf{u}}
\newcommand\bv{\mathbf{v}}
\newcommand\petal{\mathcal{P}}
\newcommand\twig{\mathcal{Q}}
\newcommand\std{\mathrm{std}}
\newcommand\Split{\mathfrak{S}}

\newcommand\git{ / \! \! / }
\newcommand{\p}[2]{\, ^{{\scalebox{0.4}[0.4]{(#1)}}} \! p_{#2}}
\renewcommand{\setminus}{\smallsetminus}

\title{Higher Cotangent Cohomology for Stanley-Reisner Rings}

\author{Nathan Ilten}
\address{Department of Mathematics,
Simon Fraser University,
8888 University Drive,
Burnaby BC V5A 1S6,
Canada}
\email{nilten@sfu.ca}

\author{Francesco Meazzini}
\address{Dipartimento di Matematica Guido Castelnuovo, Sapienza Universit\`a di Roma, Piazzale Aldo Moro 5, 00185 Roma, Italy}
\email{francesco.meazzini@uniroma1.it}

\author{Andrea Petracci}
\address{Dipartimento di Matematica, Universit\`a di Bologna, Piazza di Porta San Donato 5, 40126 Bologna, Italy}
\email{a.petracci@unibo.it}

\begin{document}
\begin{abstract}
	Inspired by work of Altmann and Christophersen, we study the graded pieces of the cotangent cohomology $T^i_{S_{\cK}}$, $i\geq 3$ of the Stanley-Reisner ring $S_\cK$ associated to a simplicial complex $\cK$. We prove a localization formula allowing one to reduce to the case of negative weights. Our results give a complete description of $T^3$ and $T^4$ in terms of the topology of $\cK$ whenever $\cK$ is a flag complex. As an application, we give a sufficient criterion for the vanishing of $T^3$ for simplicial spheres, classify two-spheres that have vanishing $T^3$, and show that the boundary complex of the dual associahedron has vanishing $T^3$. Our results make use of the arborescent resolutions considered by Hancharuk, Laurent-Gengoux, and Strobl. We give an alternative and self-contained treatment of these resolutions that may be of independent interest.
\end{abstract}

\maketitle
\tableofcontents

\section{Introduction}
\subsection{Background} Let $S$ be a finitely generated (commutative) algebra over a field $\KK$ of characteristic zero. The cotangent cohomology modules $T^i_{S/\KK}$ for $i\in\ZZ_{\geq 0}$ introduced in \cite{andre} are finitely generated $S$-modules that encode aspects of the deformation theory of $\Spec S$: $T^0_{S/\KK}$ is just the set of $\KK$-linear derivations from $S$ to itself, $T^1_{S/\KK}$ is the space of isomorphism classes of first order deformations of $\Spec S$ over $\KK$, $T^2_{S/\KK}$ contains the obstructions to lifting infinitesimal deformations of $\Spec S$, and the higher cotangent cohomology $T^i_{S/\KK}$ for $i>2$ is relevant for deformations of complete intersections in $\Spec S$, see e.g.~\cite{behnkechristophersen,iltenchristophersen}. Likewise, the higher cotangent cohomology modules $T^i_{S/\KK}$ are important in derived deformation theory, see e.g.~\cite[\S6]{pantev_vezzosi}.
For an example geometric application in our setting, see Corollary \ref{cor:codim2} in \S\ref{sec:grass}.

While the modules $T^i_{S/\KK}$ for $i=0,1,2$ have relatively concrete descriptions, they can still be challenging to explicitly compute in practice. For the higher cotangent cohomology modules the situation is even more challenging: in characteristic zero they can be described via Harrison cohomology \cite[Proposition~4.5.13]{loday}, but this is not particularly amenable to explicit computations. It is thus an interesting problem to give more explicit descriptions of the $T^i_{S/\KK}$ for special classes of rings $S$.

Along these lines, Altmann and Christophersen \cite{ac1} studied the cotangent cohomology modules $T^1_{S/\KK}$ and $T^2_{S/\KK}$ when $S$ is a Stanley-Reisner ring, that is, a quotient of a polynomial ring by a square-free\footnote{i.e.~radical}  monomial ideal. Recall that to any simplicial complex $\cK$ on vertex set  $\cV$, there is an associated Stanley-Reisner ideal $I_\cK\subseteq \KK[x_v \ | \ v\in\cV]$, see \S\ref{sec:SR}. The ideal $I_\cK$ is precisely the square-free monomial ideal generated by products of variables whose corresponding vertices do not belong to a common face, and all square-free monomial ideals arise in this way.
Hence for a fixed finite set $\cV$, this establishes a one-to-one correspondence between simplicial complexes with vertex set contained in $\cV$ and square-free monomial ideals in the polynomial ring $\KK [x_v \ | \ v \in \cV]$.

Denote by $S_\cK$ the Stanley-Reisner ring \[S_\cK:=\KK[x_v\ | \ v\in\cV]/I_\cK.\]
We set
\[
T^i(\cK) := T^i_{S_\cK/\KK}
\]
for every $i \in \ZZ_{\geq 0}$.
The ring $S_\cK$ has a natural $\ZZ^{\cV}$-grading; this descends to a $\ZZ^{\cV}$-grading on the cotangent cohomology modules as well. For any $\bc\in\ZZ^{\cV}$, we denote by $T^i_{\bc}(\cK)$ the weight\footnote{We use the term \emph{weight} here instead of degree to differentiate from the (co)homological degrees that will occur elsewhere in the paper.} $\bc$ piece of $T^i(\cK)$. The main contribution of \cite{ac1} is a description for $i=1,2$ of $T^i_{\bc}(\cK)$ in terms of the $(i-1)$st relative cohomology groups of some subsets determined by $\bc$ of the cone over the topological realization of $\cK$ (cf.~\cite[Theorem 13]{ac1}.

Inspired by the results of Altmann and Christophersen, the goal of this paper is to provide general arguments that also allow us to study the higher cotangent cohomology modules $T^i_{S/\KK}$, $i>2$ when $S$ is a Stanley-Reisner ring.

\subsection{General results}\label{sec:overview}
To summarize our results, we introduce a little more notation.

\begin{notation} \label{notation}
For every $\ba \in \ZZ_{\geq 0}^\cV$, we denote by $a \subseteq \cV$ the \emph{support} of $\ba$, i.e.\ the subset of $\cV$ consisting of the indexes $v \in \cV$ such that the $v$-th entry $\ba_v$ of $\ba$ is non-zero.
Similarly, $b$ denotes the support of $\bb \in \ZZ_{\geq 0}^\cV$.
Every weight vector $\bc\in \ZZ^\cV$ can be uniquely written as $\bc=\ba-\bb$ where $\ba,\bb\in \ZZ_{\geq 0}^\cV$ are such that $a \cap b = \emptyset$;
in all statements we will always assume this.
\end{notation}

If $\cK$ is a simplicial complex, recall that the \emph{link} of a face $f\in \cK$ is
\[
	\link(f,\cK)=\{g\in \cK\ |\ f\cap g=\emptyset\ \textrm{and}\ f\cup g\in \cK\}.
\]
Our first result is the following localization formula, generalizing \cite[Proposition 11]{ac1} to the cases $i>2$:

\begin{theorem}[Localization formula, Theorem~\ref{thm:localizationformula}]\label{thm:localization}
Let $\cK$ be a simplicial complex on vertex set $\cV$.
Let $i \in \ZZ_{\geq 1}$. Let $\bc=\ba-\bb\in \ZZ^\cV$ as in Notation~\ref{notation}.
Then:
\begin{enumerate}
    \item if $a \notin\cK$, then $T^i_\bc(\cK)=0$;
    \item if $a \in \cK$, then $T^i_\bc(\cK) \cong T^i_{-\bb}(\link(a,\cK))$.
\end{enumerate}
\end{theorem}

In \cite[Lemma 2]{ac1}, Altmann and Christophersen show that $T^i_{\ba-\bb}(\cK)$ for $i=1,2$ vanishes unless $\bb$ only has $0,1$-entries. This is not true for higher cotangent cohomology (see Example \ref{ex:211}), but we are able to show the following bound:

\begin{proposition}[See Theorem \ref{thm:bound}]\label{prop:bound}
Let $\cK$ be a simplicial complex on vertex set $\cV$.
Let $i \in \ZZ_{\geq 2}$, $\ba, \bb \in \ZZ_{\geq 0}^\cV$, $\bc=\ba-\bb\in \ZZ^\cV$.
If $\bb\notin \{0,\ldots,i-1\}^\cV$, then 
$T^i_{\bc}=0$.
\end{proposition}

For any simplicial complex $\cK$, 
Theorem~\ref{thm:localization} and Proposition~\ref{prop:bound} allow us to reduce
the computation of $T^i_{S_\cK/\KK}$ to the computation of $T^i_{-\bb}(\cK')$ as $\cK'$ ranges over the links of $\cK$ and $\bb$ ranges over finitely many possibilities. Such $T^i_{-\bb}(\cK')$ can be computed as the cohomology of an explicit complex of finite dimensional vector spaces (see \S\ref{sec:cot}), but the terms in this complex are often too large for computations to be practical.

One might hope to realize $T^i_\bc(\cK)$ for $i>2$ as the $(i-1)$st relative cohomology of some subsets of the cone over the topological realization of $\cK$, generalizing the results of \cite{ac1} for $i=1,2$. However, this cannot be possible, as there are simplicial complexes $\cK$ such that $T^i_{S_\cK/\KK}\neq 0$ for $i$ arbitrarily large (see Example \ref{ex:rnc3}).
Instead, we will have to settle for describing $T^i_{\ba-\bb}(\cK)$ for special choices of $\bb$. By the localization formula, we can and will always reduce to the case $\ba=0$.

For $\bb\in\ZZ_{\geq 0}^\cV$ with non-zero entries $j_1,\ldots,j_k$ listed in non-increasing order, we say $\bb$ has \emph{type} $(j_1,j_2,\ldots,j_k)$.
We define $|\bb|$ to be the sum $j_1+\ldots+j_k$.
Given any subset $b\subseteq\cV$,
set
\[
	\cK_b=\left\{f\in \cK \ \Big|\ f\cap \big(b \cup \cV(\bigcap_{v\in b}\link (v,\cK))\big)=\emptyset\right\}.
\]
Here, $\cV(\bigcap_{v\in b}\link (v,\cK))$ denotes the set of vertices of $\bigcap_{v\in b}\link (v,\cK)$.
The set $\cK_b$ is a subcomplex of $\cK$. When considering cohomology groups of complexes such as $\cK_b$, we will implicitly identify $\cK_b$ with its topological realization, see \S\ref{sec:basic}.

Our first result along these lines is the following:
\begin{theorem}[See \S\ref{sec:b1}]\label{thm:b1}
Assume that every element of $\cV$ belongs to $\cK$ and consider $\bb\in\ZZ_{\geq 0}^\cV$ of type $(1)$. Then for $i\geq 1$,
\[
	T^i_{-\bb}(\cK)\cong \widetilde H^{i-1}(\cK_b,\KK),
\]
the $(i-1)$st reduced cohomology of $\cK_b$.
\end{theorem}

\subsection{Results for flag complexes}
For our further results, we will restrict to the case that $I_\cK$ is generated by quadrics. This is equivalent to $\cK$ being a \emph{flag complex}, see \S\ref{sec:flag}.
In this situation, we have a significant improvement on Proposition \ref{prop:bound}.
\begin{proposition}\label{prop:bound2}
	Let $\cK$ be a flag complex. Then $T^i_{\ba-\bb}(\cK)=0$ if $i\geq 1$ and $|\bb|> i+1$, or if $i\geq 3$ and $\bb\notin \{0,\ldots,i-2\}^\cV$.
\end{proposition}
We also show that $T^i_{\ba-\bb}(\cK)=0$ if $|\bb|= i+1$ for $i\leq 4$  and conjecture that this holds for all $i$ (see Remark \ref{rem:koszul}, and Propositions \ref{prop:t3bound} and \ref{prop:t4bound}).

We likewise obtain:
\begin{theorem}[See \S\ref{sec:b2}]\label{thm:b2}
	Consider $\bb\in\ZZ_{\geq 0}^\cV$ of type $(1,1)$. Suppose that $\cK$ is a flag complex. If $b\in\cK$, then $T^i_{-\bb}(\cK)=0$ for $i\geq 1$. If $b\notin\cK$, then for $i\geq 2$,
\[
	T^i_{-\bb}(\cK)\cong \widetilde H^{i-2}(\cK_b,\KK),
\]
the $(i-2)$nd reduced cohomology\footnote{Recall that $\widetilde{H}^{-1}(\cK_b,\KK)$ vanishes unless $\cK_b=\emptyset$, in which case it is $\KK$.} of $\cK_b$.
\end{theorem}

To describe our results for $\bb$ of type $(1,1,1)$, we need a different simplicial complex $\widehat\cK_b$. Recall that the \emph{join} $\cK_1*\cK_2$ of simplicial complexes $\cK_1$ and $\cK_2$ on disjoint vertex sets is the complex whose faces are of the form $f_1\cup f_2$, where $f_1\in \cK_1$ and $f_2\in \cK_2$. Let $\Delta(\cE_b)$ be the full simplex whose vertices correspond to $2$-element subsets of $b$ not contained in $\cK$.
We define $\widehat\cK_b$ to be the subcomplex of $\Delta(\cE_b)*\cK_b$ arising as the union of 
\[
	\Delta(\cE_b)^{(0)}*\cK_b
\]
with
\[
	\Delta(f)*(\link(\{w_f\}, \cK)\cap \cK_b)
\]
as $f$ ranges over all edges of $\Delta(\cE_b)$. Here, $\Delta(\cE_b)^{(0)}$ is the $0$-skeleton of $\Delta(\cE_b)$, $\Delta(f)$ is the simplicial complex consisting of $f$ and all its subsets, and $w_f$ is the element of $b$ contained in both vertices of $f$.

\begin{theorem}[See \S\ref{sec:b3}]\label{thm:b3}
	Consider $\bb\in \ZZ_{\geq 0}^\cV$ of type $(1,1,1)$. Suppose that $\cK$ is a flag complex.
If more than one edge of $\cK$ is a subset of $b$, then $T^i_{-\bb}(\cK)=0$ for $i=3,4$. If instead at most one edge of $\cK$ is a subset of $b$, then for $i=3,4$
\[
	T^i_{-\bb}(\cK)\cong H^{i-2}(\widehat\cK_b,\KK).
\]
\end{theorem}
The above results combine to give a complete topological description of $T^3_{\bc}(\cK)$ when $\cK$ is a flag complex.
In fact, similar statements also apply when $\cK$ is not a flag complex as long as one chooses the weight $\bc$ carefully, see Remarks \ref{rem:nonflagb2} and \ref{rem:nonflagb3}.

For describing $T^4(\cK)$ for a flag complex $\cK$, it turns out that the interesting remaining case of $\bb$ to deal with are types $(1,1,1,1)$ and $(2,1,1)$, see Proposition \ref{prop:t4bound}. We deal with these in \S\ref{sec:b211} and \S\ref{sec:b1111}, also obtaining a complete topological description of $T^4_{\bc}(\cK)$ when $\cK$ is a flag complex.

We illustrate these results with a simple example.
\begin{example}\label{ex:hex}
	We consider the simplicial complex $\cK$ pictured in Figure \ref{fig:hex1}
	on the vertex set $\cV=\{1,2,\ldots,6\}$. This is a flag complex. The corresponding Stanley-Reisner ideal is generated by $x_{i-1}x_{i+1}$ for $i=1,\ldots 6$ and $x_ix_{i+3}$ for $i=1,2,3$, with indices taken modulo $6$. 

	The proper links of $\cK$ are either the empty set or a disjoint union of two points, corresponding to the rings $\KK$ and $\KK[x_1,x_2]/(x_1x_2)$. These both have vanishing $T^i$ for $i\geq 2$ (the latter ring is a complete intersection ring, see also Example~\ref{ex.onegenerator}). By Theorem \ref{thm:localization}, we obtain that $T_\bc^i(\cK)\neq 0$ for some $i\geq 2$ only if $\ba=0$. We now consider possibilities for $\bb$.

For $\bb$ of type $(1)$, up to symmetry $\cK_b$ looks as depicted in Figure \ref{fig:hex2}.
This is contractible, so by Theorem \ref{thm:b1}, $T^i_{-\bb}(\cK)=0$ for $i\geq 2$.

For $\bb$ of type $(1,1)$ and $b\notin\cK$, up to symmetry $\cK_b$ looks as depicted in Figure \ref{fig:hex2} or Figure \ref{fig:hex3}. By Theorem \ref{thm:b2}, 
$T_{-\bb}^i(\cK)=0$ for $i\geq 2$ in the former case and $i\geq 3$ in the latter case, and in the latter case $T_{-\bb}^2(\cK)\cong \KK$.
By the bounds on possible $\bb$ (Proposition \ref{prop:bound2} and the sentence immediately following it), it follows that $T^2(\cK)$ is three dimensional, supported in degrees $(-1,0,0,-1,0,0)$. $(0,-1,0,0,0,0,-1,0)$ and $(0,0,-1,0,0,-1)$. This is also described in \cite[Example 17]{ac1}.

For $\bb$ of type $(1,1,1)$ and at most one edge of $\cK$ contained in $b$, up to symmetry $\cK_b$ looks as depicted in Figure \ref{fig:hex4} or Figure \ref{fig:hex5}. In the former case $\Delta(\cE_b)$ is a $1$-simplex and
the complex $\widehat \cK_b$ is depicted in Figure \ref{fig:tkb1}. This is clearly contractible, hence $T^3_{-\bb}(\cK)=T^4_{-\bb}=0$ in this case.
In the latter case $\Delta(\cE_b)$ is a $2$-simplex and
the complex $\widehat \cK_b$ is depicted in Figure \ref{fig:tkb2}. This deformation retracts to the boundary of $\Delta(\cE_b)$, hence $T^3_{-\bb}(\cK)\cong \KK$ and $T^4_{-\bb}=0$.

By the above, Proposition \ref{prop:bound2}, and the sentence immediately following it, we see that $T^3(\cK)\cong \KK^2$, supported in degrees $-(1,0,1,0,1,0)$ and $-(0,1,0,1,0,1)$. We will conclude computing $T^4(\cK)$ in Examples \ref{ex:hex211} and \ref{ex:hex1111}.
\end{example}

\begin{figure}
\begin{subfigure}{.3\textwidth}
	\begin{tikzpicture}
		\draw[fill] (0,0) circle [radius=0.05] node[below left] {$1$};
		\draw[fill] (2,0) circle [radius=0.05] node[below right] {$4$};
		\draw[fill] (.5,1) circle [radius=0.05] node[above] {$6$};
		\draw[fill] (1.5,1) circle [radius=0.05] node[above] {$5$};
		\draw[fill] (.5,-1) circle [radius=0.05] node[below] {$2$};
		\draw[fill] (1.5,-1) circle [radius=0.05] node[below] {$3$};
\draw (0,0) -- (.5,-1) -- (1.5,-1) -- (2,0) -- (1.5,1) -- (.5,1) -- (0,0);
	\end{tikzpicture}
	\caption{The complex $\cK$}\label{fig:hex1}
\end{subfigure}
\begin{subfigure}{.3\textwidth}
	\begin{tikzpicture}
		\draw[fill,lightgray] (0,0) circle [radius=0.05] node[below left] {$1$};
		\draw[fill,lightgray] (2,0) circle [radius=0.05] node[below right] {$4$};
		\draw[fill,lightgray] (.5,1) circle [radius=0.05] node[above] {$6$};
		\draw[fill,lightgray] (1.5,1) circle [radius=0.05] node[above] {$5$};
		\draw[fill,lightgray] (.5,-1) circle [radius=0.05] node[below] {$2$};
		\draw[fill,lightgray] (1.5,-1) circle [radius=0.05] node[below] {$3$};
		\draw[lightgray] (0,0) -- (.5,-1) -- (1.5,-1) -- (2,0) -- (1.5,1) -- (.5,1) -- (0,0);
		\draw[fill] (2,0) circle [radius=0.05] node[below right] {$4$};
		\draw[fill] (1.5,1) circle [radius=0.05] node[above] {$5$};
		\draw[fill] (1.5,-1) circle [radius=0.05] node[below] {$3$};
\draw (1.5,-1) -- (2,0) -- (1.5,1) ;
	\end{tikzpicture}
	\caption{$\cK_{\{1\}}=\cK_{\{2,6\}}$}\label{fig:hex2}
\end{subfigure}
\begin{subfigure}{.3\textwidth}
	\begin{tikzpicture}
		\draw[fill,lightgray] (0,0) circle [radius=0.05] node[below left] {$1$};
		\draw[fill,lightgray] (2,0) circle [radius=0.05] node[below right] {$4$};
		\draw[fill,lightgray] (.5,1) circle [radius=0.05] node[above] {$6$};
		\draw[fill,lightgray] (1.5,1) circle [radius=0.05] node[above] {$5$};
		\draw[fill,lightgray] (.5,-1) circle [radius=0.05] node[below] {$2$};
		\draw[fill,lightgray] (1.5,-1) circle [radius=0.05] node[below] {$3$};
		\draw[lightgray] (0,0) -- (.5,-1) -- (1.5,-1) -- (2,0) -- (1.5,1) -- (.5,1) -- (0,0);
		\draw[fill] (.5,1) circle [radius=0.05] node[above] {$6$};
		\draw[fill] (1.5,1) circle [radius=0.05] node[above] {$5$};
		\draw[fill] (.5,-1) circle [radius=0.05] node[below] {$2$};
		\draw[fill] (1.5,-1) circle [radius=0.05] node[below] {$3$};
\draw (.5,-1) -- (1.5,-1);
\draw (1.5,1) -- (.5,1);
	\end{tikzpicture}
	\caption{$\cK_{\{1,4\}}$}\label{fig:hex3}
\end{subfigure}
\begin{subfigure}{.3\textwidth}
	\begin{tikzpicture}
		\draw[fill,lightgray] (0,0) circle [radius=0.05] node[below left] {$1$};
		\draw[fill,lightgray] (2,0) circle [radius=0.05] node[below right] {$4$};
		\draw[fill,lightgray] (.5,1) circle [radius=0.05] node[above] {$6$};
		\draw[fill,lightgray] (1.5,1) circle [radius=0.05] node[above] {$5$};
		\draw[fill,lightgray] (.5,-1) circle [radius=0.05] node[below] {$2$};
		\draw[fill,lightgray] (1.5,-1) circle [radius=0.05] node[below] {$3$};
		\draw[lightgray] (0,0) -- (.5,-1) -- (1.5,-1) -- (2,0) -- (1.5,1) -- (.5,1) -- (0,0);
		\draw[fill] (.5,1) circle [radius=0.05] node[above] {$6$};
		\draw[fill] (1.5,1) circle [radius=0.05] node[above] {$5$};
		\draw[fill] (1.5,-1) circle [radius=0.05] node[below] {$3$};
\draw (1.5,1) -- (.5,1);
	\end{tikzpicture}
	\caption{$\cK_{\{1,2,4\}}$}\label{fig:hex4}
\end{subfigure}
\begin{subfigure}{.3\textwidth}
	\begin{tikzpicture}
		\draw[fill,lightgray] (0,0) circle [radius=0.05] node[below left] {$1$};
		\draw[fill,lightgray] (2,0) circle [radius=0.05] node[below right] {$4$};
		\draw[fill,lightgray] (.5,1) circle [radius=0.05] node[above] {$6$};
		\draw[fill,lightgray] (1.5,1) circle [radius=0.05] node[above] {$5$};
		\draw[fill,lightgray] (.5,-1) circle [radius=0.05] node[below] {$2$};
		\draw[fill,lightgray] (1.5,-1) circle [radius=0.05] node[below] {$3$};
		\draw[lightgray] (0,0) -- (.5,-1) -- (1.5,-1) -- (2,0) -- (1.5,1) -- (.5,1) -- (0,0);
		\draw[fill] (.5,1) circle [radius=0.05] node[above] {$6$};
		\draw[fill] (2,0) circle [radius=0.05] node[below right] {$4$};
		\draw[fill] (.5,-1) circle [radius=0.05] node[below] {$2$};
	\end{tikzpicture}
	\caption{$\cK_{\{1,3,5\}}$}\label{fig:hex5}
\end{subfigure}\\
\begin{subfigure}{.3\textwidth}
	\begin{tikzpicture}
		\draw[fill,lightgray] (-1,0) -- (0,-1) -- (1,0) -- (0,2) -- (-1,0);
		\draw[fill] (-1,0) circle [radius=0.05] node[ left] {$\{1,4\}$};
		\draw[fill] (1,0) circle [radius=0.05] node[ right] {$\{2,4\}$};
		\draw[fill] (0,1) circle [radius=0.05] node[above right] {$5$};
		\draw[fill] (0,2) circle [radius=0.05] node[above right] {$6$};
		\draw[fill] (0,-1) circle [radius=0.05] node[below right] {$3$};
		\draw (-1,0) -- (0,2) -- (1,0) -- (0,1) -- (-1,0) -- (1,0) -- (0,-1) -- (-1,0);
		\draw (0,1) -- (0,2);
	\end{tikzpicture}
	\caption{$\widehat{\cK}_{\{1,2,4\}}$ }\label{fig:tkb1}
\end{subfigure}\qquad\qquad
\begin{subfigure}{.3\textwidth}
	\begin{tikzpicture}
		\pgfsetfillopacity{0.3}
		\draw[fill,red] (0,2) -- (-2,-2) -- (2,0) -- (-1,-1) -- (0,2);
		\draw[fill,blue] (-1,-1) -- (-.2,3) -- (2,0) -- (0,2) -- (-1,-1);
		\draw[fill,green] (-1,-1) -- (3,-.2) -- (0,2) -- (2,0) -- (-1,-1);
		\pgfsetfillopacity{1}
		\draw[fill] (-1,-1) circle [radius=0.05] node[ above right=.4cm] {$\hspace{-.1cm}\{1,3\}$};
		\draw[fill] (2,0) circle [radius=0.05] node[ left=.35cm] {$\{3,5\}$};
		\draw[fill] (0,2) circle [radius=0.05] node[below=.5cm] {$\ \ \ \{1,5\}$};
		\draw[fill] (-.2,3) circle [radius=0.05] node[above right] {$6$};
		\draw[fill] (3,-.2) circle [radius=0.05] node[below right] {$4$};
		\draw[fill] (-2,-2) circle [radius=0.05] node[below right] {$2$};
		\draw (-1,-1) -- (2,0) -- (0,2) -- (-1,-1);
		\draw (0,2) -- (3,-.2) -- (2,0) -- (-2,-2) -- (-1,-1) -- (-.2,3) -- (0,2);
		\draw (0,2) -- (-2,-2) -- (-1,-1) -- (3,-.2) -- (2,0) -- (-.2,3) -- (0,2);
	\end{tikzpicture}

	\caption{$\widehat{\cK}_{\{1,3,5\}}$}\label{fig:tkb2}
\end{subfigure}
	\caption{Simplicial complexes from Example \ref{ex:hex}}\label{fig:hex}
\end{figure}
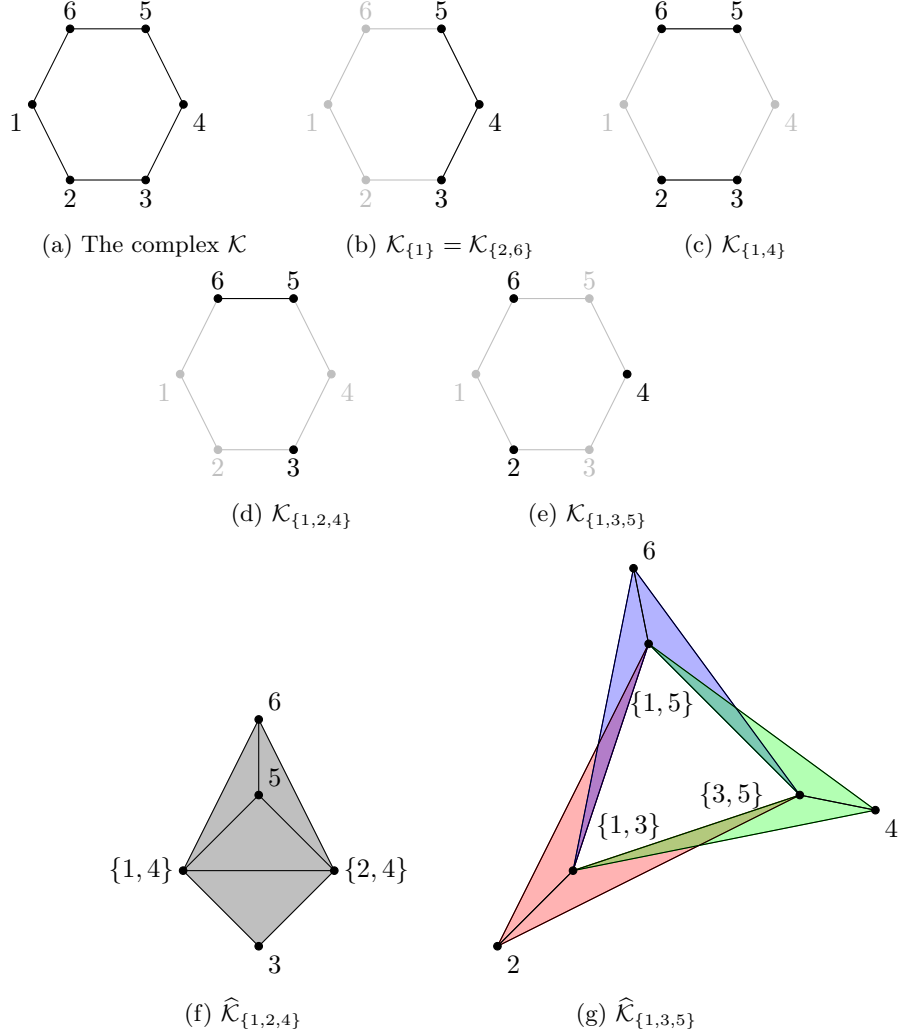
\subsection{Applications}\label{sec:app}
For applications in deformation theory, we are especially interested in situations where cotangent cohomology vanishes. Christophersen and the first author
give concise sufficient criteria for $T^2$-vanishing when $\cK$ is a flag complex that is topologically a sphere, see \cite[\S5]{ic1} .
Using our results, we obtain similar criteria for $T^3$-vanishing:

\begin{cor}[See \S\ref{sec:vanishing}]\label{cor:sphere}
Let $\cK$ be a flag complex whose topological realization is homeomorphic to a sphere. Suppose that $\cK$ satisfies the following properties:
\begin{enumerate}
	\item for every vertex $v$ of $\cK$, $T^3(\link(v,\cK))=0$; and
	\item for $v,w$ vertices of $\cK$ with $\{v,w\}\notin \cK$,
		either
		$\link(v,\cK)\cap \link(w,\cK)$ is contractible, or $\link(v,\cK)\cap \link(w,\cK)=\emptyset$ and $\dim \cK\neq 2$; and
	\item for $u,v,w$ vertices of $\cK$ with $\{u,v\}\notin \cK$ and $\{u,w\}\notin\cK$, 
		\begin{enumerate}
			\item \label{hyp:a} either $\link(u,\cK)\cap \link(v,\cK)\cap \link(w,\cK)$ is contractible, or $\link(u,\cK)\cap \link(v,\cK)\cap \link(w,\cK)=\emptyset$ and $\dim \cK\neq 1$; and 
			\item \label{hyp:b} there exists a vertex $u'\in \link(u,\cK)$ such that either $\{u',v\}$ or $\{u',w\}$ is not in $\cK$.
		\end{enumerate}
		\end{enumerate}
Then $T^3(\cK)=0$.
\end{cor}

In \cite[Corollary 2.5]{ishida}  there is a classification of all triangulations of the $2$-sphere with vanishing $T^2$. We obtain here classifications of triangulations of the $2$-sphere with vanishing $T^3$.
\begin{cor}\label{cor:spheret3}
Let $\cK$ be a simplicial complex that is homeomorphic to the $2$-sphere.
Then $T^3(\cK)=0$ if and only if $\cK$ is one of the $8$ triangulations of the $2$-sphere with at most $9$ vertices and  each vertex having valency at most $5$, excluding the triangulation with $8$ vertices that is not a flag complex. See Figure \ref{fig:triangulations} in \S\ref{sec:spheres}.
Moreover, for those triangulations $\cK$ with $T^3(\cK)=0$, $T^4(\cK)=0$ if and only if $\cK$ has at most $8$ vertices. 
\end{cor}

For $n\geq4$, let $\cA_n$ denote the simplicial complex whose set of vertices are the diagonals of a regular $n$-gon, and whose faces consist of collections of non-crossing diagonals, see \S\ref{sec:assoc}. This is the boundary of the polytope dual to the $(n-4)$-dimensional associahedron. Christophersen and the first author \cite{ic1} show that $T^2(\cA_n)=0$. We may extend this:
\begin{cor}\label{cor:assoc}
For $n\geq 4$, $T^3(\cA_n)=0$.
\end{cor}

\noindent In Corollary \ref{cor:codim2} in \S\ref{sec:grass} we illustrate one geometric consequence of the above vanishing.

We note that since $S_{\cA_4}$ is a complete intersection, $T^i(\cA_4)=0$ for all $i\geq 2$. Using our results, we may compute that  $T^4(\cA_5)=0$, see Example \ref{ex:pentagon}. On the other hand, $T^4(\cA_n)\neq 0$ for $n\geq 6$, see Remark \ref{rem:assoct4}.

\subsection{Method}
The cotangent cohomology modules $T^i_{S/\KK}$ may be defined as follows. Let $P_\bullet$ be a Koszul-Tate resolution of $S$, that is, a semifree commutative differential (non-negatively) graded algebra (cdga) that is quasi-isomorphic to $S$, see \S\ref{sec:kt}. Then $T^i_{S/\KK}$ is the $i$th cohomology of the complex $\Der^\bullet(P_\bullet,S)$ of graded derivations. 

The first challenge is finding an explicit Koszul-Tate resolution $P_\bullet$ of $S$. Viewing $S$ as a quotient of a polynomial ring $R$, Hancharuk, Laurent-Gengoux, and Strobl \cite{arborescent} show how to construct such $P_\bullet$ by starting with a free resolution of $S$ as an $R$-module only making finitely many choices along the way. They call the resulting resolution an \emph{arborescent resolution} of $S$. In \S\ref{sec:inductive}, we give an alternative inductive construction of arborescent resolutions that avoids some of the technicalities of \cite{arborescent}.

When $S$ is the quotient of a polynomial ring $R$ by a monomial ideal $I$, $S$ has a canonical free resolution as an $R$-module, the so-called Taylor resolution (see \S\ref{sec:taylor}). In fact, the Taylor resolution has the structure of a cdga, but is not semifree. The results of \cite[\S4.2]{arborescent} imply that there is a canonical arborescent resolution of $S$ in this case, see \S\ref{sec:dga} for a self-contained treatment.
We make this resolution as explicit as we can in \S\ref{sec:differential}. In \S\ref{sec:cot}, we use this resolution to define explicit complexes $W_\bullet^\bc$ of finite-dimensional vector spaces whose $i$th homology is dual to the weight $\bc$ graded piece of $T^i_{S/\KK}$.

Some results, such as the localization formula in the Stanley-Reisner case (Theorem \ref{thm:localization}), follow directly from the above discussion.
For our topological interpretations of $T^i_\bc(\cK)$, more work is needed. Under appropriate assumptions on $\bc$ and homological degree, we show in \S\ref{sec:simplified} that one may replace the complex $W_\bullet^\bc$ by a simpler complex we call the non-splitting complex. We then show in \S\ref{sec:topological} that these resulting complexes are homotopic to the relevant complexes of (reduced) simplicial chains arising in the topological setting.

\subsection*{Notation and conventions}
Throughout, $\KK$ will always denote a field of characteristic zero.
All rings are commutative with identity unless otherwise stated.
Any weight vector $\bc \in \ZZ^\cV$ will always be written as $\bc = \ba - \bb$ where $\ba, \bb \in \ZZ_{\geq 0}^\cV$ are as in Notation~\ref{notation}, i.e.\ they have disjoint supports.

\subsection*{Acknowledgements}
We are grateful to Klaus Altmann, Adam Boocher, Jan Christophersen, Enrico Fatighenti and Frank Moore for helpful conversations.

All three authors were partially supported by an NSERC Alliance Catalyst grant.
NI was partially supported by an NSERC Discovery grant.
FM and AP are members of INdAM-GNSAGA.
AP acknowledges funding from the European Union - NextGenerationEU under the National Recovery and Resilience Plan (PNRR) - Mission 4 Education and research - Component 2 From research to business - Investment 1.1, Prin 2022, ``Geometry of algebraic structures: moduli, invariants, deformations'', DD N.~104, 2/2/2022, proposal code 2022BTA242 - CUP J53D23003720006.

\section{Arborescent resolutions}\label{sec:ar}
\subsection{Koszul-Tate resolutions and cotangent cohomology}\label{sec:kt}
Let $S$ be a finitely-generated $\KK$-algebra. A \emph{Koszul-Tate resolution} of $S$ is a semifree\footnote{This means it is free as a commutative graded algebra after forgetting the differential.} commutative differential graded algebra (cdga)
$P=P_\bullet$ with $P_i=0$ for $i<0$, along with a quasi-isomorphism $P \to S$. Here we use a homological grading for $P$: the differential $\partial$ of $P_\bullet$ decreases the homological degree by one.
We consider $S$ as a cdga concentrated in homological degree $0$. Such a Koszul-Tate resolution may be constructed inductively as in \cite{tate} but we note that at each step (of which there are potentially infinitely many), there are choices that must be made.

The cotangent cohomology $T^i_{S/\KK}$ may be defined in terms of $P_\bullet$.
For $i\in\ZZ$, let $\Hom_\KK^i(P,S)$ consist of those $\KK$-linear maps $\phi$ from $P_\bullet$ to $S$ satisfying $\phi(P_j)=0$ for $j\neq i$.
We let $\Der_\KK^i(P,S)$ consist of those $\phi\in\Hom_\KK^i(P,S)$ satisfying the graded Leibniz rule
\[\phi(yz)=\phi(y)z+(-1)^{ij}y\phi(z)\ \qquad\forall y\in P_j,z \in P.\]
The vector spaces $\Der_\KK^i(P,S)$ form a cochain complex $\Der_\KK^\bullet(P,S)$ with differential $\partial^*$ given by
\[
	\phi\mapsto (-1)^{\deg \phi+1}\phi\circ \partial.
\]
The $i$th cotangent cohomology of $S$ over $\KK$ is the $i$th cohomology of this complex:
\begin{equation}\label{eq.cotangentcohomology}
	T^i_{S/\KK}=H^i(\Der_\KK^\bullet(P,S)),
\end{equation}
see e.g.~\cite[Lemma~7.4 and Theorem~7.11]{MaMe20}.

\begin{example}\label{ex.onegenerator}
    Given $S=\KK[x_1,\ldots,x_n]/(f)$, one can construct a Koszul-Tate resolution as the cdga concentrated in homological degrees $0$ and $1$:
    \[ P_\bullet\colon \quad \cdots\to 0\to \KK[x_1,\ldots,x_n]\xrightarrow{f\cdot } \KK[x_1,\ldots,x_n]\to 0\to \cdots \]
    and the morphism $P_\bullet\to S$ is non-trivial
    only in degree zero, where it is given by the projection $\KK[x_1,\ldots,x_n]\to S$. We obtain
    \[ \Der_\KK^i(P,S)=\begin{cases}
        0 & \text{ for } i\neq 0,1\\
        S & \text{ for } i= 1\\
        \Der_\KK(\KK[x_1,\ldots,x_n],S) & \text{ for } i= 0
    \end{cases} \]
    and the differential $\partial^* \colon \Der_\KK^0(P,S)\to S=\Der_\KK^1(P,S)$ is given by
    \[  \phi \mapsto \sum_{j=1}^n\frac{\partial f}{\partial x_j}\phi(x_j). \]
    Hence by~\eqref{eq.cotangentcohomology} we recover that $T^1_{S/\KK}\cong\frac{\KK[x_1,\ldots,x_n]}{(f,\partial_{x_1}f,\ldots,\partial_{x_n}f)}$ and $T^i_{S/\KK}=0$ for every $i\geq2$.
\end{example}

\begin{remark}\label{rmk.doublegrading}
	By definition, $P$ carries a homological grading by $\ZZ$. 
In the setting of the remainder of this paper, $S$ will typically also carry a grading by $\ZZ^n$ for some $n\geq 1$.
We can and will choose $P$ so that each $P_i$ is $\ZZ^n$-graded, the differential $\partial$ is homogeneous of degree $0$ with respect to this grading, and so is the quasi-isomorphism $P\to S$.
With $P$ chosen in this fashion, each $\Der_\KK^i(P,S)$ decomposes as a direct sum of homogeneous pieces with respect to the $\ZZ^n$-grading. This induces
 a similar decomposition of 
$T^i_{S/\KK}$.

To avoid ambiguity, we will always refer to degrees in the (co)homological $\ZZ$-grading of $P_\bullet$, $\Der^\bullet(P,S)$, or $T^\bullet_{S/\KK}$ as \emph{(co)homological} degrees. 
The degrees associated to the $\ZZ^n$-grading inherited from $S$ will be called \emph{weights}.
\end{remark}

\subsection{Inductive construction of arborescent resolutions}\label{sec:inductive}
Here we give an alternate and self-contained description of arborescent resolutions \cite{arborescent}.\footnote{The terminology ``arborescent'' comes from \cite{arborescent}, in which rooted trees play an important role in constructing the resolution. We avoid the use of such trees here, but keep the name.} Let $R$ be a polynomial ring over $\KK$, and $S$ a quotient of $R$.
Let $F_\bullet$ be a free resolution of $S$ as an $R$-module with differential $d$.\footnote{We continue using homological conventions, so 
$d$ decreases degree by $1$.} We may assume that $F_0=R$; moreover, by Hilbert's syzygy theorem, we may assume that $F_\bullet$ has only finitely many non-zero terms. Our goal is to use $F_\bullet$ to construct a Koszul-Tate resolution $P_\bullet\to S$.

We give an inductive description of $P_\bullet=\bigoplus_{i\geq 0} P_i$, first as a commutative graded algebra (without differential). 
Let $P(0)=R$, viewed as a commutative graded algebra concentrated in degree $0$. Set $P_0^\old=0$ and $P_0^\new=0$.
We now proceed inductively on $i \geq 0$, assuming we have defined $P(i)$ and $P_j^\old$ and $P_j^\new$ for all $0 \leq j\leq i$.
We set
\[
	P_{i+1}^\new:=F_{i+1}\oplus P^{\old}_i
\]
and we define
$P(i+1)$ to be the (graded) symmetric algebra over $R$ of the graded free $R$-module
\[
\bigoplus_{1 \leq j \leq i+1} P_j^\new,
\]
 where the elements of $P_j^\new$ have homological degree $j$.
Set $P_{i+1}^{\old}=P(i)_{i+1}$.

We have inclusions $P(i)\subseteq P(i+1)$, and \[P(i)_j=P(i+1)_j=P_j^\new\oplus P_j^\old\] for any $j\leq i$.
Finally, we set 
\[
	P_\bullet=\bigcup_{i\geq 0} P(i)=\bigoplus_{i\geq 0} \left(P_i^\new \oplus P_i^\old\right),
\]
with algebra structure induced by those of the $P(i)$.\footnote{In other words, $P_\bullet$ is the colimit of the $P(i)$ in the category of commutative graded algebras.}

For instance, the homogeneous summands of low degree of $P_\bullet$ are:
\[
	\begin{aligned}
    P_0 &= R, \\
    P_1 &= F_1, \\
    P_2 &= \underbrace{F_2}_{P_2^\new} \oplus \underbrace{\wedge^2_R F_1}_{P_2^\old}, \\[.2cm]
    P_3 &= F_3 \oplus \underbrace{\wedge^2_R F_1}_{P_2^\old} \oplus
    \underbrace{\wedge^3_R F_1 \oplus (F_1 \otimes_R F_2)}_{P^\old_3}, \\[.2cm]
    P_4 &= F_4 \oplus \underbrace{\wedge^3_R F_1 \oplus  (F_1 \otimes_R F_2)}_{P_3^\old} 
    \oplus
    \underbrace{
    (F_1 \otimes_R \wedge^2_R F_1) \oplus
    (\wedge^2_R F_1 \otimes_R F_2) \oplus
    \wedge^4_R F_1 \oplus \mathrm{S}^2_R F_2}_{P_4^\old}.
\end{aligned}
\]

We now wish to endow $P_\bullet$ with a differential $\partial$, which we will also inductively define. With respect to the direct sum decomposition
\[
	P_j=F_j\oplus P_{j-1}^\old \oplus P_j^\old,
\]
the differential $\partial_j$ we construct will have the form
\begin{equation}\label{eqn:diffform}
\begin{array}{l}
	\hspace{1.1cm}\begin{array}{c c c}
	F_j& P_{j-1}^\old & P_j^\old
	\vspace{.3cm}
\end{array}\\
	\partial_j=
\left(\begin{array}{c c c}
				d_j & -\rho_{j} & \alpha_j\\ \\
		0 & -\gamma_{j-1} & \beta _j\\ \\
		0 & \id & \gamma_j
	\end{array}\right)
	\begin{array}{c}
		F_{j-1}\\ \\ P_{j-2}^\old \\ \\ P_{j-1}^\old
	\end{array}
\end{array}
\end{equation}
where $d$ is the differential of $F_\bullet$ and $\id$ is the identity map. 
The differential must satisfy the graded Leibniz rule:
\begin{equation}\label{eqn:gl}
	\partial_j(yz)=\partial_k(y)z+(-1)^ky\partial_{k-j}(z)\qquad \forall y\in P_k,\ z\in P_{k-j}.
\end{equation}
Likewise, it must square to zero:
\begin{equation}\label{eqn:closed}
\partial_{j-1} \partial_j=0.
\end{equation}

Since $P_0^\old=P_1^\old=0$, we begin with $\partial_0=0$, $\partial_1=d_1$.
\begin{theorem}[cf.~{\cite[Theorem 2.28]{arborescent}}]\label{thm:arborescent}
Let $i\geq 1$.	Suppose that for $j\leq i$, we have maps $\partial_j$ of the form \eqref{eqn:diffform} satisfying \eqref{eqn:gl} and \eqref{eqn:closed}.
	\begin{enumerate}	
		\item \label{item:rho}
			There exists $\rho_{i+1}$ satisfying
\[
	d_i\rho_{i+1}=\rho_{i}\gamma_i+\alpha_i.
\]
\item After choosing such $\rho_{i+1}$, there is a unique map $\partial_{i+1}$ of the form \eqref{eqn:diffform} satisfying \eqref{eqn:gl} and \eqref{eqn:closed} for $j=i+1$.
\item The differential $\partial$ for $P_\bullet$ obtained via induction yields a Koszul-Tate resolution of $S$, that is,
	\[
		H_k(P_\bullet)\cong\begin{cases} 
			S & k=0,\\
			0 & \textrm{else}.
		\end{cases}
	\]
\end{enumerate}
We call any resolution $(P_\bullet,\partial)$ as above an \emph{arborescent resolution} of $S$.
\end{theorem}

\begin{proof}
	Multiplying out, we see that \eqref{eqn:closed} is equivalent to
	\begin{align}
		d_{j-1}d_j&=0\label{eqn:n1}\\
		d_{j-1}\rho_j&=\rho_{j-1}\gamma_{j-1}+\alpha_{j-1}\label{eqn:n2}\\
		d_{j-1}\alpha_j+\alpha_{j-1}\gamma_j&=\rho_{j-1}\beta_{j}\label{eqn:n3}\\
		\beta_{j-1}&=-\gamma_{j-2}\gamma_{j-1}\label{eqn:n4}\\
		\beta_j&=-\gamma_{j-1}\gamma_j\label{eqn:n5}\\
		\gamma_{j-2}\beta_j&=\beta_{j-1}\gamma_j\label{eqn:n6}.
	\end{align}
	Equation \eqref{eqn:n1} is always fulfilled since $F_\bullet$ is a complex. We see that \eqref{eqn:n6} follows from \eqref{eqn:n4} and \eqref{eqn:n5}, so we may disregard it. Using \eqref{eqn:n2} and \eqref{eqn:n5}, we also see that \eqref{eqn:n3} is equivalent to
\[
	d_{j-1}(\rho_j\gamma_j+\alpha_j)=0.
\]

We now prove the first claim.
If $i=1$, then $P_i^\old=0$, so there is nothing to prove. Assume $i>1$.
By the above discussion of \eqref{eqn:closed} for $j=i$, we see that $\rho_i\gamma_i+\alpha_i$ factors through the kernel of $d_{i-1}$, and by the exactness of $F_\bullet$ at everywhere except the $0$th position, we see that $\rho_i\gamma_i+\alpha_i$ factors through the image of $d_i$. Since $P_i^\old$ is a free $R$-module, there exists $\rho_{i+1} \colon P_i^\old\to F_i$ such that $d_i\rho_{i+1}=\rho_{i}\gamma_i+\alpha_i.$

For the second claim, the graded Leibniz rule \eqref{eqn:gl} for $j=i+1$ uniquely determines what $\alpha_{j+1}$, $\beta_{j+1}$, and $\gamma_{j+1}$ must be. The fact that $\alpha_{j+1},\beta_{j+1},\gamma_{j+1}$ are indeed well-defined follows inductively from  \eqref{eqn:gl} for $j\leq i$.
To satisfy \eqref{eqn:closed} for $j=i+1$, we may restrict to the summands $F_{i+1}$, $P_i^\old$, and $P_{i+1}^\old$ of $P_{i+1}$.
For $F_{i+1}$, we just have \eqref{eqn:n1}, which as noted above, is always satisfied. For $P_i^\old$, we have \eqref{eqn:n2} and \eqref{eqn:n4}. The first follows by choice of $\rho_{i+1}$, the second follows from \eqref{eqn:n5} for $j=i$. Finally, it is straightforward to see directly that the graded Leibniz rule applied twice implies \eqref{eqn:closed} for $j=i+1$ after restricting to $P_{i+1}^\old$. 

For the third claim, we first notice that $H_0(P_\bullet)=\coker \partial_1=\coker d_1 = S$. We now show that the homology of $P_\bullet$ vanishes at all other positions. For $i<0$ this is clear. For $i>0$, after applying row operations to \eqref{eqn:diffform} using the third row, and applying \eqref{eqn:n2} and \eqref{eqn:n4} for $j=i+1$, we obtain that 
\[
	\ker \partial_i=\ker\left(\begin{array}{c c c}
		d_i & 0 & d_i\rho_{i+1}\\
		0 & \id & \gamma_i
	\end{array}\right).
\]
In other words, $(x,y,z)\in F_i\oplus P_{i-1}^\old\oplus P_i^\old$ is in the kernel if and only if $y=-\gamma_i(z)$ and $d_i(x+\rho_{i+1}(z))=0$. But by the form of $\partial_{i+1}$ and the exactness of $F_\bullet$, this is spanned by the images of $F_{i+1}$ and $P_i^\old$ under $\partial_{i+1}$.
\end{proof}
\begin{remark}
	Assuming that the free resolution $F_\bullet$ has finite length, there are only finitely many $\rho_i$ that need to be chosen as in Theorem \ref{thm:arborescent} to determine the differential $\partial$ of $P_\bullet$. The information in these maps $\rho_i$ is related to the so-called ``arborescent operations'' of \cite[\S 2.2]{arborescent}.
\end{remark}

\begin{remark}\label{rem:piF}
Consider the map $\pi_F \colon P_\bullet\to F_\bullet$
whose $i$th component with domain $P_i=F_i\oplus P_{i-1}^\old \oplus P_i^\old$ restricts to the identity on $F_i$, $0$ on $P_{i-1}^\old$, and $\rho_{i+1}$ on $P_i^\old$. Then $\pi_F$ is a map of complexes, that is, $d\circ \pi_F=\pi_F\circ\partial$. Indeed, the condition that $\pi_F$ commutes with the differential amounts to the condition on $\rho_{i+1}$ from Theorem \ref{thm:arborescent}\eqref{item:rho}.
\end{remark}

\subsection{Arborescent resolutions for cdgas}\label{sec:dga}
Suppose that the free resolution $F_\bullet$ in fact has the additional structure of a cdga. In this case, there is a canonical choice for the differential $\partial$ of $P_\bullet$, and the map $\pi_F$ of Remark \ref{rem:piF} is a map of cdgas. 

Indeed, we will define maps $\rho_{i+1} \colon P_{i}^\old\to F_{i}$ as follows. For any $y\in P_j^\new$, let $\mathrm{pr}(y)$ denote its image under the projection $\mathrm{pr} \colon F_j\oplus P_{j-1}^\old\to F_{j}$  to the first factor.
We then set
\[
	\rho_{i+1}(y_1\cdots y_k)=\mathrm{pr}(y_1)\cdots \mathrm{pr}(y_k)
\]
for any $y_j\in P_{\delta_j}^\new$, $j=1,\ldots,k$ satisfying $\delta_1+\ldots+\delta_k=i$. Here, the product on the left hand side is in $P_\bullet$, and the product on the right hand side is in $F_\bullet$.
This extends linearly to give an $R$-linear map $\rho_{i+1} \colon P_{i}^\old\to F_{i}$.

\begin{proposition}[cf.~{\cite[\S 4.2]{arborescent}}]\label{prop:dga}
The maps $\rho_{i+1}$ defined above satisfy the condition of Theorem \ref{thm:arborescent}\eqref{item:rho}, and thus determine a differential $\partial$ of $P_\bullet$.
Moreover, the map $\pi_F$ of Remark \ref{rem:piF} is multiplicative, hence a map of cdgas.
\end{proposition}
\begin{proof}
	It follows immediately from the definitions of $\rho_{i+1}$ and $\pi_F$ that $\pi_F$ is multiplicative. For the first claim of the proposition, we need to show that the $\rho_{i+1}$ satisfy the condition of Theorem \ref{thm:arborescent}\eqref{item:rho}. But as claimed in Remark \ref{rem:piF}, this is the same as the equality $d_{i-1} \circ (\pi_F)_i=(\pi_F)_{i-1}\circ \partial_i$. We show this equality by induction on $i$. For elements $y\in P_i$ belonging to $F_i$ or $P_{i-1}^\old$, this equality is clear. For $y\in P_i^\old$, it suffices to consider products $y=y_1y_2$ with $y_1,y_2$ of lower homological degree.
	Then
	\begin{align*}
		(d \circ\pi_F)(y_1y_2)&=d(\pi_F(y_1)\pi_F(y_2))\\
		&=d(\pi_F(y_1))d(\pi_F(y_2))\\
		&=\pi_F(d(y_1))\pi_F(d(y_2))\\
		&=(\pi_F\circ d)(y_1y_2)
	\end{align*}
with the first equality following from the multiplicativity of $\pi_F$, the second from the multiplicativity of $d$, the third by the induction hypothesis, and the fourth again by multiplicativity of $\pi_F$ and $d$.
\end{proof}

\subsection{Monomial ideals and the Taylor resolution}\label{sec:taylor}
In the remainder of this paper, we will always be considering the following situation: $R=\KK[x_1,\ldots,x_n]$ will be a polynomial ring over $\KK$, $I\subsetneq R$ will be a proper ideal generated by monomials, and $S=R/I$ will be the quotient.
The ring $R$ has a natural $\ZZ^n$-grading. Since $I$ is a monomial ideal, $R/I$ also has a $\ZZ^n$-grading. 
In $R$ we will use multi-index notation, e.g. for any $\bc = (\bc_1,\dots,\bc_n) \in\ZZ^n$, we set
\[
	x^\bc=\prod_{i=1}^n x_i^{\bc_i}.
\]

The \emph{Taylor resolution} $F_\bullet$ is a $\ZZ^n$-graded free resolution of $S$, as we will now recall (cf.~\cite[\S2]{bps98}).
For input data, we fix $\cG$ to be any set and $\lambda \colon \cG\to I$ to be a map whose image consists of monomials generating $I$ as an ideal. Typically, we will take $\cG$ to the minimal monomial generators of $I$, and $\lambda$ simply to be the inclusion.

We now take $F_0=S$, and let $F_1$ to be the free $\ZZ^n$-graded $R$-module with basis $e_f$ indexed by the elements $f \in \cG$. Here, the weight $\wt(e_f)$ is simply the weight $\wt(f) \in \ZZ^n$ of $\lambda(f)$, and the differential $d_1$ sends $e_f$ to $\lambda(f)$.

For $i>1$, we set $F_i=\bigwedge^i_R F_1$, but with a modified $\ZZ^n$-grading.
Denote
\[
	e_{f_1}\wedge e_{f_2}\wedge \cdots \wedge e_{f_i}\in F_i
\]
by 
\[
	\{f_1 f_2\ldots f_i\}
\]
for any $f_1,\ldots,f_i\in \cG$.
Then we set
\[
	\wt(\{f_1\ldots f_i\})=\wt(\lcm(\lambda(f_1),\ldots,\lambda(f_m))).
\]

The differential $d_i \colon F_i\to F_{i-1}$ is defined via
\begin{equation}\label{eq.differentialfree}
	d_i(\{f_1\ldots f_i\})=\sum_{j=1}^i (-1)^{j+1} x^{\wt (\{f_1\ldots f_i\})-\wt(\{f_1\ldots \widehat{f_j}\ldots f_i\})}
		\{f_1\ldots \widehat{f_j}\ldots f_i\}.
\end{equation}
Here, $\widehat{f_j}$ indicates we are removing the $f_j$ term.
The differential $d$ is similar to the differential for the exterior algebra on $F_1$, but its terms are multiplied by monomials to make $d$ homogeneous with respect to the $\ZZ^n$-grading.

As described above, $(F_\bullet,d)$ is a $\ZZ^n$-graded free resolution of $S$ as an $R$-module. Moreover, it has a canonical cdga  structure \cite[Proof of Corollary 3.6]{bps98}:
the product of $\{f_1\ldots f_i\}\in F_i$ and $\{f_1'\ldots f_j'\}\in F_j$ is defined to be
\[
	x^{\wt(\{f_1\ldots f_i\})+\wt(\{f_1'\ldots f_j'\})-\wt(\{f_1\ldots f_i  f_1'\ldots f_j'\})}
	\{f_1\ldots f_i f_1'\ldots f_j'\}\in F_{i+j}.
\]
Again, this is essentially the product on the exterior algebra on $F_1$, but modified so as to preserve the $\ZZ^n$-grading.

We now find ourselves in the setting of \S\ref{sec:dga}: the Taylor resolution $F_\bullet$ of $S$ is itself a cdga, so from it we obtain a \emph{canonical} Koszul-Tate resolution $P_\bullet \to S$ (Proposition \ref{prop:dga}). See also \cite[\S4.2.3]{arborescent}.
We will call $P_\bullet \to S$ a \emph{Koszul-Tate-Taylor resolution} of $S$.

\subsection{Explicit differential in monomial case}\label{sec:differential}
We now wish to explicitly describe $R$-module generators and the differential for a Koszul-Tate-Taylor resolution $P_\bullet$ of $S$.
As in \S\ref{sec:taylor}, we have $R$-module generators for $F_i$ of the form $\{f_1\ldots f_i\}$. Given $p\in P_i$ and $q\in P_j$, we denote their product in $P_{i+j}^\old\subseteq P_{i+j}$ simply by $pq$. On the other hand, given $y\in P_i^\old$, we may view it as an element of $P_{i+1}^\new\subseteq P_{i+1}$; we denote this element by $[y]$.

Hence, we obtain $R$-module generators of $P_\bullet$ from strings  involving elements $f\in \cG$, and matched pairs of $\{ \ \}$ and $[ \ ]$.
Note that the (graded) commutativity of $P_\bullet$ induces relations among these generators, e.g. $[\{f_1\}\{f_2\}]=-[\{f_2\}\{f_1\}]$.

\begin{example}[Generators of low homological degree]
We record all possible forms for generators of $P_i^\new $ for $i\leq 5$, up to sign.
\[
\begin{array}{l l l }
	& \textrm{hom.\ deg.} & \textrm{weight} \\
\hline
\{f\}& \qquad 1& \wt(f)  \\
	\{f_1f_2\} & \qquad  2 & \wt(\lcm(f_1,f_2))\\
	\{f_1f_2f_3\} & \qquad  3 & \wt(\lcm(f_1,f_2,f_3)) \\
    {[} \{f_1\} \{f_2\} {]} & \qquad  3 & \wt(f_1)+\wt(f_2)\\
	\{f_1f_2f_3f_4\} & \qquad  4 & \wt(\lcm(f_1,f_2,f_3,f_4)) \\
	{[}\{f_1f_2\}\{f_3\}{]} & \qquad  4 & \wt(\lcm(f_1,f_2))+\wt(f_3)\\
	{[}\{f_1\}\{f_2\}\{f_3\}{]} & \qquad  4 & \wt(f_1)+\wt(f_2)+\wt(f_3)\\
    \{f_1f_2f_3f_4f_5\} & \qquad  5 & \wt(\lcm(f_1,f_2,f_3,f_4, f_5)) \\
    {[} \{f_1 f_2 f_3\} \{f_4\} {]} & \qquad  5 & \wt(\lcm(f_1,f_2,f_3))+\wt(f_4)\\
    {[} \{f_1 f_2\} \{f_3 f_4\} {]} & \qquad  5 & \wt(\lcm(f_1,f_2))+\wt(\lcm(f_3f_4))\\
    {[} \{f_1 f_2\} \{f_3 \} \{ f_4\} {]} & \qquad  5 & \wt(\lcm(f_1f_2))+\wt(f_3)+\wt(f_4)\\
    {[} \{f_1 \} \{ f_2 \} \{f_3 \} \{ f_4\} {]} & \qquad  5 & \wt(f_1)+\wt(f_2)+\wt(f_3)+\wt(f_4) \\
    {[}[\{f_1\}\{f_2\}]\{f_3\}{]} & \qquad 5 & \wt(f_1)+\wt(f_2)+\wt(f_3)
\end{array}
\]
\end{example}
\begin{remark}
	We will occasionally need to distinguish between a \emph{string} $s$ involving elements $f\in \cG$ and matched pairs of $\{ \ \}$ and $[ \ ]$, and the \emph{generator} $\overline s\in P_\bullet$ that it represents. We will thus always use the term ``string'' to refer to an actual string, and ``generator'' to refer to an element of $P_\bullet$ represented by a string. Given a string $s$, we denote the corresponding generator by $\overline s$.
    When actually writing down a string we will include quotation marks, so e.g.\ $``[\{f_1\}\{f_2\}]"$ denotes a string, whereas $[\{f_1\}\{f_2\}]$ denotes the generator it represents.
Following this notation, we for instance have (for $f_1,f_2,f_3\in \cG$ distinct)
\[
	\{f_1f_2f_3\}=\{f_2f_3f_1\}\qquad \mathrm{but}\ ``\{f_1f_2f_3\}"\neq ``\{f_2f_3f_1\}".
\]
\end{remark}

\begin{definition}
	Strings of the form $``\{f_1\cdots f_i\}"$ we will call \emph{primitive},\footnote{Note that the string $``\{\}"$ representing $1\in P_0=R$ is primitive.} and those strings beginning with $[$ and ending with $]$ but containing no other $[$ we will call \emph{quasiprimitive}. We will call a generator primitive (respectively quasiprimitive) if it can be represented by a primitive (respectively quasiprimitive) string.
\end{definition}

Notice that all generators of homological degree at most $4$ are primitive or quasiprimitive; in homological degree $5$ this is no longer the case, e.g.\ $[[\{f_1\}\{f_2\}]\{f_3\}]$.

In order to describe the differential $\partial$, we need a little more notation.

\begin{definition}
For any string $s$ as above, we define the following sequences:
\begin{itemize}
	\item $\twig(s)$ is the sequence of all quasiprimitive strings contained in $s$;
	\item $\petal(s)$ is the sequence in $s$ of all $f\in\cG$ and all left brackets $[$;
\end{itemize}
For $q$ in $\twig(s)$ or $\petal(s)$, we define its \emph{sign} in $s$ to be $\sign(q;s) :=
		(-1)^m$
where $m$ is the number of preceding elements in  $\petal(s)$.
\end{definition}

\begin{remark}\label{rmk.homdegvslengthpetals}
    Notice that the homological degree of $\overline s$ is the length of $\petal(s)$.
\end{remark}

\begin{example}\label{example.signs}
If $s=``[[\{f_1\}\{f_2\}]\{f_3\}]"$ then $\twig(s)$ consists of $``[\{f_1\}\{f_2\}]"$ with sign $-1$ and $\petal(s)$ consists of $[\,,[\,,f_1,f_2,f_3$ with signs $1$, $-1$, $1$, $-1$, $1$.
    
		If $s=``[[\{f_1f_2\}\{f_3\}][\{f_4\}\{f_5\}]]"$,
        then
        $\twig(s)$ consists of $``[\{f_1f_2\}\{f_3\}]", ``[\{f_4\}\{f_5\}]"$ with signs $-1$, $-1$, while $\petal(s)$ consists of $[\,,[\,,f_1,f_2,f_3,[\,,f_4,f_5$ with signs $1$, $-1$, $1$, $-1$, $1$, $-1$, $1$, $-1$.
\end{example}

\begin{definition}
    For a string $s$, and an element $q$ in $\petal(s)$, we let 
$s\setminus q$ be the string obtained from $s$ by removing $q$ (and the matching bracket $]$ if $q$ is a bracket).

Likewise, for $q\in \twig(s)$, we let $s \git q$ be the string obtained from $s$ by replacing $q$ with the unique primitive string $p$ with $\petal(p)$ equal to $\petal(q)$ after omitting the first left bracket in $\petal(q)$.

For any string $p \setminus q$ or $p \git q$ we obtain this way that makes no sense (for instance, we apply $[ \ ]$ to an element of $P_i^\new$), we treat the corresponding generator as zero.

\end{definition}

\begin{theorem}\label{thm.stringsdifferential}
    Let $s$ be any string as above. Then
	\begin{align*}
		\partial(\overline s)=
        \sum_{q\in \petal(s)}\sign(q;s)x^{\wt (\overline s)-\wt (\overline{s \setminus q})}(\overline{s\setminus q})
        -
        \sum_{q\in \twig(s)}\sign(q;s)x^{\wt(\overline s)-\wt(\overline{s \git q})}(\overline{s\git q}).
	\end{align*}
\end{theorem}
\begin{proof}
	We consider three cases separately: $s$ a primitive string, $s$ a non-primitive string starting with a $[$, and finally $s$ a non-primitive string not starting with $[$. Throughout, we will use $x^\star$ as an abbreviation for  the correct monomial to keep everything $\ZZ^n$-homogeneous.
For any string $s'$ beginning and ending with $[$ and $]$ respectively, we let $\langle s' \rangle$ be the string obtained by omitting the first $[$ and last $]$.

		Firstly, for a primitive string $s$, by looking at the first column in the matrix in \eqref{eqn:diffform} the differential reduces to the Taylor differential, which agrees with the stated formula, see \eqref{eq.differentialfree}.

		Secondly, consider a non-primitive string of the form $s=``[s_1 \ldots s_k]"$ representing a generator $\overline{s}$ in $P_{m-1}^\old\subseteq P_m^{\new} \subseteq P_m$ where the $s_1,\ldots,s_k$ are themselves strings. By \eqref{eqn:diffform} the map to $F_{m-1}$ is given by $-\rho_{m}$, which means this part of the image of $\overline s$ is exactly
        $-x^\star(\overline{s\git s})$ if $s$ is quasiprimitive and $0$ otherwise.
		The map to $P_{m-1}^\old$ gives us exactly $\overline{s_1\ldots s_k}$.

		To understand the map $-\gamma_{m-1}\colon P^{\old}_{m-1}\to P^{\old}_{m-2}$, we need to take the negative of the image of $\overline{s_1\ldots s_k}$ under the differential, and replace each generator $y$ occurring there by $[y]$. Applying the Leibniz rule we have
    \[
	    \partial (\overline{s_1\ldots s_k})=\sum_{i=1}^k \epsilon_i \overline{s_1}\ldots \partial (\overline{s_i})\ldots \overline{s_k}
    \]
    where $\epsilon_i=(-1)^{\sum\limits_{j<i} \deg \overline{s_j}}$ and $\deg$ denotes the homological degree.
    By induction this is
 \[
	    \sum\limits_{i=1}^k \epsilon_i\cdot \overline{s_1}\ldots 
	    \left(
	    \sum_{q\in \petal(s_i)}\sign(q;s_i)x^\star(\overline{s_i\setminus q})
	    -\sum_{q\in \twig(s_i)}\sign(q;s_i)x^\star(\overline{s_i\git q})
	    \right)
	    \ldots \overline{s_k}.
    \]
    Recall that $\deg(\overline{s_i})=\vert\petal(s_i)\vert$ by Remark~\ref{rmk.homdegvslengthpetals}; therefore $\sign(q;s)=-\epsilon_i\sign(q;s_i)$. Hence rearranging the terms we obtain
  \[ 
	    -\sum_{q\in {\petal(\langle s\rangle )}}\sign(q;s)x^\star (\overline{\langle s\setminus q\rangle})
	    +\sum_{q\in \twig(s),q\neq s} \sign(q;s)x^\star(\overline{\langle s\git q\rangle})
    \]
    Adding back $[\; ]$ and taking the opposite sign, we end up with the desired formula for 
$s=[s_1 \ldots s_k]$.
   
    Finally, consider a string  of the form $s=s_1\ldots s_k$ with $\overline s\in P_m^\old\subseteq P_m$. Similar to above we have
 \[ \begin{aligned}
		    \partial(\overline{s_1\ldots s_k}) &=\sum_{i=1}^k \epsilon_i \overline{s_1}\ldots \partial(\overline{s_i})\ldots \overline{s_k}\\
		    &=
		    \sum_{q\in \petal(s)}\sign(q;s)x^\star (\overline{s\setminus q})
		    -\sum_{q\in \twig(s)} \sign(q;s)x^\star(\overline{s \git q})
    \end{aligned} \]
    where the only difference with respect to the case above is the relation
    \[\sign(q;s)=\epsilon_i \sign(q;s_i).\qedhere\]
\end{proof}

\begin{example}
We compute the differentials of the generators represented by the strings of Example~\ref{example.signs}.
    For $y=[[\{f_1\}\{f_2\}]\{f_3\}]$, we obtain
    \begin{align*}
    	\partial(y)={x^{\wt(\gcd(f_1,f_2))}}[\{f_1f_2\}\{f_3\}]+[\{f_1\}\{f_2\}]\{f_3\}-[\{f_1\}\{f_2\}\{f_3\}].
    \end{align*}
    Indeed, the remaining three terms (e.g.~$[[\{f_2\}]\{f_3\}]$) are all treated as zero.
    
    For $y=[[\{f_1f_2\}\{f_3\}][\{f_4\}\{f_5\}]]$, we obtain
    \begin{align*}
    	\partial(y) &={x^{\wt(\gcd(\lcm(f_1,f_2),f_3))}}[\{f_1f_2f_3\}[\{f_4\}\{f_5\}]] 
        + {x^{\wt(\gcd(f_4,f_5))}} [[\{f_1f_2\}\{f_3\}]\{f_4f_5\}]\\
        &+ [\{f_1f_2\}\{f_3\}][\{f_4\}\{f_5\}]
        - [\{f_1f_2\}\{f_3\}[\{f_4\}\{f_5\}]]\\
	&+{x^{\wt(f_1)-\wt(\gcd(f_1,f_2))}} [[\{f_2\}\{f_3\}][\{f_4\}\{f_5\}]] \\
        &-{x^{\wt(f_2)-\wt(\gcd(f_1,f_2))}} [[\{f_1\}\{f_3\}][\{f_4\}\{f_5\}]]
        -[[\{f_1f_2\}\{f_3\}]\{f_4\}\{f_5\}]
\end{align*}
    where the remaining three terms have been treated as zero.
\end{example}

\subsection{Computing graded pieces of cotangent cohomology}\label{sec:cot}
Continuing with $S=\KK[x_1,\ldots,x_n]/I$ for a monomial ideal $I$, let $P=P_{\bullet}\to S$ be a Koszul-Tate resolution respecting the $\ZZ^n$-grading of $S$, for instance, a Koszul-Tate-Taylor resolution.
As noted in Remark \ref{rmk.doublegrading}, the derivation complex $\Der_{\KK}^\bullet(P,S)$ is also $\ZZ^n$-graded, and so the $T^i_{S/\KK}$ are as well.
For any weight $\bc\in \ZZ^n$, we will describe an explicit complex of finite dimensional $\KK$-vector spaces whose homology is dual to the weight $\bc$ piece of  $T^i_{S/\KK}$.

To that end, for each $i\geq 0$, fix a $\ZZ^n$-homogeneous basis for $P_i^\new$ as an $R$-module. For any $\bu\in \ZZ_{\geq 0}^n$, let $W_{i,\bu}\subseteq P_i^\new$ be the $\KK$-vector space generated by homogeneous basis elements of weight $\bu$. Each $W_{i,\bu}$ is finite dimensional. Notice that 
\[
	P_i^\new=\bigoplus_{\bu\in \ZZ_{\geq 0}^n} W_{i,\bu}\otimes_\KK R. 
\]
Typically, we will take the homogeneous basis of $P_i^\new$ to consist of generators as described in \S\ref{sec:differential}.
The differential $\partial$ of $P_\bullet$ induces a $\KK$-linear map 
\[ \overline\partial_i \colon \bigoplus_{\bu\in\ZZ_{\geq 0}^n} W_{i,\bu} \to \bigoplus_{ \bu\in \ZZ_{\geq 0}^n} W_{i-1,\bu}
\] defined via the composition
\[ \bigoplus_{\bu\in\ZZ_{\geq 0}} W_{i,\bu} \xrightarrow{\iota} P_i\xrightarrow{\partial_i} P_{i-1} \xrightarrow{\pi_R} P_{i-1}^\new \xrightarrow{\pi_{\KK}} \bigoplus_{\bu\in\ZZ_{\geq 0}^n} W_{i-1,\bu} \]
where $\iota$ denotes the inclusion, $\pi_R$ the natural projection of $R$-modules coming from the decomposition $P_{i-1}=P_{i-1}^\new\oplus P_{i-1}^\old$, and $\pi_\KK$ is obtained by tensoring with the map $R=\KK[x_1,\ldots,x_n]\to \KK$ defined by $x_j\mapsto1$ for every $j\in\{1,\ldots,n\}$.

\begin{definition}
	Following Notation~\ref{notation}, for every weight $\bc=\ba-\bb\in \ZZ^n$, we define
    \[ \nabla_\bc=\{\bu\in\ZZ_{\geq 0}^n\ |\ \bu+\bc\in\ZZ^n_{\geq 0}\ \textrm{and}\ x^{\bu+\bc}\notin I\}. \]
\noindent For $i\geq 1$, we then define
\[
	W_i^\bc=\bigoplus_{\bu\in\nabla_\bc} W_{i,\bu}.
\]
For the case $i=0$, we set
\[
W_0^\bc=
	\begin{cases}
		\KK & \text{if } |\bb|= 1 \ \textrm{and}\ x^\ba\notin I\\
		0 & \text{if } |\bb|\neq 1\ \textrm{or}\ x^\ba\in I.
	\end{cases}
\]
For $i\geq 0$, we let
\begin{align*}
	\partial^\bc_{i+1} \colon W_{i+1}^\bc&\to W_{i}^\bc
\end{align*}
be the $\KK$-linear map obtained by applying $\overline\partial$ and then projecting onto the summands indexed by $\bu\in \nabla_\bc$ (or simply onto $\KK$ or $0$ in the case $i=0$).
\end{definition}

\begin{theorem}\label{thm.computingTviahomology}
	Let $\bc\in\ZZ^n$. 
	\begin{enumerate}
		\item For $i\geq 0$, the composition $\partial^\bc_i \circ \partial^\bc_{i+1}$ is equal to zero, so $W_\bullet^\bc$ is a complex.
		\item For $i\geq 1$, there is an isomorphism			\[
				\Hom_\KK(H_i(W_\bullet^\bc),\KK)\to (T^i_{S/\KK})_\bc.
			\]
	\end{enumerate}
		\end{theorem}
\begin{proof} For every $i\geq 1$ we have
        \[ \begin{aligned}
        \Der^i_\KK(P,S)_\bc&=\Hom_R(P_i^\new,S)_\bc=\bigoplus_{\bu \in \ZZ^n} \Hom_R(W_{i,\bu}\otimes_\KK R,S)_\bc\\
            &=\bigoplus_{\bu \in \ZZ^n} \Hom_\KK(W_{i,\bu},S)_{\bc}
            = \bigoplus_{\bu \in \ZZ^n} \Hom_\KK(W_{i,\bu}, S_{\bu +\bc})\\
            &=\bigoplus_{\bu \in \nabla_{\bc}} \Hom_\KK(W_{i,\bu},S_{\bu +\bc})
            \cong\Hom_\KK\left(\bigoplus_{\bu \in \nabla_{\bc}} W_{i,\bu},\KK\right)=\Hom_\KK(W_i^\bc,\KK).
        \end{aligned} \]
		Moreover, under the above identifications, the differential
        \[ \begin{aligned}
            \Der^i(P,S)_\bc &\to \Der^{i+1}(P,S)_\bc\\
            \phi &\mapsto (-1)^{i+1}\phi \circ \partial_{i+1}
        \end{aligned} \]
        of the weight $\bc$ part of the derivation complex first becomes
        \[ \begin{aligned}
            \bigoplus_{\bu \in \ZZ^n} \Hom_\KK(W_{i,\bu},S)_{\bc} &\to \bigoplus_{\bu \in \ZZ^n} \Hom_\KK(W_{{i+1},\bu},S)_{\bc}\\
            \phi &\mapsto (-1)^{i+1}\phi \circ \bar\partial_{i+1}
        \end{aligned} \]
        and finally
        \[ \begin{aligned}
            \Hom_\KK(W_i^\bc,\KK) &\to \Hom_\KK(W_{i+1}^\bc,\KK)\\
            \phi &\mapsto (-1)^{i+1}\phi \circ \partial_{i+1}^{\bc}.
        \end{aligned} \]
        
        We now deal with the case $i=0$. We have
      \[\begin{aligned}
          \Der^0_{\KK}(P,S)=\Der_{\KK}(R,S)=\bigoplus_{j=1}^n S\frac{\partial}{\partial x_j}
      \end{aligned}
      \]
      hence
\[
\begin{aligned}
    \Der^0_\KK(P,S)_\bc \cong \bigoplus_{j=1}^n S_{\bc + \be_j}. 
\end{aligned}
\]
If $|\bb|>1$  or $x^\ba\in I$, then  $S_{\bc + \be_j}=0$ for all $j$ and $\Der^0_\KK(P,S)_\bc$ vanishes. If instead  $\bb=e_k$  and $x^\ba\notin I$, then  
\[
\begin{aligned}
    \Der^0_\KK(P,S)_\bc \cong S_{\bc + \be_k}\cong \KK. 
\end{aligned}
\]
	Thus, we also have that $\Der_\KK^0(P,S)_\bc$ agrees with $\Hom_\KK(W_0^\bc,\KK)$. 
It is again straightforward to verify that under the above identification, the differential 
         \[ \begin{aligned}
            \Der^0(P,S)_\bc &\to \Der^{1}(P,S)_\bc
        \end{aligned} \]
of the weight $\bc$ part of the derivation complex is dual to $-\partial_0^\bc$.

We conclude that $\partial^\bc$ is a differential, and the homology of $(W_\bullet^\bc,\partial^\bc)$ is dual to the cohomology of $\Der_{\KK}^\bullet(P,S)_\bc$. The result then follows by \eqref{eq.cotangentcohomology}.
		\end{proof}

\begin{remark}
	By Theorem~\ref{thm.stringsdifferential}, choosing a homogeneous basis of $P_i^\new$ consisting of elements represented by strings as in \S\ref{sec:differential}, the differential $\partial^\bc$ takes a very explicit form. For any string $s$ of homological degree at least two with $\wt(\overline s)\in \nabla_\bc$ we obtain
\begin{align*}
		\partial^\bc(\overline s)=&
		\sum_{\substack{q :\ q\in \petal(s)\ \textrm{and}\\ \wt(\overline{s\setminus q})\in \nabla_\bc}}\sign(q;s)(\overline{s\setminus q})
		-\sum_{\substack{q:\ q\in \twig(s)\ \textrm{and}\\ \wt(\overline{s \git q})\in \nabla_\bc}}\sign(q;s)(\overline{s \git q}).
	\end{align*}
	On the other hand $\partial^\bc$ takes any string $s$ of homological degree one with $\wt(\overline s)\in \nabla_\bc$ to either $0$  or $1$.
\end{remark}

We give an easy example where Theorem~\ref{thm.computingTviahomology} allows us to compute the cotangent cohomology.

\begin{example}[$T^2$ and $T^3$ vanishing for two generators]\label{ex:2gens}
Let $I=\langle f_1,f_2\rangle$ be generated by two monomials $f_1,f_2\in R$, fix the Koszul-Tate-Taylor resolution $P_\bullet\to S=R/I$ coming from this choice of generators. For any $\bc\in\ZZ^n$, the only possible generator in $W_2^\bc$ is $\{f_1f_2\}$, the only possible generator in $W_3^\bc$ is $[\{f_1\}\{f_2\}]$, and the only possible generators in $W_4^\bc$ are $[\{f_1f_2\}\{f_j\}]$, for $j\in\{1,2\}$. 

The condition $T^2_{S/\KK}\neq 0$ would imply that
    \[ W_3^\bc=0\to W_2^\bc=\KK\{f_1f_2\}\to W_1^\bc=0. \]
    This means in particular that
    \[ \begin{cases}
        x^{\wt(f_1)+\wt(f_2)+\bc}\in I\\
        \wt(f_1)+\bc\notin \ZZ^n_{\geq0}\\
        \wt(f_2)+\bc\notin \ZZ^n_{\geq0},
    \end{cases} \]
    which is impossible. Hence we conclude that $T^2_{S/\KK}=0$.

Similarly, the condition $T^3_{S/\KK}\neq 0$ would imply that
    \[ W_4^\bc=0\to W_3^\bc=\KK[\{f_1\}\{f_2\}]\to W_2^\bc=0 \]
    which in turn implies
    \[ \begin{cases}
        x^{\wt(\lcm(f_1,f_2))+\wt(f_j)+\bc}\in I & \text{for } j=1,2\\
        \wt(\lcm(f_1,f_2))+\bc\notin \ZZ^n_{\geq0}
    \end{cases} \]
Again, this is impossible and we conclude that $T^3_{S/\KK}=0$.
	
In general, we need not have $T_{S/\KK}^4=0$ in this setting. Consider for example
	\[I=\langle f_1,f_2\rangle=\langle x_1^2,x_1x_2\rangle \subseteq \KK[x_1,x_2].\]
	The generators $y$ of homological degree $3,4,5$ (up to sign) and their weights are found in Table \ref{table:2gens}.
	In Figure \ref{fig:2gens}, we plot the weights $\bc\in\ZZ^2$ for which these generators belong to $W_\bullet^\bc$. Black and gray dots correspond to generators in homological degrees $3$, and $5$, respectively. The red and blue hollow circles correspond to the two generators in homological degree $4$. Note that the diagram extends further upwards.

We obtain that for $\bc=(-2,-2)$, the only one of these generators belonging to $W_\bullet^\bc$ is $[\{f_1f_2\}\{f_1\}]$, so in degrees $5,4,3$ the complex looks like $0\to \KK \to 0$. Hence, $(T^4_{S/\KK})_{(-2,-2)}\cong \KK$.

For $\bc=(-3,-1)$, the only ones of these generators belonging to $W_\bullet^\bc$ are $[\{f_1f_2\}\{f_1\}]$, $[\{f_1f_2\}\{f_1\}]$, and $[\{f_1\}\{f_2\}]$. The complex in degrees $5,4,3$ looks like $0\to \KK^2\to \KK$, with the second map a surjection. Hence,  $(T^4_S)_{(-3,-1)}\cong \KK$.

It is straightforward to check that these are the only degrees with non-zero $T^4$. In all other degrees $\bc$, either $W_4^\bc=0$, or $W_4^\bc\cong \KK$ with either an injection into $W_3^\bc$ or a surjection from $W_4^\bc$ (see Figure \ref{fig:2gens}).
\end{example}
	\begin{table}
\[
	\begin{array}{c c c c}
		\textrm{Degree} &y & \textrm{Weight} & \overline\partial(y)\\
		\hline
		3 & [\{f_1\}\{f_2\}] & (3,1) & -\{f_1f_2\}\\
		4 & [\{f_1f_2\}\{f_1\}] & (3,2) & -[\{f_2\}\{f_1\}]=[\{f_1\}\{f_2\}]\\
		4 & [\{f_1f_2\}\{f_2\}] & (4,1) & [\{f_1\}\{f_2\}]\\
		5 & [\{f_1f_2\}\{f_1\}\{f_2\}] & (5,2) & -[\{f_1f_2\}\{f_2\}]+[\{f_1f_2\}\{f_1\}]\\
		5 & [[\{f_1\}\{f_2\}]\{f_1\}] & (5,1) & [\{f_1f_2\}\{f_1\}]\\
		5 & [[\{f_1\}\{f_2\}]\{f_2\}] & (4,2) &[\{f_1f_2\}\{f_2\}]\\
		5 & [\{f_1f_2\}\{f_1f_2\}] & (4,2) &-2[\{f_1f_2\}\{f_2\}]+2[\{f_1f_2\}\{f_1\}]\\
	\end{array}
\]
\caption{Generators for Example \ref{ex:2gens}}\label{table:2gens}
\end{table}

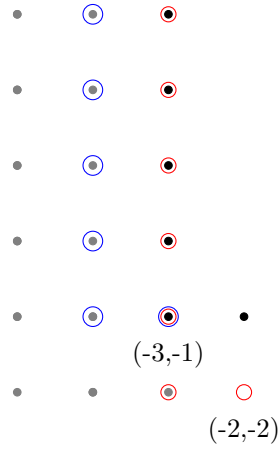
\begin{figure}
	\begin{tikzpicture}
		\draw[fill] (-3,-1) circle [radius=0.05] ;
		\draw[fill] (-2,-1) circle [radius=0.05] ;
		\draw[fill] (-3,0) circle [radius=0.05] ;
		\draw[fill] (-3,1) circle [radius=0.05] ;
		\draw[fill] (-3,2) circle [radius=0.05] ;
		\draw[fill] (-3,3) circle [radius=0.05] ;
		\draw[red] (-3,-2) circle [radius=0.1]  ;
		\draw (-3,-1.2) node[below] {(-3,-1)}  ;
		\draw (-2,-2.2) node[below] {(-2,-2)}  ;
		\draw[red] (-2,-2) circle [radius=0.1] ;
		\draw[red] (-3,-1) circle [radius=0.1] ;
		\draw[red] (-3,-0) circle [radius=0.1] ;
		\draw[red] (-3,1) circle [radius=0.1] ;
		\draw[red] (-3,2) circle [radius=0.1] ;
		\draw[red] (-3,3) circle [radius=0.1] ;
		\draw[blue] (-4,-1) circle [radius=0.13] ;
		\draw[blue] (-3,-1) circle [radius=0.13] ;
		\draw[blue] (-4,0) circle [radius=0.13] ;
		\draw[blue] (-4,1) circle [radius=0.13] ;
		\draw[blue] (-4,2) circle [radius=0.13] ;
		\draw[blue] (-4,3) circle [radius=0.13] ;
		\draw[fill,gray] (-5,-2) circle [radius=0.05] ;
		\draw[fill,gray] (-4,-2) circle [radius=0.05] ;
		\draw[fill,gray] (-5,-1) circle [radius=0.05] ;
		\draw[fill,gray] (-5,0) circle [radius=0.05] ;
		\draw[fill,gray] (-5,1) circle [radius=0.05] ;
		\draw[fill,gray] (-5,2) circle [radius=0.05] ;
		\draw[fill,gray] (-5,3) circle [radius=0.05] ;
		\draw[fill,gray] (-5,-1) circle [radius=0.05] ;
		\draw[fill,gray] (-4,-1) circle [radius=0.05] ;
		\draw[fill,gray] (-5,0) circle [radius=0.05] ;
		\draw[fill,gray] (-5,1) circle [radius=0.05] ;
		\draw[fill,gray] (-5,2) circle [radius=0.05] ;
		\draw[fill,gray] (-5,3) circle [radius=0.05] ;
		\draw[fill,gray] (-4,-2) circle [radius=0.05] ;
		\draw[fill,gray] (-3,-2) circle [radius=0.05] ;
		\draw[fill,gray] (-4,-1) circle [radius=0.05] ;
		\draw[fill,gray] (-4,0) circle [radius=0.05] ;
		\draw[fill,gray] (-4,1) circle [radius=0.05] ;
		\draw[fill,gray] (-4,2) circle [radius=0.05] ;
		\draw[fill,gray] (-4,3) circle [radius=0.05] ;
	\end{tikzpicture}
	\caption{Weights $\bc$ with $W_i^\bc\neq 0$ in Example \ref{ex:2gens}}\label{fig:2gens}
\end{figure}
\begin{remark}
Suppose that $I=\langle f_1,f_2\rangle$ is generated by two monomials $f_1=g_1\cdot h,f_2=g_2\cdot h$ with $g_1,g_2,h$ pairwise relatively prime monomials.
Then the analysis of Example \ref{ex:2gens} can be extended to show that $T^4_{S/\KK}=0$. In particular, $T^4$ vanishes whenever $I$ is a square-free monomial ideal with two generators. We leave the details to the reader. 
\end{remark}

\section{Simplicial complexes and Stanley-Reisner rings}\label{sec:SR}
\subsection{Basic notions}\label{sec:basic}
Let $\cV$ be a finite non-empty set. 
An abstract simplicial complex on the vertex set $\cV$ is a collection $\cK$ of subsets of $\cV$ such that if $f\in \cK$, then any subset of $f$ is also in $\cK$. Any simplicial complex on $\cV$ is always a subcomplex of the complex $\Delta(\cV)$, the set of all subsets of $\cV$.

The \emph{faces} of $\cK$ are simply the elements of $\cK$. The \emph{dimension} of a face is one less than its cardinality. Zero-dimensional faces of $\cK$ are called \emph{vertices}. One-dimensional faces of $\cK$ are called \emph{edges}.
We let $\cV(\cK)$ consist of all those $v\in \cV$ such that $\{v\}$ is a vertex of $\cK$.

To any simplicial complex $\cK$, we may construct an associated Stanley-Reisner ideal $I_\cK$ and ring $S_\cK$, cf.~\cite[Chapter II]{stanley}. Let $R=\KK[ x_v\ |\ v\in \cV]$. The Stanley-Reisner ideal of $\cK$ is the monomial ideal
\[
	I_\cK=\left\langle \prod_{v\in f} x_v\ \middle|\ f\subseteq \cV,\ f\notin \cK\right\rangle.
\]
Note that to generate $I_\cK$, it suffices to consider those monomials corresponding to minimal non-faces $f$ of $\cK$.
The Stanley-Reisner ring $S_\cK$ is the quotient $R/I_\cK$. Since $I_\cK$ is a monomial ideal, it is homogeneous with respect to the natural $\ZZ^\cV$-grading on $R$, and hence $S_\cK$ is also $\ZZ^\cV$-graded.

\begin{remark}\label{rmk.vertexset}
In general, we do not insist that $\cV(\cK)=\cV$. However, given a simplicial complex $\cK$ on vertex set $\cV$ with $\cV(\cK) \subsetneqq \cV$, we may also consider $\cK$ as a simplicial complex on the vertex set $\cV(\cK)$. Up to canonical isomorphism, the resulting Stanley-Reisner rings are isomorphic.
\end{remark}

The \emph{topological realization} of a simplicial complex $\cK$ is the union in $\RR^\cV$ of the simplices
\[
	\conv \{e_v\ |\ v\in f\}\subseteq \RR^\cV
\]
where $\conv$ denotes convex hull, $e_v$ is the basis vector of $\RR^\cV$ corresponding to $v\in\cV$, and $f$ ranges over all faces of $\cK$. In this paper, when considering topological chain complexes, homology, and the like, we will abuse notation by identifying $\cK$ with its topological realization.

We recall from \S\ref{sec:overview} than for any face $f\in\cK$, its \emph{link} in $\cK$ is 
\[
	\link(f,\cK)=\{g\in \cK\ |\ f\cap g=\emptyset\ \textrm{and}\ f\cup g\in \cK\}. 
\]
Given simplicial complexes $\cK_1$ and $\cK_2$ on $\cV_1$ and $\cV_2$, their \emph{join} $\cK_1*\cK_2$ is the simplicial complex on the disjoint union of $\cV_1$ and $\cV_2$ whose faces are of the form
\[
	f_1\cup f_2\qquad\mathrm{where}\ f_1\in \cK,f_2\in \cK'.
\]

\begin{example}\label{ex:rnc3}
	Consider the simplicial complex $\cK$ on the vertex set $\cV=\{0,1,2,3\}$ with faces as pictured in Figure \ref{fig:rnc3}. The Stanley-Reisner ideal $I_\cK$ is generated by $x_0x_2$, $x_1x_3$, and $x_0x_3$. It is well-known that $S_\cK$ is a Gr\"obner degeneration of the coordinate ring $S'$ of the twisted cubic curve, see e.g.~\cite[\S4]{sturmfels}.

	By \cite[\S3-4]{stevens}, there are infinitely many $i\in \NN$ for which $T^i_{S'/\KK}\neq 0$. By semicontinuity of cotangent cohomology (cf.~\cite[\S1.1.7]{behnkechristophersen}), the same follows for $T^i_{S_\cK/\KK}=T^i(\cK)$. This means that for arbitrary $i$, we cannot hope to realize the $\ZZ^4$-graded pieces of $T^i(\cK)$
	as $(i-1)$st relative cohomology groups of topological subspaces of the cone over $\cK$, as in \cite[Theorem 13]{ac1} for $i=1,2$. Indeed, these cohomology groups all vanish for $i>3$ for dimension reasons.
\end{example}
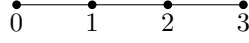
\begin{figure}
	\begin{tikzpicture}
		\draw[fill] (0,0) circle [radius=0.05] node[below] {$0$};
		\draw[fill] (1,0) circle [radius=0.05] node[below] {$1$};
		\draw[fill] (2,0) circle [radius=0.05] node[below] {$2$};
		\draw[fill] (3,0) circle [radius=0.05] node[below] {$3$};
	\draw (0,0)--(3,0);
	\end{tikzpicture}

	\caption{The simplicial complex of Example \ref{ex:rnc3}}\label{fig:rnc3}
\end{figure}

\subsection{Flag complexes}\label{sec:flag}
A \emph{flag complex} is an abstract simplicial complex $\cK$ such that the minimal non-faces have cardinality $2$. This is equivalent to $I_\cK$ being generated by quadratic monomials. Likewise, it is equivalent to the condition that $\cV=\cV(\cK)$, and $f\subseteq \cV$ belongs to $\cK$ if and only if every $2$-element subset of $f$ belongs to $\cK$. These ideals $I_\cK$ also arise as edge ideals of graphs, see e.g.~\cite{edge} for details.
It is straightforward to see that any link of a flag complex remains a flag complex.

When considering a flag complex $\cK$, it will  be useful to consider the minimal non-faces of $\cK$:
we let $G(I_\cK)$ be the graph with vertices $\cV$ and edges consisting of all $2$-element subsets of $\cV$ not belonging to $\cK$. The simplicial complex $\cK$ is determined by $G(I_\cK)$.
To avoid confusion between $G(I_\cK)$ and $\cK$, when drawing figures, we will denote edges as dashed lines.
For example,  in Figure \ref{fig:graph}, we depict $G(I_\cK)$ for the complex $\cK$ from Example \ref{ex:rnc3}.

Given any generator $y\in P_\bullet$ in a Koszul-Tate-Taylor resolution of $S_\cK$, we let $G(y)$ be the subgraph of $G(I_\cK)$ whose edges are the edges of $G(I_\cK)$ appearing in any string $s$ representing $y$, and whose vertices are exactly the vertices of the aforementioned edges. We note that a primitive generator is determined up to sign by $G(y)$.

\begin{remark}\label{rem:graphweight}
	Let $\bb\in\{0,1\}^\cV$, and let $y\in P_\bullet$ be a primitive generator. Then $y\in W_\bullet^{-\bb}$ if and only if $b$ is contained in the vertices of $G(y)$, and the complement of $b$ in the set of vertices of $G(y)$ is a face of $\cK$. This latter condition is equivalent to requiring that the restriction of $G(I_K)$ to the complement of $b$ in the vertex set of $G(y)$ has no edges. In particular, every edge of $G(y)$ must have at least one vertex in $b$.
\end{remark}

\subsection{Shellings}
The concept of a \emph{shelling}, which we now recall, can be useful in computing the cohomology groups of the topological realization of a simplicial complex. We will use this in Example \ref{ex:hex1111}.

Let $\cK$ be an abstract simplicial complex. A shelling (cf.~\cite[Definition 2.1]{bjorner}) of $\cK$ is a sequence $(f_1,\ldots,f_k)$ whose entries $f_i$ are the maximal faces of $\cK$, and that satisfies the following property for each $1\leq j<k$: each maximal face of
\begin{equation}\label{eqn:shelling}
	\bigcup_{i=1}^{j} \Delta(f_i) \cap \Delta(f_{j+1})
\end{equation}
has dimension $(\dim f_{j+1})-1$.

Not every simplicial complex is shellable, but for those that are, it is easy to understand their topology. Given a shelling $(f_1,\ldots,f_k)$, and $1\leq j<k$, we say that $f_{j+1}$ is a \emph{homology face} if the complex from \eqref{eqn:shelling} is the boundary of $\Delta(f_{j+1})$, that is, contains all proper subfaces of $f_{j+1}$.

\begin{theorem}[{\cite[Theorem 4.1]{bjorner}}]
	Let $\cK$ be a simplicial complex with shelling $(f_1,\ldots,f_k)$. The homotopy type of $\cK$ is a wedge sum spheres, with one $(\dim f_{j+1})-1$-dimensional sphere for each homology face $f_{j+1}$. 
\end{theorem}

\begin{example}\label{ex:shelling}
	Consider the simplicial complex $X$ on the vertex set $\{e_1,e_2,e_3,e_4,3,6\}$ 
that has the maximal faces
\begin{align*}
\{e_1,e_2,e_3,e_4\},\ \{6,e_1,e_2,e_4\},\  \{6,e_1,e_2,e_3\},\\
\{3,e_2,e_3,e_4\},\ 
	\{3,e_1,e_3,e_4\},\ \{6,e_3,e_4\},\ \{3,e_1,e_2\}.
\end{align*}
	The reason for this notation will become apparent in Example \ref{ex:hex1111}.

The list of maximal simplices in this order gives a shelling; only the last two simplices are homology faces, so $X$  is homotopy equivalent to the wedge sum of two $2$-spheres, hence has 
$H_2(X,\KK)\cong \KK^2$.
If we replace the last two simplices in the shelling with $\{6,e_2,e_3,e_4\}$ and $\{3,e_1,e_2,e_3\}$, the modified list gives a shelling of the modified complex, and now there are no spanning simplices. Thus, the classes of the boundaries of $\{6,e_2,e_3,e_4\}$ and $\{3,e_1,e_2,e_3\}$
give a basis for $H_2(X,\KK)\cong \KK^2$.
\end{example}

\begin{figure}
	\begin{tikzpicture}
		\draw[fill] (0,0) circle [radius=0.05] node[below] {$0$};
		\draw[fill] (1,0) circle [radius=0.05] node[below] {$1$};
		\draw[fill] (2,0) circle [radius=0.05] node[below] {$2$};
		\draw[fill] (3,0) circle [radius=0.05] node[below] {$3$};
		\draw[dashed] plot [smooth] coordinates {(1,0) (1.5,-.5) (2.5,-.5) (3,0)};
		\draw[dashed] plot [smooth] coordinates {(0,0) (0.5,.3) (1.5,.3) (2,0)};
		\draw[dashed] plot [smooth] coordinates {(0,0) (0.5,.5) (2.5,.5) (3,0)};
	\end{tikzpicture}
	\caption{The graph of Example \ref{ex:rnc3}}\label{fig:graph}
\end{figure}
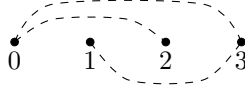

\section{The localization formula}
Let $\cK$ be an abstract simplicial complex on vertex set $\cV$. As in \S\ref{sec:overview}, given $\bc\in\ZZ^n$, denote by $\ba,\bb\in\ZZ^\cV_{\geq0}$ the unique weights with disjoint supports $a,b\subseteq \cV$ satisfying $\bc=\ba-\bb$.

\begin{theorem}[Localization formula]\label{thm:localizationformula}
In the above notation we have
\begin{enumerate}
    \item $T^i_\bc(\cK)=0$ whenever $a\notin \cK$,
    \item $T^i_\bc(\cK)=T^i_{-\bb}(\link(a,\cK))$ whenever $a\in\cK$,
\end{enumerate}
for every $i\in \ZZ_{\geq1}$.
\end{theorem}
\begin{proof}
	The claims were shown in \cite[Proposition 11]{ac1} when $i=1,2$. Henceforth, we will always assume that $i\geq 2$ for simplicity. 

	We first show the first item.
If $a\notin\cK$, then $x^{\ba}\in I_{\cK}$, so that any $\bu\in\ZZ^\cV_{\geq0}$ satisfying $\bu+\bc\in\ZZ^{\cV}_{\geq0}$ also satisfies
    \[ x^{\bu+\bc}=x^\ba\cdot x^{\bu-\bb}\in I_{\cK} \; . \]
    In other words, the set $\nabla_\bc$ from \S\ref{sec:cot} is empty, and hence $W^\bc_\bullet=0$. The first item then follows from Theorem~\ref{thm.computingTviahomology}. 

    We now assume that $a\in \cK$, and consider the complex $\cK'=\link(a,\cK)$. We will view this as a complex on the vertex set $\cV\setminus a$, so $I_{\cK'}\subseteq \KK[x_v\ |\ v\in\cV\setminus a]$. Let  $\cG$ be the set of minimal generators of $I_\cK$.
This comes with an inclusion $\lambda$ into $I_\cK$, and with a map $\lambda'$ to $I_{\cK'}$ obtained by setting $x_v=1$ for all $v\in a$. It is straightforward to see that the image of $\cG_{\cK'}=\cG_\cK$ does indeed generate $I_{\cK'}$. Using $\cG$ with the maps $\lambda,\lambda'$ respectively, we obtain
 Koszul-Tate-Taylor resolutions for $I_\cK$ and $I_{\cK'}$.

 With notation as in \S\ref{sec:cot}, we will use $\nabla_\bc(\cK)$, $W_{i,\bu}(\cK)$, $W_i^\bc(\cK)$, and $\partial^\bc(\cK)$ (respectively $\nabla_\bc(\cK')$, $W_{i,\bu}(\cK')$, $W_i^\bc(\cK')$, and $\partial^\bc(\cK')$) to denote the $\nabla_\bc$, $W_{i,\bu}$, $W_i^\bc$, and $\partial^\bc$ for the Koszul-Tate-Taylor resolutions for $I_\cK$ (respectively $I_{\cK'}$).
   Since both resolutions use the same generator set $\cG$, there is a bijection between the algebra generators coming from strings for these two resolutions.
    Letting $\pi \colon \ZZ^\cV\to \ZZ^{\cV\setminus a}$ be the projection, we thus obtain for every $\bv\in\ZZ_{\geq 0}^{\cV\setminus a}$ and every $i\geq 1$ an isomorphism
    \[
	    W_{i,\bv}({\cK'})\cong \bigoplus_{\substack{\bu\in\ZZ_{\geq 0}^{\cV}\\ \pi(\bu)=\bv}} W_{i,\bu}(\cK).
    \]

We now claim that 
\[
	\nabla_\bc(\cK)=\{\bu\in\ZZ_{\geq 0}^{\cV}\ |\ \pi(\bu)\in\nabla_{\pi(\bc)}(\cK')\}.
\]
Indeed, since $\ba\in\ZZ_{\geq 0}$, for $\bu\in\ZZ_{\geq 0}$ we have that $\bu+\bc=\bu+\ba-\bb\in\ZZ_{\geq 0}^{\cV}$ if and only if $\pi(\bu)+\pi(\bc)=\pi(\bu)-\pi(\bb)\in\ZZ_{\geq 0}^{\cV\setminus a}$.
Clearly for $\bu\in\ZZ_{\geq 0}^\cV$ with $\bu+\bc\in\ZZ_{\geq 0}^{\cV}$, we have that $x^{\bu+\bc}\in I_\cK$ implies that $x^{\pi(\bu)+\pi(\bc)}\in I_{\cK'}$. On the other hand, for $\bu\in\ZZ_{\geq 0}^\cV$ with $\bu+\bc\in\ZZ_{\geq 0}^{\cV}$ and with $x^{\pi(\bu)+\pi(c)}\in I_{\cK'}$, we must also have $x^{\bu+\bc}\in I_\cK$, since the support of $\bu+\bc$ is the union of the support of $\pi(\bu)+\pi(\bc)$ with $a$. The claim follows.

Using the claim and the above isomorphism, we now have
\begin{align*}
	W_i^{\pi(\bc)}(\cK')=\bigoplus_{\bv\in \nabla_{\pi(\bc)}(\cK')} W_{i,\bv}(\cK')&\cong
	\bigoplus_{\substack{\bu\in\ZZ^\cV_{\geq 0}\\ \pi(\bu)\in\nabla_{\pi(\bc)}(\cK')}} W_{i,\bu}(\cK)\\
	&= \bigoplus_{\substack{\bu\in \nabla_{\bc}(\cK)}} W_{i,\bu}(\cK)\\
	&=W_i^{\bc}(\cK).
\end{align*}
Moreover, these isomorphisms are compatible with the differentials $\partial^\bc(\cK)$ and $\partial^\bc(\cK')$. 
The second claim now follows by again applying Theorem \ref{thm.computingTviahomology}.
\end{proof}
\begin{remark}It is possible to partially generalize the localization formula to monomial ideals that are not square-free.
	If $I\subseteq \KK[x_1,\ldots,x_n]=R$ is an arbitrary monomial ideal and $a\subseteq \cV=\{1,\ldots,n\}$, we may consider the monomial ideal $I_a\subseteq \KK[x_i\ |\ i\in \cV\setminus a]=R'$ with generators obtained from monomial generators of $I$ after setting $x_i=1$ for all $i\in a$. 
	If $I=I_\cK$ and $a\in\cK$, then we in fact used above that $I_{\link(a,\cK)}=(I_\cK)_a$.
Set $S=R/I$ and $S_a=R'/I_a$.

	Proceeding with notation similar to the above proof of Theorem \ref{thm:localizationformula},
	in general, it is not true that 
	\begin{equation}\label{eqn:nablaI}
	\nabla_\bc(I)=\{\bu\in\ZZ_{\geq 0}^{\cV}\ |\ \pi(\bu)\in\nabla_{\pi(\bc)}(I_a)\}.
\end{equation}
This might fail, since for $\bu\in\ZZ_{\geq 0}^\cV$ with $\bu+\bc\in\ZZ_{\geq 0}^{\cV}$ and with $x^{\pi(\bu)+\pi(\bc)}\in I_a$, it need not follow that $x^{\bu+\bc}\in I$.

Given $a\subset \cV$ and a monomial ideal $I$, let $\ba^{\min}=(\ba^{\min}_i)_{i\in\cV}\in \ZZ^\cV_{\geq 0}$ be defined by setting $\ba^{\min}_i$ to be the largest power of $x_i$ appearing in any minimal generator of $I$ whenever $i\in a$, and zero otherwise. Whenever $\ba-\ba^{\min}\in\ZZ_{\geq 0}^\cV$, 
it is straightforward to verify that \eqref{eqn:nablaI} does indeed hold, and the rest of the proof of Theorem \ref{thm:localizationformula} goes through to show that for $i\geq 2$, 
\[
	(T_{S/\KK}^i)_\bc\cong (T^i_{S_a/\KK})_{-\bb}.
\]
\end{remark}

\section{Weight bounds}\label{sec:degree}

\subsection{General bounds}\label{sec:general}
We now use the Koszul-Tate-Taylor resolution to establish some general bounds on the negative part of the weight $\bc$ implying vanishing of $T^i_\bc(\cK)$.

\begin{theorem}\label{thm:bound}
	Let $\cK$ be a simplicial complex and $\bc=\ba-\bb\in \ZZ^\cV$. Let $\delta$ be the maximum degree of a minimal generator of $I_\cK$, or equivalently, one more than the dimension of a minimal non-face of $\cK$. Let $i\in \ZZ_{\geq 2}$.
\begin{enumerate}
	\item $T^i_{\bc}=0$ if $|\bb|>i\cdot \delta$;
	\item $T^i_{\bc}=0$ if $\bb\notin \{0,\ldots,i-1\}^\cV$;
	\item If $\cK$ is a flag complex and $i\geq 3$, $T^i_{\bc}=0$ if $\bb\notin \{0,\ldots,i-2\}^\cV$.
\end{enumerate}
\end{theorem}

\begin{proof}
	By the localization formula (Theorem \ref{thm:localizationformula}), we may assume without loss of generality that $\bc=-\bb$.
	We now take the Koszul-Tate-Taylor resolution of $S_\cK$ obtained from the minimal generators of $I_\cK$. 
The weight of any generator $y\in P_i^\new$ satisfies:
\begin{enumerate}
	\item $|\wt(y)|\leq i\cdot \delta$;
	\item for any $v\in \cV$, the $v$th coordinate of $\wt(y)$ is at most $i-1$.
\end{enumerate}
The first two claims follow immediately by Theorem 
\ref{thm.computingTviahomology} since $W_i^\bc=0$.

For the third claim, to simplify notation we will assume that $\cV=\cV(\cK)=\{1,\ldots,n\}$. Suppose that $\bb$ has $i-1>1$ at some position, which we may assume to be the first. Consider now any generator $y\in P_i^\new$ satisfying $\wt(y)-\bb\in \ZZ_{\geq 0}^n$.
Then $y$ must have the form
\[
	y=[\{f_1\}\{f_2\}\cdots\{f_{i-1}\}]
\]
for some $f_1,\ldots,f_{i-1}\in I_\cK$, all with $x_1|f_j$ for $j=1,\ldots,i-1$, and moreover, the remaining coordinates of $\bb$ are all zero or one.

If $\wt(y)=\bb$, then $y'=[\{f_1f_2\}\{f_2\}\cdots\{f_{i-1}\}] \in W_{i+1}^\bc$ since $x^{\wt(y')-\bb}=f_2/x_1\notin I_\cK$. If instead $\wt(y)\neq\bb$, there exists $j\leq i-1$ so that the support of $\wt(f_j/x_1)$ is disjoint from the support of $\bb$. In that case,
$y'=[\{f_1f_j\}\{f_2\}\cdots\{f_{i-1}\}] \in W_{i+1}^\bc$ since the supports of $\wt(y)-\bb$ and $\wt(y')-\bb$ coincide. Either way $\partial^\bc(y')=y$, so $y$ is a boundary. It follows that $W_i^\bc=\partial^\bc(W_{i+1}^\bc)$ and the claim follows again by Theorem 
\ref{thm.computingTviahomology}.
\end{proof}

\subsection{Koszul case}
We now show how to improve on the bounds of \S\ref{sec:general} in a special case. Let $S$ be a finitely generated $\ZZ$-graded $\KK$-algebra, generated in degree $1$ with homogeneous maximal ideal $\mfm$. Recall that $S$ is \emph{Koszul} if for all $i\geq 0$, $\Ext^i_S(S/\mfm,S/\mfm)$ is concentrated in weight $i$. See \cite[Definition-Theorem 1]{froberg} for alternate characterizations of this property.

\begin{theorem}\label{thm:koszul}
	Let $R$ be a standard graded polynomial ring, and let $I$ be a homogeneous ideal of $R$ such that $S=R/I$ is Koszul. Then the algebra generators in homological degree $i$ of a minimal Koszul-Tate resolution $P\to S$ have weight $i+1$. This is in particular true if $I$ is a monomial ideal generated by quadrics, or more generally, has a quadratic Gr\"obner basis.
\end{theorem}
\begin{proof}
	We refer to \cite{avramov} for definitions and details about acyclic closures and minimal models.

Let $Q$ be an acyclic closure of $\KK$ over $S$. This is a minimal graded resolution of $\KK$ over $S$, see \cite[Theorem 3.2.23]{debellevue}.
Since $S$ is Koszul, we have that the number of free generators of $Q_{ij}$ as an $S$-module is $0$ unless $i=j$; here $Q_{ij}$ is the weight $j$ piece of the homological degree $i$ piece $Q_i$.
On the other hand, the number of algebra generators in $Q_i$ of weight $j$ is by definition $\varepsilon_{ij}$, the $ij$th deviation (cf.~\cite[\S3.1]{berglund}), so we obtain that $\varepsilon_{ij}=0$ unless $i=j$.

Finally, let $P\to S$ be a minimal Koszul-Tate resolution of $S=R/I$ with $P_0=R$; this is a minimal model for $S$ over $R$. It follows from \cite[Lemma 5]{berglund} that $\varepsilon_{ij}$ is the number of algebra generators of $P_{i+1}$ in weight $j$. In other words, all algebra generators of $P_{i+1}$ must have weight $i$.
The first claim of the theorem follows. 

For the second claim, we simply remind the reader that any graded ring $S/I$ where $I$ has a quadratic Gr\"obner basis is Koszul; this includes in particular the case that $I$ is a monomial ideal generated by quadrics, see~\cite[Theorem 2.2]{koszul}.
\end{proof}

\begin{cor}\label{cor:koszul}
	Let $S$ be a finitely generated $\ZZ$-graded $\KK$-algebra that is Koszul. Then for all $i\in\ZZ_{\geq 0}$ and for all $j\in \ZZ_{\geq 0}$, the weight $-j$ graded piece of $T^i_{S/\KK}$ vanishes if $j> i+1$. In particular, for any flag complex $\cK$ and $\ba,\bb\in\ZZ_{\geq0}^\cV$ with disjoint supports, if $|\bb|>i+1$ then $T^i_{\ba-\bb}(\cK)=0$.
\end{cor}
\begin{proof}
	Let $P_\bullet\to S$ be a minimal $\ZZ$-graded Koszul-Tate resolution of $S$. 
	By Theorem \ref{thm:koszul}, the $S$-module $\Der_\KK^i(P,S)=\Der_\KK(P_i,S)$ is generated as an $S$-module in degrees larger than or equal to $-(i+1)$. The first claim follows since $T^i_{S/\KK}$ is the $i$th cohomology of the complex $\Der^\bullet(P,S)$. 

	The second claim is immediate since for a flag complex $\cK$, $I_\cK$ is a monomial ideal generated by quadrics, and hence $S_\cK$ is Koszul.
\end{proof}

\begin{remark}\label{rem:koszul}
	We conjecture that the bound in Corollary \ref{cor:koszul} can be improved by one, that is, for $S$ Koszul, the weight $-(i+1)$ graded piece of $T^i_{S/\KK}$ also vanishes. In the case that $S=S_\cK$ for a flag complex $\cK$, this was known for $i=1,2$ \cite[Lemma 5.2]{ic1}. We show this as well for $i=3,4$ in the case of flag complexes, see Propositions \ref{prop:t3bound} and \ref{prop:t4bound}.
\end{remark}

\section{The non-splitting complex}\label{sec:simplified}
\subsection{Splittings} \label{sec:splitting}
Fix a monomial ideal $I\subseteq \KK[x_1,\ldots,x_n]$ and a weight $\bc=\ba-\bb\in \ZZ^n$.
Fix a Koszul-Tate-Taylor resolution $P\to S=R/I$.
The complex $W_\bullet^\bc$ of Theorem 
\ref{thm.computingTviahomology} is quite large in general. In cases of special $\bb$, we would like to replace it by a smaller complex
$\widetilde W_\bullet^\bc$ that will be more amenable to computations.

\begin{definition}
	For each primitive generator $y\in W_{i}^\bc$ with $i\geq 0$, a \emph{splitting} of $y$ is a quasiprimitive generator $y'\in W_{i+1}^\bc$ such that 
$\partial^\bc (y')-y$ is a sum of quasiprimitive generators, i.e.~the primitive part of the image of $y'$ under $\partial^\bc$ is precisely $y$.
\end{definition}

\noindent Since primitive generators belong to $W_\bullet^\bc$ only if $\bb\in\{0,1\}^\cV$, we will assume throughout this section that $\bb\in\{0,1\}^\cV$.

\begin{example}
	Generators $y$ of homological degrees $0$ or $1$ never have a splitting. In degree $2$, the generator $\{f_1f_2\}$ has at most one splitting, namely $-[\{f_1\}\{f_2\}]$. In degree $3$, the potential splittings of $\{f_1f_2f_3\}$ are  $-[\{f_1\}\{f_2f_3\}]$, $-[\{f_1f_2\}\{f_3\}]$, $-[\{f_3f_1\}\{f_2\}]$, and $-[\{f_1\}\{f_2\}\{f_3\}]$.
\end{example}

\begin{remark}\label{rem:splittings}
	In the case $i>1$, $\bc=-\bb$, and $\cK$ a flag complex, it is easy to characterize those primitive generators $y\in W_i^\bc$ admitting a splitting. Indeed, let $G(y)$ be the subgraph of $G(I_\cK)$ whose edges are the elements of $\petal(s)$ for any string $s$ representing $y$, and with no isolated vertices (cf. \S\ref{sec:flag}).
Then $y$ has a splitting if and only if the edges of $G(y)$ can be partitioned into two non-empty disjoint sets $E_1,E_2$ such that the union of the set of vertices of $G(y)$ not belonging to $b$ with the set of vertices belonging to both $E_1$ and $E_2$ is a face of $\cK$. In particular, $y$ has a splitting if
\begin{enumerate}
	\item $G(y)$ is disconnected; or
	\item $\vert\bb\vert\geq i+1$.
\end{enumerate}
\end{remark}

\begin{remark}\label{rem:nonflag}
	The previous remark can be slightly generalized beyond the case of flag complexes. Suppose that $i>1$, $\bc=-\bb$, $\cV=\cV(\cK)$, and the support of the weight of any non-quadratic minimal generator of $I_\cK$ is disjoint from $b$ (or equivalently, any minimal non-face of $\cK$ of dimension at least $2$ is disjoint from $b$). Then strings representing primitive generators $y\in W_i^\bc$ only involve quadratic generators of $I_\cK$. It thus still makes sense to consider the graph $G(y)$, and the criteria for the existence of a splitting are identical to in Remark \ref{rem:splittings}.
\end{remark}

For $i\geq 0$, let $\widetilde W_i^\bc\subseteq W_i^\bc$ be the subspace generated by primitive generators that do \emph{not} admit splittings. 
There is a natural projection $\widetilde\pi \colon W_\bullet^\bc\to \widetilde W_\bullet^\bc$. We set 
\[
	\widetilde\partial_i^\bc=\widetilde\pi_{i-1}\circ\partial_i^\bc \colon \widetilde W_i^\bc\to \widetilde W_{i-1}^\bc.
\]

We will show:
\begin{lemma}\label{lemma:iscomplex}
The map $\widetilde \partial^\bc$ is a differential, and $\widetilde \pi$ is a map of complexes.
\end{lemma}

We will call $\widetilde W_\bullet^\bc$ the \emph{non-splitting complex}. We will also show the following:
\begin{proposition}\label{prop:nonsplitting}
	Fix $i\geq 1$. 
	Suppose that $W_i^\bc$ and $W_{i-1}^\bc$ are generated by primitive and quasiprimitive generators, and that $\bb\in\{0,1\}^\cV$. Assume further that:
	\begin{enumerate}
		\item 	Any primitive generator $y\in W_{i-2}^\bc$ has at most one splitting;
		\item For any primitive generator $y\in W_{i-1}^\bc$, the difference of any two splittings of $y$ belongs to $\partial_{i+1}^\bc(W_{i+1}^\bc)$.
	\end{enumerate}
	Then the map $H_i(W_\bullet^\bc)\to H_i(\widetilde W_\bullet^\bc)$ induced by $\widetilde \pi$ is an isomorphism.
\end{proposition}

\noindent The proofs of both the lemma and the proposition are delayed until \S\ref{sec:nonsplittingproof}.

\subsection{A homological algebra intermezzo}
To prove Proposition \ref{prop:nonsplitting} we take a brief detour into homological algebra.
Let $(A_\bullet,d_A)$ and $(B_\bullet,d_B)$ be complexes of $\KK$-vector spaces (or abelian groups), and $f \colon A_\bullet\to B_\bullet$ be a map of complexes.
We let $\Cone(f)_\bullet$ be the mapping cone of $f$. We thus have, for every $i\in\ZZ$, that $\Cone(f)_i= A_{i-1} \oplus B_i$ with differential given by
\begin{equation*} 
	d_{\Cone(f)} \colon (a,b) \mapsto (-d_A a, f a + d_B b).  
\end{equation*}
Of course there is a short exact sequence of complexes
\[
0 \to B_\bullet \to \Cone(f)_\bullet \to A_\bullet[-1] \to 0
\]
with maps $b \mapsto (0,b)$ and $(a,b) \mapsto -b$.
Note that only in this subsection and the very next subsection, we will use $a,b$ to refer to elements of the complexes $A,B$, as opposed to referring to the support of the weights $\ba,\bb$.

We also recall that $\coker(f)_\bullet$ has a natural differential making 
\begin{equation*} 
	A_\bullet \overset{f}\to B_\bullet \overset{g} \to \coker(f)_\bullet \to 0
\end{equation*}
into an exact sequence of complexes.
Moreover, although the projection $\Cone(f)_\bullet\to B_\bullet$ sending $(a,b)\in A_{i-1}\oplus B_i$ to $b$ is in general \emph{not} a map of complexes, it is straightforward to check that composing with $g$ does give a map of complexes $h:\Cone(f)_\bullet \to \coker(f)_\bullet$.

\begin{lemma}\label{lemma:homological}
Fix $i \in \ZZ$ and consider the map $H_i(h) \colon H_i(\Cone(f)_\bullet) \to H_i(\coker(f)_\bullet)$ induced by $h$.
Then:
\begin{enumerate}
    \item if $f_{i-2}$ is injective, then $H_i(h)$ is surjective;
    \item the kernel of $H_i(h)$ consists of homology classes of elements of the form $(a,0)$, where $a \in \ker(f_{i-1})$ is such that $d_A a = 0$;
    \item if $f_{i-1}$ is injective, then $H_i(h)$ is injective;
    \item if $f_{i-2}$ is injective and $\ker(f_{i-1})$ is contained in $d_{\Cone(f)}(\Cone(f)_{i+1})$, then $H_i(h)$ is bijective.
\end{enumerate}
\end{lemma}

\begin{proof}
\begin{comment}
    The statements (1)-(4) can be proved via diagram chasing. For example let us see the proof of (1).

    (1) Let $l \in L_i$ be such that $d_L l = 0$. We want to show that the homology class of $l$ lies in the image of $H_i(h)$.
    Since $g_i$ is surjective, there exists $b \in B_i$ such that $g_i b = l$.
    Now $g_{i-1} d_B b = d_L g_i b = d_L l = 0$, i.e.\ $d_B b \in \ker g_{i-1}$.
    By exactness of \eqref{eq:successione_esatta_complessi} there exists $a \in A_{i-1}$ such that $f_{i-1} a = d_B b$.
    Now $f_{i-2} d_A a = d_B f_{i-1} a = d_B d_B b = 0$.
    Since $f_{i-2}$ is injective we get $d_A a = 0$.
    Hence we have found $a \in A_{i-1}$ and $b \in B_i$ such that $d_A a = 0$ and $fa = d_B b$.
    In other words, $(a,b)$ is a cycle in $\Cone(f)_i$ such that $h(a,b) = l$.

    We show a more conceptual proof of the remaining statements.
    (2) Let $(a',b) \in A_{i-1} \oplus B_i = \Cone(f)_i$ be a cycle whose homology class lies in the kernel of $H_i (h)$.
    We want to show that $(a',b)$ is homologous to $(a,0)$ for some cycle $a \in \ker f_{i-1}$.
    We know that $h(a',b) = g_i b$ is a boundary in $L_i$, hence there exists $l \in L_{i+1}$ such that $d_L l = g_i b$.
    Since $g_{i+1}$ is surjective, there exists $b' \in B_{i+1}$ such that $g_{i+1} b' = l$.
    One immediately sees that $d_B b' - b \in B_i$ lies in the kernel of $g_i$.
    By exactness of \eqref{eq:successione_esatta_complessi}....
    TO FINISH
\end{comment}

Let $\psi \colon (\Image f)_\bullet \to B_\bullet$ be the inclusion, so we have a short exact sequence of complexes
 \[
0\to (\Image f)_\bullet \overset{\psi}\to B_\bullet \overset{g}\to \coker(f)_\bullet \to 0.
    \]
By applying the $5$-lemma to the diagram coming from the long exact sequence of homology from this sequence, and the mapping cone of $\psi$, we obtain that the map $\Cone(\psi)_\bullet \to \coker(f)_\bullet$ given by $(b',b) \mapsto gb$, for $b' \in \Image f$, $b \in B$, is a quasi-isomorphism. Hence we can replace $\coker(f)_\bullet$ with $\Cone(\psi)_\bullet$.

The commutative diagram
\begin{equation*}
    \xymatrix{
0 \ar[r] & B_\bullet \ar[r] \ar@{=}[d] & \Cone(f)_\bullet \ar[d] \ar[r] & A_\bullet[-1] \ar[r] \ar@{->>}[d]^f & 0 \\
    0 \ar[r] & B_\bullet \ar[r] & \Cone(\psi)_\bullet \ar[r] & (\Image f)_\bullet [-1] \ar[r] & 0
    }
\end{equation*}
has exact rows. Hence, by the snake lemma we obtain the short exact sequence of complexes
\begin{equation*}
    0 \to (\ker f)_\bullet[-1] \to \Cone(f)_\bullet \to \Cone(\psi)_\bullet \to 0.
\end{equation*}
From the long exact sequence of homology, we thus have an exact sequence
\[
	H_{i-1}(\ker(f)_\bullet)\to H_i(\Cone(f)_\bullet)\to H_i(\Cone(\psi)_\bullet)\to H_{i-2}(\ker (f)_\bullet).
\]
The claims (1)-(4) now easily follow from the exactness of this sequence.
\end{proof}

\subsection{Proof of Lemma~\ref{lemma:iscomplex} and Proposition~\ref{prop:nonsplitting}}\label{sec:nonsplittingproof}
To prove Lemma \ref{lemma:iscomplex} and Proposition \ref{prop:nonsplitting}, we will realize the complex $W_\bullet^\bc$ as a mapping cone. Indeed, for each $i\in \ZZ$, let $B_i$ be the subspace of $W_i^\bc$ generated by primitive generators. Since $\partial^\bc$ takes $B_i$ to $B_{i-1}$, we obtain that $B_\bullet$ is a subcomplex of $W_\bullet^\bc$.

We similarly let $A_{i-1}$ be the subspace of $W_i^\bc$ generated by generators that are not primitive. This also gives a complex $A_\bullet$, with differential $d_A$ induced by $-\partial^\bc$. The map $f_i:A_i\to B_i$ obtained by composing $\partial^\bc$ with the projection from $W_i^\bc$ to $B_i^\bc$ gives a map of complexes $f:A_\bullet\to B_\bullet$.
Indeed, for $a\in A_i$, we have 
\[0=\partial^\bc\circ\partial^\bc(a)=\partial^\bc (-d_A(a)+f(a))=-f(d_A(a))+d_B(f(a)).\]

We now observe that $W_\bullet^\bc$ is precisely the mapping cone of $f:A_\bullet\to B_\bullet$.
Moreover, the cokernel of $f_i$ is $\widetilde W_i^\bc$, and $\widetilde\partial^\bc$ agrees with the natural differential on $\coker f$. Lemma \ref{lemma:iscomplex} follows.

For  Proposition \ref{prop:nonsplitting}, we will apply Lemma \ref{lemma:homological}.  Fix $i\geq 2$, and assume that $W_{i-1}^\bc$ is generated by primitive and quasiprimitive generators, and $\bb\in\{0,1\}^\cV$. Since $\bb\in\{0,1\}^\cV$ the image of any quasiprimitive generator $y\in A_{i-2}$ under $f_{i-2}$ is always non-zero, and in fact, it is precisely a primitive generator $f_{i-2}(y)$. Thus, we see that if every primitive string in $W_{i-2}^\bc$ has at most one splitting, $f_{i-2}$ must be injective.

In a similar fashion, under the assumption that $W_{i}^\bc$ is generated by primitive and quasiprimitive strings, we see that the kernel of $f_{i-1}$ is generated by differences of splittings of primitive generators in $W_{i-1}^\bc$. Under the hypotheses of Proposition \ref{prop:nonsplitting}, it thus follows from Lemma \ref{lemma:homological} that the map on homology induced by $\widetilde \pi$ is an isomorphism.
\qed
\subsection{Applicable cases}
When can we use the non-splitting complex introduced in \S\ref{sec:splitting} to compute cotangent cohomology? The following proposition gives sufficient conditions:
\begin{proposition}\label{prop:applicable}
	Suppose that $\bb\in\{0,1\}^\cV$.
	The hypotheses of Proposition \ref{prop:nonsplitting} are satisfied in the following cases:
	\begin{enumerate}
		\item $i=1$, $i=2$, or $i=3$;
		\item $\bb$ has type $(1)$ or $(1,1)$; or
		\item $I$ is square-free, $\bb$ has type $(1,1,1)$, and $i=4$.
	\end{enumerate}
\end{proposition}
\begin{proof}
For $j\leq 2$, $W_j^\bc$ is generated by primitive generators. For $j=3,4$, $W_j^\bc$ is generated by primitive and quasiprimitive generators. 
The hypotheses for the cases $i=1,2$ follow trivially. 

Quasiprimitive generators in $W_3^\bc$ have the form $y=[\{f_1\}\{f_2\}]$ and are thus determined by their image $\{f_1f_2\}\in W_2^\bc$. The hypotheses for the case $i=3$ follow. 

If $\bb$ has type $(1)$, then for any $j\in\ZZ_{\geq 0}$, all generators in $W_j^\bc$ are primitive, hence the type $(1)$ case is immediate. Suppose instead $\bb$ has type $(1,1)$. 
Then $b$ has the form $\{v_1,v_2\}$ for $v_1,v_2\in\cV$. Any quasiprimitive generator $q\in W_j^\bc$ has the form $q=[p_1 p_2]$ where $p_1,p_2$ are themselves primitive strings, and for $i=1,2$
and any $f\in \cG$ appearing in $p_i$, the support of $\wt(f)$ contains $v_i$ but not the other element of $b$. Hence, up to sign, $q$ is determined by the set of elements $f\in\cG$ appearing in $p_1$ and $p_2$. It follows that any primitive generator has at most one splitting as desired.

Finally, we consider the case when $I=I_\cK$ is square-free, $\bb$ has type $(1,1,1)$, and $i=4$. We just need to show that for any primitive generator $y=\{f_1f_2f_3\}\in W_3^\bc$, the difference of any two splittings is a boundary. 
	Potential splittings for $y$ are $-[\{f_1f_2\}\{f_3\}]$, $-[\{f_1\}\{f_2f_3\}]$, $-[\{f_3f_1\}\{f_2\}]$ and $-[\{f_1\}\{f_2\}\{f_3\}]$.
If $[\{f_1\}\{f_2\}\{f_3\}]\in W_4^\bc$, then the three other potential splittings belong to $W_4^\bc$, as do generators such as $q=[[\{f_1\}\{f_2\}]\{f_3\}]\in W_5^\bc$. But 
\[
	\partial^\bc(q)=-[\{f_1\}\{f_2\}\{f_3\}]+[\{f_1f_2\}\{f_3\}]
\]
and we similarly see that the difference of any two of the splittings of $y$ is a boundary.

To finish, we claim that if $[\{f_1\}\{f_2\}\{f_3\}]\notin W_4^\bc$, then there is at most one splitting. For a contradiction, assume without loss of generality that both $[\{f_1f_2\}\{f_3\}]$ and $[\{f_1\}\{f_2f_3\}]$ belong to $W_4^\bc$. Then $b$ must contain elements $v_1,v_3$ such that $v_1$ is in the support of $\wt(f_1)$ but not $\wt(f_2),\wt(f_3)$ and $v_3$ is in the support of $\wt(f_3)$ but not $\wt(f_1),\wt(f_2)$. The remaining element $v_2$ of $b$ necessarily in the support of $\wt(f_2)$, but cannot be in the support of $\wt(f_1),\wt(f_3)$. But then $[\{f_1\}\{f_2\}\{f_3\}]\in W_4^\bc$, contradiction.
\end{proof}

\begin{remark}
	When $I=I_\cK$ is square-free, we know from Theorem \ref{thm:bound} that $T^3_\bc(\cK)=0$ if $\bb\notin \{0,1,2\}^{\cV}$. In fact, one can show that $T^3_\bc(\cK)=0$ if $\bb\notin \{0,1\}^{\cV}$; we leave this to the reader. Coupled with Proposition \ref{prop:applicable}, we see that we can use the non-splitting complex to compute $T^3_\bc(\cK)$ for any weight $\bc$ with $T^3_\bc(\cK)\neq 0$.
\end{remark}

\section{Weight bounds for $T^3$ and $T^4$}
In this section, we will further improve on the bounds from \S\ref{sec:degree} on the weights in which $T^3$ and $T^4$ are supported in the special case of flag complexes.
We first consider the case of $T^3$.

\begin{proposition}\label{prop:t3bound}
Let $\cK$ be a flag complex. Given $\ba,\bb\in\ZZ_{\geq0}^\cV$ with disjoint supports, if $|\bb|>3$ then $T_{\ba-\bb}^3(\cK)=0$.
\end{proposition}
\begin{proof}
    Let $\bc=\ba-\bb$. By the localization formula (Theorem~\ref{thm:localizationformula}), we may assume without loss of generality that $\bc=-\bb$. We use Theorem~\ref{thm.computingTviahomology} to compute $T^3$.
	By Theorem~\ref{thm:bound} we may assume that $\bb\in\{0,1\}^\cV$.
	By Propositions~\ref{prop:nonsplitting} and~\ref{prop:applicable}, $H_3(W_\bullet^\bc)=H_3(\widetilde W_\bullet^\bc)$.

    If $|\bb|>3$, then $\widetilde{W}_3^\bc=0$. Indeed, every primitive string $s$ in $W_3^\bc$ has a splitting by Remark \ref{rem:splittings}.
    It follows that $T_\bc^3(\cK)=H_3(W_\bullet^\bc)=H_3(\widetilde W_\bullet^\bc)=0$.
\end{proof}

We now move on to $T^4$.

\begin{lemma}\label{lemma:t4splitting}
    Let $\cK$ be a flag complex, $\bb\in\ZZ_{\geq 0}^\cV$.
For each of the following statements, let $f_1,f_2$ or $f_1,f_2,f_3$ be distinct elements from the set of minimal generators of $I_\cK$.
    \begin{enumerate}
        \item\label{i1}	If $\bb$ has type $(2,2)$ or $(2)$, then for any generator $[\{f_1f_2\}\{f_1\}]\in W_4^{-\bb}$, we have $y=[\{f_1f_2\}\{f_1f_2\}]\in W_5^{-\bb}$. Moreover,
		\[ \overline \partial(y)={2[\{f_1f_2\}\{f_1\}]-}2\underline{[\{f_1f_2\}\{f_2\}]}. \]
		\item \label{i0} For any generator $[\{f_1\}\{f_2\}\{f_3\}]\in W_4^{-\bb}$, we have $y=[[\{f_1\}\{f_2\}]\{f_3\}]\in W_5^{-\bb}$. Moreover, 
			\[
\overline\partial(y)=[\{f_1f_2\}\{f_3\}]-\underline{[\{f_1\}\{f_2\}\{f_3\}]}.
			\]
		\item\label{i2}	If $\bb$ has type $(2,2)$ or $(2)$, then for any generator $[\{f_1f_2\}\{f_3\}]\in W_4^{-\bb}$, we have $y=[\{f_1f_2f_3\}\{f_3\}]\in W_5^{-\bb}$.
			Moreover,
			\[
	\overline \partial(y)={-[\{f_2f_3\}\{f_3\}]+[\{f_1f_3\}\{f_3\}]-}\underline{[\{f_1f_2\}\{f_3\}]}.
\]
\item\label{i4} If $\bb$ has type $(2,1)$ then for any generator  $[\{f_1f_2\}\{f_3\}]\in W_4^{-\bb}$,
			we have either $[\{f_1f_2f_3\}\{f_3\}]\in W_5^{-\bb}$ or $[\{f_1f_2\}\{f_1f_3\}]\in W_5^{-\bb}$. If additionally $[\{f_1f_3\}\{f_2\}]\in W_4^{-\bb}$, then either $[\{f_1f_2f_3\}\{f_3\}]$ or $[\{f_1f_2f_3\}\{f_2\}]$ belong to $W_5^{-\bb}$.
Moreover,
\begin{align*}
\overline\partial([\{f_1f_2f_3\}\{f_3\}])&={-[\{f_2f_3\}\{f_3\}]+[\{f_1f_3\}\{f_3\}]-}\underline{[\{f_1f_2\}\{f_3\}]}\\
\overline\partial([\{f_1f_2f_3\}\{f_2\}])&={-[\{f_2f_3\}\{f_2\}]+\underline{[\{f_1f_3\}\{f_2\}]}-}{[\{f_1f_2\}\{f_2\}]}\\
\overline\partial([\{f_1f_2\}\{f_1f_3\}])&=\underline{-[\{f_1f_3\}\{f_2\}]}+[\{f_1f_3\}\{f_1\}]-\underline{[\{f_1f_2\}\{f_3\}]}+[\{f_1f_2\}\{f_1\}].
\end{align*}

		\item\label{i5}  If $\bb\notin \{0,1\}^\cV$, $\bb\in \{0,1,2\}^\cV$, $|\bb|=5$, and \[[\{f_1f_2\}\{f_3\}],[\{f_2f_3\}\{f_1\}],[\{f_3f_1\}\{f_2\}]\in W_4^{-\bb},\]
	then we have $y=[\{f_1f_2f_3\}\{f_3\}]\in W_5^{-\bb}$. Moreover,
			\[
	\overline \partial(y)={-[\{f_2f_3\}\{f_3\}]+[\{f_1f_3\}\{f_3\}]-}\underline{[\{f_1f_2\}\{f_3\}]}.
\]
\end{enumerate}
    The underlined generator will be used in the proof of Proposition \ref{prop:t4bound}.
\end{lemma}
\begin{proof}
    In each of the claims, the images under $\overline\partial$ follow immediately from Theorem~\ref{thm.stringsdifferential}. In each case, we now show that the given hypotheses imply the existence of certain elements of $W_5^{-\bb}$.

    For the claims \eqref{i1}, \eqref{i0}, \eqref{i2}, we always have that the support of  $\wt(y)-\bb$ is the same as the support of $\wt(z)-\bb$, where $z\in W_4^{-\bb}$ was the given generator. Indeed, for \eqref{i0} this is immediate since $\wt(y)=\wt(z)$. For the other two cases,  in Figures \ref{fig:graphs1} and \ref{fig:graphs2}, we picture the possible configurations of $\bb$, $f_1$, $f_2$, (and $f_3$) satisfying the hypotheses in item \eqref{i1} and \eqref{i2}. We depict the subgraph of $G(I_\cK)$ corresponding to $f_1$, $f_2$ (and $f_3$). The support of $\bb$ is marked in gray, and the coordinate entries are labeled in the figure.
    The claims on the support of the weights now follows.
    But since $\wt(y)-\bb$ and $\wt(z)-\bb$ have the same support and $z\in W_4^{-\bb}$, we have $y\in W_5^{-\bb}$ as desired.

    We now consider claims \eqref{i4} and \eqref{i5}:
    \begin{enumerate}\setcounter{enumi}{3}
        \item Let $v,w\in \cV$ be respectively the vertices with $\bb_v=2$, $\bb_w=1$. Hence the only possible configuration for which $[\{f_1f_2f_3\}\{f_3\}]\notin W_5^{-\bb}$ requires $x_v$ dividing $f_1,f_2,f_3$ and $x_w$ dividing $f_3$, i.e. we are dealing with the case depicted in Figure~\ref{fig:graphsnew2}. In this case, it is straightforward to see that $[\{f_1f_2\}\{f_1f_3\}]\in W_5^{-\bb}$. If additionally we assume $[\{f_1f_3\}\{f_2\}]\in W_4^{-\bb}$ then again restricting to the case depicted in Figure~\ref{fig:graphsnew2}, it is straightforward to see that $[\{f_1f_2f_3\}\{f_2\}]\in W_5^{-\bb}$.
        \item Let $v\in \cV$ be a vertex with $\bb_v=2$. Then by hypothesis $x_v$ divides $f_1$, $f_2$, and $f_3$. In particular, we deduce that $\bb$ is of type $(2,1,1,1)$, and the only possible configuration is pictured in Figure~\ref{fig:graphs4}. Notice that in this case
        \[ \bb=\wt([\{f_1f_2\}\{f_3\}])=\wt([\{f_2f_3\}\{f_1\}])=\wt([\{f_3f_1\}\{f_2\}])\; , \]
        and $|\wt([\{f_1f_2f_3\}\{f_3\}])-\bb|=1$, whence the statement.\qedhere
    \end{enumerate}
\end{proof}

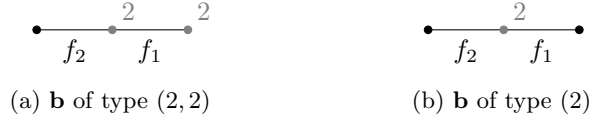
\begin{figure}
\begin{subfigure}{.4\textwidth}
	\begin{tikzpicture}
\draw (0,0) -- (1,0) -- (2,0);
		\draw[white,fill] (-1.5,0) circle [radius=0.05]; 
		\draw[fill] (0,0) circle [radius=0.05]; 
		\draw[fill,gray] (1,0) circle [radius=0.05] node[above right,gray] {$2$};
		\draw[fill,gray] (2,0) circle [radius=0.05] node[above right,gray] {$2$};
\draw[] (1.5,0) node[below] {$f_1$}; 
\draw[] (.5,0) node[below] {$f_2$}; 
	\end{tikzpicture}
	\caption{$\bb$ of type $(2,2)$}
\end{subfigure}
\begin{subfigure}{.4\textwidth}
	\begin{tikzpicture}
\draw (0,0) -- (1,0) -- (2,0);
		\draw[white,fill] (-1.5,0) circle [radius=0.05]; 
		\draw[fill] (0,0) circle [radius=0.05]; 
		\draw[fill] (2,0) circle [radius=0.05]; 
		\draw[fill,gray] (1,0) circle [radius=0.05] node[above right,gray] {$2$};
\draw[] (1.5,0) node[below] {$f_1$}; 
\draw[] (.5,0) node[below] {$f_2$}; 
	\end{tikzpicture}
	\caption{$\bb$ of type $(2)$}
\end{subfigure}
\caption{Cases for Lemma \ref{lemma:t4splitting}\eqref{i1}}\label{fig:graphs1}
\end{figure}

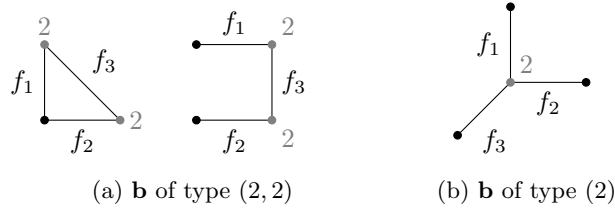
\begin{figure}
\begin{subfigure}{.4\textwidth}
	\begin{tikzpicture}
\draw (0,0) -- (1,0) -- (0,1) -- (0,0);
		\draw[fill] (0,0) circle [radius=0.05]; 
		\draw[fill,gray] (1,0) circle [radius=0.05] node[right,gray] {$2$};
		\draw[fill,gray] (0,1) circle [radius=0.05] node[above,gray] {$2$};
\draw[] (.5,.5) node[above right] {$f_3$}; 
\draw[] (0,.5) node[left] {$f_1$}; 
\draw[] (.5,0) node[below] {$f_2$}; 
	\end{tikzpicture}
	\hspace{.3cm}
	\begin{tikzpicture}
\draw (0,0) -- (1,0) -- (1,1) -- (0,1);
		\draw[fill] (0,0) circle [radius=0.05]; 
		\draw[fill] (0,1) circle [radius=0.05]; 
		\draw[fill,gray] (1,0) circle [radius=0.05] node[below right,gray] {$2$};
		\draw[fill,gray] (1,1) circle [radius=0.05] node[above right,gray] {$2$};
\draw[] (1,.5) node[right] {$f_3$}; 
\draw[] (.5,1) node[above] {$f_1$}; 
\draw[] (.5,0) node[below] {$f_2$}; 
	\end{tikzpicture}
	\caption{$\bb$ of type $(2,2)$}
\end{subfigure}
\begin{subfigure}{.28\textwidth}
	\begin{tikzpicture}
\draw (0,0) -- (1,0);
\draw (0,1) -- (0,0);
\draw (-.7,-.7) -- (0,0);
		\draw[white,fill] (-1.5,0) circle [radius=0.05]; 
		\draw[fill] (1,0) circle [radius=0.05]; 
		\draw[fill] (0,1) circle [radius=0.05]; 
		\draw[fill] (-.7,-.7) circle [radius=0.05]; 
		\draw[fill,gray] (0,0) circle [radius=0.05] node[above right,gray] {$2$};
\draw[] (-.5,-.5) node[below right] {$f_3$}; 
\draw[] (0,.5) node[left] {$f_1$}; 
\draw[] (.5,0) node[below] {$f_2$}; 
	\end{tikzpicture}
	\caption{$\bb$ of type $(2)$}
\end{subfigure}
	\caption{Cases for Lemma \ref{lemma:t4splitting} \eqref{i2}}\label{fig:graphs2}
\end{figure}

%\draw[fill] (0,1) circle [radius=0.05]; 
%\draw[fill,gray] (1,0) circle [radius=0.05] node[below right,gray] {$2$}
\begin{figure}
	\begin{subfigure}{.4\textwidth}
	\begin{tikzpicture}
\draw (0,0) -- (1,0);
\draw (0,1) -- (0,0);
\draw (-.7,-.7) -- (0,0);
		\draw[white,fill] (-1.5,0) circle [radius=0.05]; 
		\draw[fill] (1,0) circle [radius=0.05] ; 
		\draw[fill] (0,1) circle [radius=0.05]; 
		\draw[fill,gray] (-.7,-.7) circle [radius=0.05] node [left,gray] {$1$}; 
		\draw[fill,gray] (0,0) circle [radius=0.05] node[above right,gray] {$2$};
\draw[] (-.5,-.5) node[below right] {$f_3$}; 
\draw[] (0,.5) node[left] {$f_1$}; 
\draw[] (.5,0) node[below] {$f_2$}; 
	\end{tikzpicture}
\caption{Case \eqref{i4}}\label{fig:graphsnew2}
\end{subfigure}
\begin{subfigure}{.4\textwidth}
	\begin{tikzpicture}
\draw (0,0) -- (1,0);
\draw (0,1) -- (0,0);
\draw (-.7,-.7) -- (0,0);
		\draw[white,fill] (-1.5,0) circle [radius=0.05]; 
		\draw[fill,gray] (1,0) circle [radius=0.05] node[above right,gray] {$1$}; 
		\draw[fill,gray] (0,1) circle [radius=0.05] node[above right,gray] {$1$}; 
		\draw[fill,gray] (-.7,-.7) circle [radius=0.05] node [left,gray] {$1$}; 
		\draw[fill,gray] (0,0) circle [radius=0.05] node[above right,gray] {$2$};
\draw[] (-.5,-.5) node[below right] {$f_3$}; 
\draw[] (0,.5) node[left] {$f_1$}; 
\draw[] (.5,0) node[below] {$f_2$}; 
	\end{tikzpicture}
\caption{Case \eqref{i5}}\label{fig:graphs4}
\end{subfigure}
\caption{More cases for Lemma \ref{lemma:t4splitting}}
\end{figure}
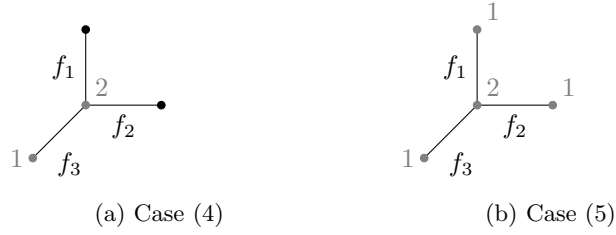

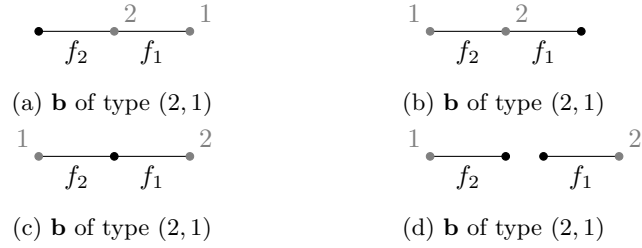
\begin{figure}
    \begin{subfigure}{.4\textwidth}
    	\begin{tikzpicture}
            \draw (0,0) -- (1,0) -- (2,0);
    		\draw[white,fill] (-1.5,0) circle [radius=0.05]; 
    		\draw[fill] (0,0) circle [radius=0.05]; 
    		\draw[fill,gray] (1,0) circle [radius=0.05] node[above right,gray] {$2$};
    		\draw[fill,gray] (2,0) circle [radius=0.05] node[above right,gray] {$1$};
            \draw[] (1.5,0) node[below] {$f_1$}; 
            \draw[] (.5,0) node[below] {$f_2$}; 
    	\end{tikzpicture}
    	\caption{$\bb$ of type $(2,1)$}
    \end{subfigure}
    \begin{subfigure}{.4\textwidth}
    	\begin{tikzpicture}
            \draw (0,0) -- (1,0) -- (2,0);
    		\draw[white,fill] (-1.5,0) circle [radius=0.05]; 
    		\draw[fill,gray] (0,0) circle [radius=0.05] node[above left,gray] {$1$}; 
    		\draw[fill,gray] (1,0) circle [radius=0.05] node[above right,gray] {$2$};
    		\draw[fill] (2,0) circle [radius=0.05];
            \draw[] (1.5,0) node[below] {$f_1$}; 
            \draw[] (.5,0) node[below] {$f_2$}; 
    	\end{tikzpicture}
    	\caption{$\bb$ of type $(2,1)$}
    \end{subfigure}
    \begin{subfigure}{.4\textwidth}
    	\begin{tikzpicture}
            \draw (0,0) -- (1,0) -- (2,0);
    		\draw[white,fill] (-1.5,0) circle [radius=0.05]; 
    		\draw[fill,gray] (0,0) circle [radius=0.05] node[above left,gray] {$1$}; 
    		\draw[fill] (1,0) circle [radius=0.05] ;
    		\draw[fill,gray] (2,0) circle [radius=0.05] node[above right,gray] {$2$};
            \draw[] (1.5,0) node[below] {$f_1$}; 
            \draw[] (.5,0) node[below] {$f_2$}; 
    	\end{tikzpicture}
    	\caption{$\bb$ of type $(2,1)$}
    \end{subfigure}
    \begin{subfigure}{.4\textwidth}
    	\begin{tikzpicture}
            \draw (0.5,0) -- (1.5,0)  (2,0) -- (3,0);
    		\draw[white,fill] (-1,0) circle [radius=0.05]; 
            \draw[fill,gray] (0.5,0) circle [radius=0.05] node[above left,gray] {$1$}; 
    		\draw[fill] (1.5,0) circle [radius=0.05]; 
    		\draw[fill] (2,0) circle [radius=0.05] ;
    		\draw[fill,gray] (3,0) circle [radius=0.05] node[above right,gray] {$2$};
            \draw[] (2.5,0) node[below] {$f_1$}; 
            \draw[] (1,0) node[below] {$f_2$}; 
    	\end{tikzpicture}
    	\caption{$\bb$ of type $(2,1)$}
    \end{subfigure}
\caption{Additional cases for Proposition~\ref{prop:t4bound}}\label{fig:graphsnew}
\end{figure}

\begin{table}
	\begin{tabular}{l l}
		Case & Generators for $U\subseteq W_4^{-\bb}$\\
		\hline
		$\bb$ of type $(2,2)$ & 0\\
        \\
		$|\bb|=5$, $\bb\notin\{0,1\}^\cV$ & 0\\
		\\
		$\bb$ of type $(2)$  & $[\{f_1f_2\}\{f_1\}]\in W_4^{-\bb}$ such that $[\{f_1f_2\}\{f_2\}]\in W_4^{-\bb}$\\
		& \hspace{3cm} and $\wt(f_1)>_\textrm{lex} \wt(f_2)$\\
		\\
        $\bb$ of type $(2,1)$  & $[\{f_1f_2\}\{f_1\}]\in W_4^{-\bb}$ such that $[\{f_1f_2\}\{f_2\}]\in W_4^{-\bb}$\\
		& \hspace{3cm} and $\wt(f_1)>_\textrm{lex} \wt(f_2)$; and\\
          & $[\{f_1f_2\}\{f_1\}]\in W_4^{-\bb}$ such that $[\{f_1f_2\}\{f_2\}]\notin W_4^{-\bb}$\\
	  & \hspace{3cm} and $[\{f_1\}\{f_2\}]\in W_3^{-\bb}$\\
		\\
		$|\bb|=5$, $\bb\in\{0,1\}^\cV$ & $[\{f_1\}\{f_2\}\{f_3\}]\in W_4^{-\bb}$\\
	\end{tabular}

	\caption{Subspaces $U\subseteq W_4^{-\bb}$ in the proof of Proposition \ref{prop:t4bound}}\label{table:U}
\end{table}

We will apply the above lemma to show the following.

\begin{proposition}\label{prop:t4bound}
Let $\cK$ be a flag complex. Let  $\ba,\bb\in\ZZ_{\geq0}^\cV$ with disjoint supports. Assume that one of the following conditions holds:
\[ |\bb|>4 \quad \text{ or } \quad \bb \text{ is of type } (2,2)\, ,\, (2,1)\,, \text{or } (2) \; . \]
Then $T^4_{\ba-\bb}(\cK)=0$.
\end{proposition}
\begin{proof}
	By the localization formula (Theorem \ref{thm:localizationformula}) we may assume without loss of generality that $\ba=0$ in $\ZZ^\cV$. Moreover, by Corollary~\ref{cor:koszul} we can restrict our attention to $\vert\bb\vert\leq5$, and by Theorem~\ref{thm:bound} we can also assume $\bb\in\{0,1,2\}^\cV$.

    Let us first describe the general strategy to obtain the desired vanishing.
    In each case, we consider the subspace $U\subseteq W_4^{-\bb}$ with generators as listed in Table~\ref{table:U}. The condition $\wt(f_1)>_\textrm{lex} \wt(f_2)$ means that $\wt(f_1)$ is larger than $\wt(f_2)$ in lexicographic order after having fixed some order of $\cV$.
    
    Notice that the restriction of $\partial_4^{-\bb}$ to $U$ is injective. Indeed,
for $\bb$ of type $(2)$ or $(2,1)$, 
if $[\{f_1f_2\}\{f_1\}]$ and $[\{f_1f_2\}\{f_2\}]$ are both in $W_4^{-\bb}$, then we are in the situation depicted in 
Figure~\ref{fig:graphs1}(B), Figure~\ref{fig:graphsnew}(A), or Figure~\ref{fig:graphsnew}(B). In all cases, $[\{f_2\}\{f_1\}]\in W^{-\bb}_3$ and it follows that $\partial_4^{-\bb}([\{f_1f_2\}\{f_1\}])=[\{f_2\}\{f_1\}]$.
If instead  $\vert\bb\vert=5$ and $\bb\in\{0,1\}^\cV$ we have 
that $\wt([\{f_1\}\{f_2\}\{f_3\}])-\bb$ and $\wt(\{f_1f_2f_3\})-\bb$ have the same support
        so that $\{f_1f_2f_3\}\in W^{-\bb}_3$ and $\partial_4^{-\bb}\vert_U$ is injective.
   
	In the remainder of the proof, we will show that in each case
    $W_4^{-\bb} = U+\mathsf{im}(\partial_5^{-\bb})$. This will imply the vanishing of $H_4(W_\bullet^{-\bb})$, and hence $T^4_{-\bb}(\cK)=0$ by Theorem~\ref{thm.computingTviahomology}.

\subsubsection*{Case I: $\bb$ of type $(2,2)$, $(2,1)$, or $(2)$}
        In this case, the only possible generators in $W_4^{-\bb}$ are of the form $[\{f_1\}\{f_2\}\{f_3\}]$ and $[\{f_1f_2\}\{f_3\}]$, with $f_1,f_2,f_3$ not necessarily distinct.
        Using Lemma~\ref{lemma:t4splitting}, we replace generators modulo boundaries: first those of the type $[\{f_1\}\{f_2\}\{f_3\}]$ (item \eqref{i0}), then of the type $[\{f_1f_2\}\{f_3\}]$ with $f_1,f_2,f_3$ distinct (items \eqref{i2} or \eqref{i4}). (In Lemma~\ref{lemma:t4splitting}, the terms being replaced have been underlined.) This reduces us to only deal with generators of type $[\{f_1f_2\}\{f_1\}]$.

        If $\bb$ is of type $(2,2)$ or $(2)$, then Lemma~\ref{lemma:t4splitting}\eqref{i1} allows us to replace $[\{f_1f_2\}\{f_1\}]$ by $[\{f_1f_2\}\{f_2\}]$ modulo boundaries. Hence $W_4^{-\bb} = U+\mathsf{im}(\partial_5^{-\bb})$  (see Figure~\ref{fig:graphs1}). 

	If instead $\bb$ is of type $(2,1)$, then we are in one of the situations depicted in Figure~\ref{fig:graphsnew}.
	In case (A) similar to Lemma~\ref{lemma:t4splitting}\eqref{i1} we have
	$[\{f_1f_2\}\{f_1f_2\}]\in W_5^{-\bb}$ so may replace $[\{f_1f_2\}\{f_1\}]$ by $[\{f_1f_2\}\{f_2\}]$ modulo boundaries and arrive at an element of $U$.
	In case (B), $[\{f_1f_2\}\{f_1\}]\in U$ since $[\{f_1f_2\}\{f_2\}]\notin W_4^{-\bb}$ but $[\{f_1\}\{f_2\}]\in W_3^{-\bb}$.  
        Finally, in cases (C) and (D) we have  $y=[[\{f_1\}\{f_2\}]\{f_1\}]\in W_5^{-\bb}$
       and   $\overline\partial(y)=[\{f_1f_2\}\{f_1\}]$.
        
	\subsubsection*{Case II: $|\bb|=5$}
	We now turn our attention to the case $\vert\bb\vert=5$.
        Suppose first that $\bb\notin \{0,1\}^\cV$, so that primitive strings cannot belong to $W_4^{-\bb}$. Hence the only possible strings in $W_4^{-\bb}$ are of the form $[\{f_1\}\{f_2\}\{f_3\}]$ and $[\{f_1f_2\}\{f_3\}]$ for $f_1,f_2,f_3$ not necessarily distinct.
	Since $|\wt [[\{f_1\}\{f_2\}]\{f_3\}]|=6$ and $|\bb|=5$, we have $y=[[\{f_1\}\{f_2\}]\{f_3\}]\in W_5^{-\bb}$ whenever $[\{f_1f_2\}\{f_3\}]\in W_4^{-\bb}$. Since 
	\[\overline\partial(y)=[\{f_1f_2\}\{f_3\}]-[\{f_1\}\{f_2\}\{f_3\}] \]
we may replace any generator of the form $[\{f_1f_2\}\{f_3\}]\in W_4^{-\bb}$ with $[\{f_1\}\{f_2\}\{f_3\}]]$ (or $0$).

But for essentially the same reason,
$[\{f_1\}\{f_2\}\{f_3\}]\in W_4^{-\bb}$ is a boundary unless
        \[ [\{f_1f_2\}\{f_3\}],[\{f_2f_3\}\{f_1\}],[\{f_3f_1\}\{f_2\}]\in W_4^{-\bb} \;,\]
        and we may use Lemma~\ref{lemma:t4splitting}\eqref{i5} to restrict to the case of strings of the form $[\{f_1f_2\}\{f_1\}]\in W_4^{-\bb}$. But we have already seen that these are boundaries.
 
	Finally, suppose that $\vert\bb\vert=5$ and $\bb\in\{0,1\}^\cV$; i.e. $\bb$ is of type $(1,1,1,1,1)$. 
	By Remark~\ref{rem:splittings} any primitive generator has a splitting, so we only need to consider quasiprimitive generators.
	Since $\bb$ is of type $(1,1,1,1,1)$, these must be of the form $[\{f_1f_2\}\{f_3\}]$ or $[\{f_1\}\{f_2\}\{f_3\}]$ for distinct $f_1,f_2,f_3$.
	Now we apply Lemma~\ref{lemma:t4splitting}(\ref{i0}) to replace generators of the form $[\{f_1f_2\}\{f_3\}]$ with elements in $U$.
\end{proof}

\section{Topological interpretations}\label{sec:topological}
\subsection{Preliminaries}

Consider an arbitrary simplicial complex $\cK$.
Fix some subset $b$ of $\cV(\cK)$. We introduce and recall some notation connected to $b$.

Recall from \S\ref{sec:overview} that we defined a subcomplex of $\cK$
\[
	\cK_b=\left\{f\in \cK \ \Big|\ f\cap \big(b \cup \cV(\bigcap_{v\in b}\link (v,\cK))\big)=\emptyset\right\}.
\]
We will also need the related complex:
\[
	\cK_b'=\left\{\coprod_{v\in b} f_v\ \Big|\ f_v\subseteq \cV(\cK)\setminus \big(b \cup \cV(\link (v,\cK))\big)\ \forall v\in b,\ \textrm{and}\ \bigcup_{v\in b}f_v\in \cK\right\}.
\]
Here $\coprod$ denotes the disjoint union. This is a simplicial complex whose vertex set is contained in the disjoint union of $\#b$ copies of $\cV(\cK)$. 
There is a natural simplicial map $\phi_b:\cK_b'\to \cK_b$ via $\coprod_{v\in b} f_v\mapsto \bigcup_{v\in b}f_v$.
In particular, for any vertex $w$ of $\cK_b'$, we obtain a corresponding vertex $\phi_b(w)$ of $\cK_b$. For any vertex $w$ of $\cK_b'$, we will denote by $\nu(w)$ the element of $b$ corresponding to the copy of $\cV(\cK)$ to which $w$ belongs. Thus, any vertex $w$ of $\cK_b'$ is determined by $\phi_b(w)$ and $\nu(w)$.

Suppose now that $\cK$ is a flag complex, and $b\subseteq \cV(\cK)$ with $|b|\geq 2$. We let $\cE_b$ be the set of all edges of $G(I_\cK)$ contained in $b$, or equivalently, the set of all two-element subsets of $b$ not contained in $\cK$.

\begin{lemma}\label{lemma:homotopic}
The map $\phi_b$ is a homotopy equivalence.
\end{lemma}
\begin{proof}
	The preimage under $\phi_b$ of any simplex is a simplex.
	Indeed, consider an arbitrary simplex $f$ of $\cK_b$. For $v\in b$, let $f_v$ be the set of all elements of $f$ disjoint from $b$ and $\link(v,\cK)$.
		The set of vertices in the preimage of $f$ is precisely the disjoint union of the $f_v$. But the disjoint union of the $f_v$ belongs to $\cK_b'$ since $\bigcup_{v\in b} f_v=f\in\cK$.
		
		The claim now follows from e.g.~\cite[Exercise 4.L.2]{hatcher}.
\end{proof}

We also need some notation related to simplicial homology. See e.g.~\cite{hatcher} for details.
 Let $X$ be a finite simplicial complex.
Recall that for $k\in\ZZ_{\geq 0}$, an \emph{oriented $k$-simplex} of $X$ is a sequence $(u_0,\ldots,u_k)$ with the set $\{u_0,\ldots,u_k\}$ a face of $X$.
A $k$-chain with coefficients in $\KK$ is a formal (finite) $\KK$-linear combination of oriented $k$-simplices, where we make the identification that 
for any permutation $\tau$ of $\{0,\ldots,k\}$, 
\[
	(u_{\tau(0)},\ldots,u_{\tau(k)})=\sign(\tau) (u_0,\ldots,u_k).
\]

We let ${C}_\bullet(X)$ denote the complex of simplicial chains in $X$ with coefficients in $\KK$;
the differential is defined via
\[
	(u_0,\ldots,u_k)\mapsto \sum_{j=0}^k(-1)^j (u_0,\ldots,u_{j-1},u_{j+1},\ldots,u_k).
\]
Let $\widetilde C_\bullet(X)$ denote the augmented version of this complex, that is, $C(X)_{-1}\cong \KK$ with the obvious map $C(X)_0\to C(X)_{-1}$ taking any oriented $0$-simplex to $1\in\KK$.
Likewise, for a subcomplex $X'$ of $X$, we let ${C}_\bullet(X,X')$ denote the complex of relative simplicial chains for $X'$ in $X$ with coefficients in $\KK$, that is, the quotient of $C_\bullet(X)$ by $C_\bullet(X')$.

\subsection{Type $(1)$}\label{sec:b1}
We now prove Theorem~\ref{thm:b1}. Since $\bb$ is of type $(1)$, the only generators in $W_\bullet^{-\bb}$ are primitive. 
Let 
\begin{align*}
Y=\{f\in\cK\ |\ f\cup b\notin \cK\}=\cK\setminus(\Delta(b)*\link(b,\cK)).
\end{align*}
Following \cite[\S2]{ac1}, for $j\geq 0$ set 
\[
	Y^{(j)}=\{(f_0,\ldots,f_j)\in Y^{j+1}\ |\ f_0\cup\ldots\cup f_j\in Y\}.
\]
From loc.~cit., $Y$ determines a cochain complex $K^\bullet(Y)$: $K^j(Y)$ is the set of alternating maps $Y^{(j)}\to \KK$, and the differential is the obvious Koszul-like differential. 

There is a bijection between strings $s$ representing elements of $W_{j+1}^{-\bb}$ and elements of $Y^{(j)}$ sending $s=\{g_0\ldots g_j\}$ with $\overline s\in W_{j+1}^{-\bb}$ to $(\wt(g_0),\ldots,\wt(g_j))$.
It follows that the complex $\Hom(W_\bullet^{-\bb},\KK)$ in cohomological degrees $\geq 1$ is isomorphic to the complex $K^\bullet(Y)[-1]$.

Following the notation of \cite[\S3]{ac1}, $\langle Y\rangle $ is
the complement of  $(\Delta(b)*\link(b,\cK))$ in the topological realization of $\cK$.
For $i\geq 2$ we then have 
\[
	T^i_{-\bb}(\cK)\cong H^i(\Hom(W_\bullet^{-\bb},\KK) \cong H^{i-1}(K^\bullet(Y))\cong H^{i-1}(\langle Y \rangle ,\KK)
\]
with the first isomorphism following from Theorem \ref{thm.computingTviahomology}, the second from the above discussion, and the third isomorphism following from \cite[Lemmas 7 and 8]{ac1}.
For $i=1$, one also has $T^1_{-\bb}(\cK)\cong \widetilde H^{0}(\langle Y \rangle ,\KK)$ by \cite[Theorem 9]{ac1} (in the notation of loc.~cit. $\langle N_{a-b}\rangle=\langle Y \rangle$ and $\langle \widetilde N_{a-b}\rangle=\emptyset$).

The claim of the theorem  now follows by observing that (the topological realization of) $\cK_b$ is a deformation retract of
$\langle Y\rangle$. Indeed, in this instance
\[
	\cK_b=\{f\in \cK\ |\ f\cap \cV(\Delta(b)*\link(b,\cK))=\emptyset\}
\]
and for any simplex $f\in \cK$, we may retract the complement $f^\circ$ of $\Delta(v)*\link(v,\cK)$ in the topological realization of $f$ to the topological realization of $f\setminus \cV(\Delta(v)*\link(v,\cK))$. This gives a retraction of $\langle Y\rangle$ to $\cK_b$.

\qed

\subsection{Type $(1,1)$}\label{sec:b2}
We now assume that $\cK$ is a flag complex and we prove Theorem~\ref{thm:b2}.
We first note by Proposition~\ref{prop:applicable} that we may apply Proposition~\ref{prop:nonsplitting} to conclude that $W_\bullet^{-\bb}$ and $\widetilde W_\bullet^{-\bb}$ are quasiisomorphic. Thus, by Theorem~\ref{thm.computingTviahomology} it will suffice to show that $H_i(\widetilde W_\bullet^{-\bb})$ is $0$ (if $b\in\cK)$ or isomorphic to $\widetilde{H}_{i-2}(\cK_b,\KK)$ (if $b\notin\cK$).

Next, for any primitive generator $y\in W_j^{-\bb}$, let $G(y)$ be the subgraph of $G(I_\cK)$ as in \S\ref{sec:flag} and Remark~\ref{rem:splittings}.
Then every edge of $G(y)$ must be adjacent to an element of $b$, every element of $b$ must be a vertex of $G(y)$, and the vertices of $G(y)$ outside of $b$ must form a simplex in $\cK$. 
Moreover, again by Remark~\ref{rem:splittings}, any primitive generator $y$ is non-split if and only if $b$ is an edge of $G(y)$.

If $b\in \cK$, then from the previous paragraph it follows that $\widetilde{W}_\bullet^{-\bb}=0$ and we are done. We thus instead assume that $b\notin\cK$.
By the previous paragraph, there is an isomorphism $\widetilde{C}_\bullet(\cK_b',\KK)[-2]\cong \widetilde W_\bullet^{-\bb}$
sending the oriented simplex 
$(u_2,\ldots,u_j)$
to 
$\{f_2\ldots f_jf_{j+1}\}$
where
\[f_{j+1}=x^{\bb} \quad \text{ and } \quad f_k=x_{\phi_b(u_k)}x_{\nu(u_k)},\qquad \text{for } k=2,\ldots,j.\]
Note that this is an isomorphism because $u_k$ is uniquely determined by the pair of vertices $(\phi_b(u_k),\nu(u_k))$.
By Lemma~\ref{lemma:homotopic} we then have for $i\geq 1$ that 
\[
	H_i(\widetilde W_\bullet^{-\bb}) \cong \widetilde{H}_{i-2}(\cK_b',\KK)\cong \widetilde{H}_{i-2}(\cK_b,\KK)
\]
as desired.
\qed

\begin{remark}\label{rem:nonflagb2}
	As in Remark \ref{rem:nonflag}, we may extend Theorem~\ref{thm:b2} slightly beyond the case of flag complexes. Suppose that $\cK$ is a simplicial complex with $\cV=\cV(\cK)$, $\bb\in \ZZ_{\geq 0}^\cV$ is of type $(1,1)$, and every minimal non-face of $\cK$ of dimension at least $2$ is disjoint from $b$. Then by Remark \ref{rem:nonflag}, the above proof goes through verbatim, and the conclusions of Theorem~\ref{thm:b2} hold in this case as well.
\end{remark}
\subsection{Type $(1,1,1)$}\label{sec:b3}
Again assuming $\cK$ is a flag complex, we now prove Theorem~\ref{thm:b3}.
By Proposition~\ref{prop:applicable} we may apply Proposition~\ref{prop:nonsplitting} to conclude that $H_i(W_\bullet^{-\bb})\cong H_i(\widetilde W_\bullet^{-\bb})$ for $i=3,4$. Thus, by Theorem \ref{thm.computingTviahomology} it will suffice to show that $H_i(\widetilde W_\bullet^{-\bb})$ 
has the desired form for $i=3,4$.

As in \S\ref{sec:b2}, for any primitive generator $y\in W_j^{-\bb}$, let $G(y)$ be the subgraph of $G(I_\cK)$ as in \S\ref{sec:flag} and Remark \ref{rem:splittings}.
Then every edge of $G(y)$ must be adjacent to an element of $b$, every element of $b$ must be a vertex of $G(y)$, and the vertices of $G(y)$ outside of $b$ must form a simplex in $\cK$. Again by Remark~\ref{rem:splittings}, any primitive generator $y$ admits a splitting unless the restriction of $G(y)$ to $b$ is connected.
It immediately follows that if $G(I_\cK)$ restricted to $b$ contains fewer than two edges, i.e. it is disconnected, then $\widetilde W_\bullet^{-\bb}=0$, and hence $T^i_{-\bb}(\cK)=0$. But this is equivalent to $\cK$ having more than one edge that is a subset of $b$.
Therefore, for the rest of the proof, we may thus assume that $G(I_\cK)$ restricted to $b$ contains at least two edges.

We now describe more precisely when generators $y\in W_j^{-\bb}$ admit a splitting.
By the discussion above, we can restrict our attention to generators with $G(y)$ containing at least two edges between the elements of $b$.
By Remark~\ref{rem:splittings}, if $G(y)$ contains three edges between the elements of $b$, then it has no splitting. On the other hand, if $G(y)$ contains precisely two edges $e_1,e_2$ between the elements of $b$, then $y$ admits a splitting if and only if the complement of $b$ in the set of vertices of $G(y)$ belongs to $\link(w_{\{e_1,e_2\}},\cK)$, where $w_{\{e_1,e_2\}}\in b$ is the unique vertex belonging to both $e_1$ and $e_2$.
Indeed, a partition $E_1,E_2$ as in Remark~\ref{rem:splittings} cannot have both $e_1,e_2$ in e.g.~$E_1$, since then any edge in $E_2$ has a vertex in common with $E_1$ (since it has a vertex in $b$, see Remark~\ref{rem:graphweight}). But then without loss of generality, we have $e_1\in E_1$, $e_2\in E_2$, and so the union of $w_{\{e_1,e_2\}}$ with the complement of $b$ in the set of vertices of $G(y)$ must be a face in $\cK$.

We now define a surjective map of complexes 
\begin{equation}\label{eqn:psi}
	\psi:C_\bullet(\Delta(\cE_b)*\cK_b')[-1]\to \widetilde W_\bullet^{-\bb}.\end{equation}
Indeed, 
for a vertex $u$ of $\Delta(\cE_b)*\cK_b'$, set
\[
g(u)=	\begin{cases}
	x_vx_w & u=\{v,w\}\in \cE_b\\
	x_{\phi_b(u_j)}x_{\nu(u)} & u\in \cV(\cK_b')\\
	\end{cases}.
\]
We then let $\psi$ be the map sending an oriented simplex $\sigma=(u_1,\ldots,u_j)\in C_{j-1}(\Delta(\cE_b)*\cK_b')$ to $(-1)^j$ times the class of the primitive generator $y=\{g(u_1)\cdots g(u_j)\}$ in $\widetilde W_j^{-\bb}$ if $y\in W_j^{-\bb}$, and to $0$ if $y\notin W_j^{-\bb}$.

We now describe the kernel of $\psi$. There are two reasons that the image of an oriented simplex $\sigma=(u_1,\ldots,u_j)$ might vanish:
the resulting generator $y=\{g(u_1)\cdots g(u_j)\}$ might not belong to $W_j^{-\bb}$, or it might have a splitting. By the above discussion on $G(y)$, we know that one of these will certainly happen if $\sigma$ is a simplex for the subcomplex $\Delta(\cE_b)^{(0)}*\cK_b'$, where $\Delta(\cE_b)^{(0)}$ is the $0$-skeleton of $\Delta(\cE_b)$.
On the other hand, if $\sigma$ is not a simplex for the subcomplex $\Delta(\cE_b)^{(0)}*\cK_b'$, then it is guaranteed that the generator $y$ belongs to $W_j^{-\bb}$, although it might still have a splitting.

Consider then an oriented simplex $\sigma=(u_1,\ldots,u_j)$ not supported on the complex $\Delta(\cE_b)^{(0)}*\cK_b'$. By our above characterization of the generators admitting splittings, we see that $\sigma$ lies in the kernel of $\psi$ if and only if
$\sigma$ is a chain supported on a subcomplex of the form \[\Delta(f)*\phi^{-1}_b(\link(w_f,\cK)\cap \cK_b)\subseteq \Delta(\cE_b)*\cK_b',\] where $f$ is an edge of $\Delta(\cE_b)$ and $w_f\in b$ is as above the unique vertex belonging to the two edges corresponding to the vertices of $f$.

Letting $\widehat \cK_b'$ be the union of these complexes $\Delta(f)*\phi^{-1}_b(\link(w_f,\cK)\cap\cK_b)$ with $\Delta(\cE_b)^{(0)}*\cK_b'$, we thus obtain that $\psi$ induces an isomorphism of complexes
\[
	C_\bullet(\Delta(\cE_b)*\cK_b',\widehat\cK_b')[-1]\cong \widetilde W_\bullet^{-\bb}. 
\]
Since $\Delta(\cE_b)*\cK_b'$ is contractible, from the long exact sequence of homology we obtain isomorphisms
\[
	H_{i-2}(\widehat \cK_b',\KK) \cong H_i(\widetilde W_\bullet^{-\bb}).
\]
Finally, the map $\phi_b$ induces a surjective simplicial map $\widehat\phi_b:\widehat \cK_b'\to \widehat \cK_b$. As in Lemma \ref{lemma:homotopic}, this again induces a homotopy equivalence of $\widehat \cK_b'$  and $\widehat \cK_b$ since the preimage of any simplex under $\widehat\phi_b$ is again a simplex. The claim of Theorem~\ref{thm:b3} now follows.
\qed

\vspace{.5cm}

\begin{remark}\label{rem:nonflagb3}
	As in Remark \ref{rem:nonflag}, we may extend Theorem~\ref{thm:b3} slightly beyond the case of flag complexes. Suppose that $\cK$ is a simplicial complex with $\cV=\cV(\cK)$, $\bb\in \ZZ_{\geq 0}^\cV$ is of type $(1,1,1)$, and every minimal non-face of $\cK$ of dimension at least $2$ is disjoint from $b$. Then by Remark \ref{rem:nonflag}, the above proof goes through verbatim in the case $i=3$, and the conclusions of Theorem~\ref{thm:b3} for $T^3$ hold in this case as well.
\end{remark}

The complex $\widehat\cK_b$ is not as nice to work with as $\cK_b$ since it is not a subcomplex of $\cK$. However, we have the following result.
Recall that for any topological space,
\[
	\widetilde H^{-1}(X,\KK):=\begin{cases}
0 & X\neq \emptyset;\\
\KK & X=\emptyset.
	\end{cases}
\]
\begin{cor}\label{cor:b3simplyconnected}
Let $\cK$ be a flag complex, and $\bb\in\ZZ_{\geq 0}^\cV$ be of type $(1,1,1)$. Suppose that at most one edge of $\cK$ is a subset of $b$. 
There there are isomorphisms
\[
	T^i_{-\bb}(\cK)\cong  \bigoplus_{\substack{f\in\Delta(\cE_b)\\ \dim f=1}} \widetilde H^{i-4} (\link(\{w_f\},\cK)\cap \cK_b,\KK)
\]
for $i=3$ if $\cK_b$ is connected, and for $i=4$ if $\cK_b$ is simply connected.
%\nathan{Unified statement of corollary}
\end{cor}
We first prove the following lemma:
\begin{lemma}\label{lemma:b3seq}
Let $\cK$ be a flag complex, and $\bb\in\ZZ_{\geq 0}^\cV$ be of type $(1,1,1)$. Suppose that at most one edge of $\cK$ is a subset of $b$. Then there is an exact sequence
\[
\begin{tikzcd}[ampersand replacement=\&]
	 \displaystyle \bigoplus_{\substack{f\in\Delta(\cE_b)\\ \dim f=1}} \widetilde H^{-1} (\link(\{w_f\},\cK)\cap \cK_b,\KK)  \& {T^3_{-\bb}(\cK)}
\arrow[d, phantom, ""{coordinate, name=Z}]
	 \& {H^1(\Delta(\cE_b)^{(0)}*\cK_b,\KK)} 
\arrow[dll,
rounded corners,
to path={ -- ([xshift=3ex]\tikztostart.east)
|- (Z) [near end]\tikztonodes
-| ([xshift=-3ex]\tikztotarget.west)
-- (\tikztotarget)}]
	 \& \\
	\displaystyle \bigoplus_{\substack{f\in\Delta(\cE_b)\\ \dim f=1}} \widetilde H^{0} (\link(\{w_f\},\cK)\cap \cK_b,\KK)  \& {T^4_{-\bb}(\cK)}
\arrow[d, phantom, ""{coordinate, name=Y}]
	\& {H^2(\Delta(\cE_b)^{(0)}*\cK_b,\KK)} 
\arrow[dll,
rounded corners,
to path={ -- ([xshift=3ex]\tikztostart.east)
|- (Y) [near end]\tikztonodes
-| ([xshift=-3ex]\tikztotarget.west)
-- (\tikztotarget)}]
	\& {} \\
	\displaystyle \bigoplus_{\substack{f\in\Delta(\cE_b)\\ \dim f=1}} \widetilde H^{1} (\link(\{w_f\},\cK)\cap \cK_b,\KK) \&
	\arrow[from=1-1, to=1-2,hook]
	\arrow[from=1-2, to=1-3]
	\arrow[from=2-1, to=2-2]
	\arrow[from=2-2, to=2-3]
\end{tikzcd}
\]
\end{lemma}
\begin{proof}
	Fix an order of the elements $e_1,\ldots,e_k$ of $\cE_b$.
The map 
\[
	\bigoplus_{\substack{f\in\Delta(\cE_b)\\ \dim f=1}} \widetilde C_\bullet (\link(\{w_f\},\cK)\cap \cK_b)[-2]  \to C_\bullet (\widehat\cK_b,\Delta(\cE_b)^{(0)}*\cK_b)
\]
taking an oriented simplex $\sigma=(u_0,\ldots,u_j)$ in the $f=\{e_\ell,e_m\}$-summand (with $\ell<m$)  to 
$(u_0,\ldots,u_j,e_\ell,e_m)$ is an isomorphism of complexes.

The claim now follows from the long exact sequence of relative cohomology for 
$\Delta(\cE_b)^{(0)}*\cK_b$ in $\widehat\cK_b$ coupled with the dual of the above isomorphism and Theorem~\ref{thm:b3}.
\end{proof}

We return to the proof of the corollary.
\begin{proof}[Proof of Corollary \ref{cor:b3simplyconnected}]
	We consider the exact sequence of Lemma \ref{lemma:b3seq}. From \cite[6.10.9]{tomdieck}, the assumption that $\cK_b$ is connected or simply connected
	implies the vanishing of $H^i(\Delta(\cE_b)^{(0)}*\cK_b,\KK)$ for respectively $i=1,2$, and the claim follows.
\end{proof}

\subsection{Type $(2,1,1)$}\label{sec:b211}
For the case $\bb$ of type $(2,1,1)$, we need to slightly modify the complex $\cK_b$.
Let $\cK$ be a simplicial complex on $\cV$, $b\subseteq \cV$, and $u\in \cV$. We set
\[
	\cK_{u,b}=\left\{f\in \cK \ \Big|\ f\cap \left(\{u\} \cup \cV(\bigcap_{v\in b}\link (v,\cK))\right)=\emptyset\right\}.
\]

\begin{theorem}\label{thm:211}
	Let $\cK$ be a flag simplicial complex. Suppose $\bb=2e_u+e_v+e_w$ for distinct $u,v,w\in \cV$. 
If $\{u,v\}\notin\cK$ and $\{u,w\}\notin\cK$, 
\[T^4_{-\bb}(\cK)=\widetilde H^0(\cK_{u,b},\KK).\]
Otherwise, $T^4_{-\bb}(\cK)=0$.
\end{theorem}
\begin{proof}
As always, 
	we make use of Theorem \ref{thm.computingTviahomology} throughout to compute $T^4$. 	
The only generators we need to consider in homological degree $4$ will have the form $[\{f_1\}\{f_2\}\{f_3\}]$ for distinct quadrics $f_1,f_2,f_3\in I_\cK$, and $[\{f_1f_2\}\{f_3\}]$ for not necessarily distinct quadrics $f_1,f_2,f_3\in I_\cK$.
Throughout, we will view elements $f_i$ as edges in $G(I_\cK)$.

We focus first on the former type. Assume $y=[\{f_1\}\{f_2\}\{f_3\}]\in W_4^{-\bb}$.
Without loss of generality, we may assume that $u$ belongs to the edges $f_1,f_2$, but not $f_3$. 
Then $[[\{f_1\}\{f_2\}]\{f_3\}]\in W_5^{-\bb}$ because its weight is the same of $y$; moreover its image under the differential is $-y$ because $[\{f_1f_2\}\{f_3\}]$ does not satisfy the weight condition. Hence $y$ is a boundary.

We may thus focus on generators of the form $y=[\{f_1f_2\}\{f_3\}]\in W_4^{-\bb}$. Many of these are also boundaries. Notice that assuming that $y\in W_4^{-\bb}$ ensures that $u$ is in $f_3$.
Now, if the other vertex of $f_3$ is not contained in $b$, then $[\{f_1f_2f_3\}\{f_3\}]\in W_5^{-\bb}$, and this maps to $-y$ under the differential because $\wt[\{f_if_3\}\{f_3\}]-\bb\notin\mathbb{Z}_{\geq0}^{\cV}$ for $i=1,2$.

Likewise, if the edges $f_1,f_2$ are disjoint, or their intersection is not contained in $b$, then $[[\{f_1\}\{f_2\}]\{f_3\}]\in W_5^{-\bb}$, and its image (modulo previously established boundaries) is $y$.
In a similar way, if $f_2=f_3$ and  $u$ is not in $f_1$, then again $[[\{f_1\}\{f_2\}]\{f_3\}]\in W_5^{-\bb}$ so that $y$ is a boundary.

Let $U\subseteq W_4^{-\bb}$ be the subspace generated by strings of the form $[\{f_1f_2\}\{f_3\}]$ and $[\{f_1\}\{f_2f_3\}]$, where
$f_1=uv$ and $f_3=uw$, and $f_2$ intersects $f_1$ or $f_3$, respectively. By the previous paragraphs, we have that $W_4^{-\bb}=U\oplus{\partial_5^{-\bb}}(W_5^{-\bb})$. This implies in particular that if $\{u,v\}$ or $\{u,w\}$ in $\cK$, then $T^4_{-\bb}(\cK)=0$.

In the remainder of this proof, we will always suppose that $f_1=uv$ and $f_3=uw$.
For distinct $f_1,f_2,f_2'$,  
$[\{f_1f_2f_2'\}\{f_3\}]\in W_5^{-\bb}$ 
if and only if  $f_2$ and $f_2'$ intersect $f_1$, and the vertices of $f_2,f_2'$ not in $f_1$ form a simplex in $\cK$.
In this case, we see that modulo boundaries, $[\{f_1f_2\}\{f_3\}]$ is equivalent to $[\{f_1f_2'\}\{f_3\}]$.

Likewise, for $f_2\neq f_1$ and $f_2'\neq f_3$ distinct, then 
$[\{f_1f_2\}\{f_2'f_3\}]\in W_5^{-\bb}$ 
if and only if $f_2$ intersects $f_1$, $f_2'$ intersects $f_3$, and the vertices of $f_2,f_2'$ respectively not in $f_1,f_3$ form a simplex in $\cK$.
In this case, we see that $[\{f_1f_2\}\{f_3\}]$ is equivalent to $[\{f_1\}\{f_2'f_3\}]$ modulo boundaries.

The relations of the previous two paragraphs are the only relations among the elements of $U$ induced by boundaries.

We define a surjective map $\phi \colon U\to \widetilde{C}_0(\cK_{u,b})$ sending $[\{f_1f_2\}\{f_3\}]$ in $U$ (respectively, $[\{f_1\}\{f_2f_3\}]$) to the vertex of $f_2$ which is not in $f_1$ (respectively, not in $f_3$). By the above discussion, $U\cap \partial_5^{-\bb}(W_5^{-\bb})$ consists exactly of the preimage under $\phi$ of the $0$-boundaries in  $C_0(\cK_{u,b})$. Moreover, any string $[\{f_1f_2\}\{f_3\}]$ or $[\{f_1\}\{f_2f_3\}]$ in $U$ maps under the differential to $[\{f_1\}\{f_3\}]\neq 0$, so the kernel of $\partial_4^{-\bb}$ may be identified with the kernel of $\widetilde{C}_0(\cK_{u,b})\to \widetilde{C}_{-1}(\cK_{u,b})$.
Hence, $H_4(W_\bullet^{-\bb})\cong \widetilde H_0(\cK_{u,b},\KK)$, so dually we obtain
$T_{-\bb}^4\cong \widetilde H^0(\cK_{u,b},\KK)$ as desired.
\end{proof}

\begin{example}\label{ex:211}
	Consider the simplicial complex $\cK$ consisting of three isolated vertices, see Figure \ref{fig:3points}.
	The corresponding ideal $I_\cK$ is generated by $x_1x_2$, $x_1x_3$, and $x_2x_3$.
This is the ideal of the three coordinate axes in $\AA^3$.

	There are three possible choices of degrees $\bb$ of type $(2,1,1)$. For such $\bb$, let $u$ be the vertex with $\bb_u=2$. The complex $\cK_{u,b}$ then consists of the other two vertices of $\cK$, hence $T^4_{-\bb}(\cK)\cong \KK$ by Theorem \ref{thm:211}.

In fact, these degrees $-(2,1,1),-(1,2,1),-(1,1,2)$ are the only degrees $\bc$ contributing to $T^4$, and thus $T^4(\cK)\cong \KK^3$. Indeed, $\cK$ has no non-trivial links, so we may assume $\bc=-\bb$. The only other types of $\bb$ to consider are $(1)$, $(1,1)$, $(1,1,1)$, and $(1,1,1,1)$. Computing $T_\bc^4(\cK)$ in the first two types involve $H^3$ or $H^2$ of $0$-dimensional simplicial complexes, and we hence obtain zero. In the third type, the relevant complex $\cK_b$ is empty, and the fourth type cannot occur since there are only three variables. 
\end{example}
	
\begin{example}\label{ex:hex211}
We return to Example \ref{ex:hex} with $\bb$ of type $(2,1,1)$.
Up to symmetry, the only choice of $\bb=2e_u+e_v+e_w$ with $\{u,v\}$ and $\{u,w\}$ not in $\cK$ have $b=\{1,2,4\}$, $u=4$ or $b=\{1,3,5\}$, $u=5$. The complexes $\cK_{b,u}$ are pictured in Figure \ref{fig:hex211}. In both cases, $\cK_{b,u}$ is connected, so we conclude that for any $\bb$ of type $(2,1,1)$, $T^4_{-\bb}(\cK)=0$.
\end{example}
    \begin{figure}
    	\begin{tikzpicture}
    		\draw[fill] (0,0) circle [radius=0.05] node[below] {$1$};
    		\draw[fill] (1,0) circle [radius=0.05] node[below] {$2$};
    		\draw[fill] (.5,1) circle [radius=0.05] node[below] {$3$};
    	\end{tikzpicture}
    	\caption{The simplicial complex of Example \ref{ex:211}}\label{fig:3points}
    \end{figure}
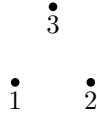

\begin{figure}
\begin{subfigure}{.3\textwidth}
	\begin{tikzpicture}
		\draw[fill,lightgray] (0,0) circle [radius=0.05] node[below left] {$1$};
		\draw[fill,lightgray] (2,0) circle [radius=0.05] node[below right] {$4$};
		\draw[fill,lightgray] (.5,1) circle [radius=0.05] node[above] {$6$};
		\draw[fill,lightgray] (1.5,1) circle [radius=0.05] node[above] {$5$};
		\draw[fill,lightgray] (.5,-1) circle [radius=0.05] node[below] {$2$};
		\draw[fill,lightgray] (1.5,-1) circle [radius=0.05] node[below] {$3$};
		\draw[lightgray] (0,0) -- (.5,-1) -- (1.5,-1) -- (2,0) -- (1.5,1) -- (.5,1) -- (0,0);
		\draw[fill] (.5,-1) circle [radius=0.05] node[below] {$2$};
		\draw[fill] (0,0) circle [radius=0.05] node[below left] {$1$};
		\draw[fill] (.5,1) circle [radius=0.05] node[above] {$6$};
		\draw[fill] (1.5,1) circle [radius=0.05] node[above] {$5$};
		\draw[fill] (1.5,-1) circle [radius=0.05] node[below] {$3$};
\draw (1.5,1) -- (.5,1) -- (0,0) -- (.5,-1) -- (1.5,-1);
	\end{tikzpicture}
	\caption{$\cK_{\{1,2,4\},4}$}
\end{subfigure}
\begin{subfigure}{.3\textwidth}
	\begin{tikzpicture}
		\draw[fill,lightgray] (0,0) circle [radius=0.05] node[below left] {$1$};
		\draw[fill,lightgray] (2,0) circle [radius=0.05] node[below right] {$4$};
		\draw[fill,lightgray] (.5,1) circle [radius=0.05] node[above] {$6$};
		\draw[fill,lightgray] (1.5,1) circle [radius=0.05] node[above] {$5$};
		\draw[fill,lightgray] (.5,-1) circle [radius=0.05] node[below] {$2$};
		\draw[fill,lightgray] (1.5,-1) circle [radius=0.05] node[below] {$3$};
		\draw[lightgray] (0,0) -- (.5,-1) -- (1.5,-1) -- (2,0) -- (1.5,1) -- (.5,1) -- (0,0);
		\draw[fill] (.5,-1) circle [radius=0.05] node[below] {$2$};
		\draw[fill] (0,0) circle [radius=0.05] node[below left] {$1$};
		\draw[fill] (.5,1) circle [radius=0.05] node[above] {$6$};
		\draw[fill] (2,0) circle [radius=0.05] node[below right] {$4$};
		\draw[fill] (1.5,-1) circle [radius=0.05] node[below] {$3$};
\draw (.5,1) -- (0,0) -- (.5,-1) -- (1.5,-1) -- (2,0);
	\end{tikzpicture}
	\caption{$\cK_{\{1,3,5\},5}$}
\end{subfigure}\\
\caption{Complexes $\cK_{b,u}$ for Example \ref{ex:hex211}}\label{fig:hex211}
\end{figure}
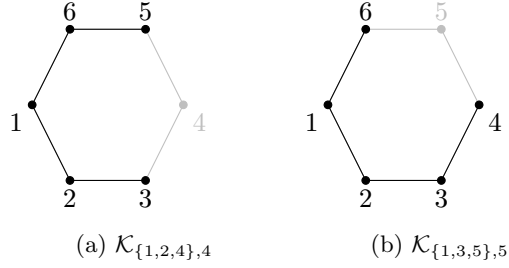

\subsection{Type $(1,1,1,1)$}\label{sec:b1111}
For this case, we need some extra notation. Fix $\bb\in\ZZ_{\geq 0}^\cV$ of type $(1,1,1,1)$. We will fix an order on the elements of $b$.
This induces an order on the edges of the subgraph of $G(I_\cK)$ obtained by restricting to the vertex set $b$ (e.g. by taking lex order).
For $i=3,4,5,6$, let $\Gamma_i(b)$ be the set of all connected graphs with vertex set $b$ and $i$ edges, none of which belong to $\cK$. These are thus subgraphs of $G(I_\cK)$. The potential combinatorial types of elements of $\Gamma_i(b)$ for $i=3,4,5,6$ are pictured in Figure~\ref{fig:gamma}. 
We may view each element of $\Gamma_i(b)$ as a face of $\Delta(\cE_b)$.
Given $\gamma\in\Gamma_i(b)$, we let $e_1(\gamma),\ldots,e_i(\gamma)$ be the edges of $\gamma$, written in order from least to greatest, and we let $\gamma_j$ be the graph obtained by deleting the edge $j$.

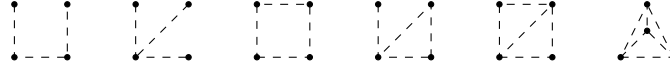
\begin{figure}
	\begin{tikzpicture}[scale=.7]
		\draw[fill] (0,0) circle [radius=0.05];
		\draw[fill] (1,0) circle [radius=0.05];
		\draw[fill] (0,1) circle [radius=0.05];
		\draw[fill] (1,1) circle [radius=0.05];
		\draw[dashed] (0,1)-- (0,0) -- (1,0) -- (1,1) ;
	\end{tikzpicture}
	\qquad\begin{tikzpicture}[scale=.7]
		\draw[fill] (0,0) circle [radius=0.05];
		\draw[fill] (1,0) circle [radius=0.05];
		\draw[fill] (0,1) circle [radius=0.05];
		\draw[fill] (1,1) circle [radius=0.05];
		\draw[dashed] (0,1) -- (0,0) -- (1,0);
		\draw[dashed] (1,1) -- (0,0);
	\end{tikzpicture}
	\qquad\begin{tikzpicture}[scale=.7]
		\draw[fill] (0,0) circle [radius=0.05];
		\draw[fill] (1,0) circle [radius=0.05];
		\draw[fill] (0,1) circle [radius=0.05];
		\draw[fill] (1,1) circle [radius=0.05];
		\draw[dashed] (0,0) -- (1,0) -- (1,1) -- (0,1) -- (0,0);
	\end{tikzpicture}
	\qquad\begin{tikzpicture}[scale=.7]
		\draw[fill] (0,0) circle [radius=0.05];
		\draw[fill] (1,0) circle [radius=0.05];
		\draw[fill] (0,1) circle [radius=0.05];
		\draw[fill] (1,1) circle [radius=0.05];
		\draw[dashed] (0,1) -- (0,0) -- (1,0);
		\draw[dashed] (1,0)-- (1,1) -- (0,0);
	\end{tikzpicture}
	\qquad\begin{tikzpicture}[scale=.7]
		\draw[fill] (0,0) circle [radius=0.05];
		\draw[fill] (1,0) circle [radius=0.05];
		\draw[fill] (0,1) circle [radius=0.05];
		\draw[fill] (1,1) circle [radius=0.05];
		\draw[dashed] (0,0) -- (1,0) -- (1,1) -- (0,1) -- (0,0) -- (1,1);
	\end{tikzpicture}
    \qquad\begin{tikzpicture}[scale=.7]
		\draw[fill] (0,0) circle [radius=0.05];
		\draw[fill] (1,0) circle [radius=0.05];
		\draw[fill] (.5,1) circle [radius=0.05];
		\draw[fill] (.5,.5) circle [radius=0.05];
		\draw[dashed] (0,0) -- (1,0) -- (.5,.5) -- (.5,1) -- (0,0) -- (.5,.5);
        \draw[dashed] (1,0) -- (.5,1);
	\end{tikzpicture}
	\caption{Combinatorial types of elements of $\Gamma_i(b)$}\label{fig:gamma}
\end{figure}

\begin{definition}
    A \emph{splitting} $\Split$ of $\gamma\in\Gamma_i(b)$ consists of two connected subgraphs of $\gamma$, each containing at least one edge, such that the intersection $\cap\Split$ of these two subgraphs is a subset of $b$ (i.e.\ only contains vertices), and moreover $\cap\Split \in \cK$.
\end{definition}

If $\gamma'\in \Gamma_{i-1}(b)$ is a subgraph of $\gamma\in \Gamma_i(b)$, we may consider the intersection of the partition $\Split$ with $\gamma'$ which we will denote by $\Split_{|\gamma'}$. Note that this might fail to be a splitting of $\gamma'$.

Graphs $\gamma$ in $\Gamma_3(b)$ have one, two, or three splittings. 
For $\gamma$ a path of length three, the splitting with $\cap  \Split$ maximal is \emph{non-standard}. The other splitting, and all three splittings for the other combinatorial type in $\Gamma_3(v)$, are \emph{standard}.
Graphs in $\Gamma_4(b)$ have zero, one, or two splittings. A splitting of $\gamma\in \Gamma_4(b)$ is \emph{standard} if for all subgraphs $\gamma'\in \Gamma_3(b)$ of $\gamma$, either $\Split_{|\gamma'}$ is standard or fails to be a splitting.\footnote{Graphs in $\Gamma_5(b)$ and $\Gamma_6(b)$ never have splittings.}

\begin{definition}\label{def.splittingsign}
    For any splitting $\Split$ of $\gamma\in\Gamma_i(b)$, we define its \emph{sign} as follows. Let \[e_1(\Split,1),\ldots,e_\ell(\Split,1),e_1(\Split,2),\ldots,e_m(\Split,2)\] be the edges of $\gamma$, labeled such that 
    \[ \Split=\{\{e_1(\Split,1),\ldots,e_\ell(\Split,1)\},\{e_1(\Split,2),\ldots,e_m(\Split,2)\}\}, \]
    the elements of the two sets above in $\Split$ are written in increasing order, and $e_1(\Split,1)=e_1(\gamma)$. Then $\sign(\Split)$ is the sign of the permutation taking $e_1(\gamma),\ldots,e_i(\gamma)$ to $e_1(\Split,1),\ldots,e_\ell(\Split,1),e_1(\Split,2),\ldots,e_m(\Split,2)$. 
\end{definition}

We now define a subcomplex $\widehat\cK_b^\std$ of $\Delta(\cE_b)*\cK_b$. 
Indeed, let $\widehat\cK_b^\std$ consist of the union of the complexes:
\begin{align*}
	\left(\Delta(\cE_b)\setminus \bigcup_{i=3}^6 \Gamma_i(b)\right)*\cK_b;\ \textrm{and}\\
	\Delta(\gamma)*\left(\link(\cap \Split,\cK)\cap \cK_b\right)
\end{align*}
as $\gamma,\Split$ range over all $\gamma\in \bigcup_{i=3,4} \Gamma_i(b)$ and standard splittings $\Split$ of $\gamma$.

Consider the map $\eta:K_1\to K_0$ arising from the sum of $\eta_1,\ldots,\eta_4$ in the following diagram:
\[
\begin{tikzcd}[ampersand replacement=\&,column sep=tiny]
K_1:\&
\begin{array}{c} \displaystyle\bigoplus_{\substack{\gamma\in \Gamma_4(b)\\ \Split\ \textrm{non-std.}}} \KK\cdot (\gamma,\Split)  \end{array} \& \oplus \&\begin{array}{c} \displaystyle\bigoplus_{\substack{\gamma\in \Gamma_3(b)\\ \Split\ \textrm{non-std.}}} C_0(\link(\cap \Split,\cK)\cap \cK_b)\cdot (\gamma,\Split) \end{array}\\
K_0:  \&{\displaystyle\bigoplus_{\substack{\gamma\in\Gamma_3(b)\\ \Split\ \textrm{non-std.}}}\KK\cdot (\gamma,\Split)} \&\oplus \&	H_2(\widehat\cK_b^\std,\KK)
\arrow["\eta_1",from=1-2, to=2-2]
\arrow["\eta_3"{pos=0.7},from=1-2, to=2-4]
	\arrow["\eta_2",from=1-4, to=2-2]
	\arrow["\eta_4",from=1-4, to=2-4]
\end{tikzcd}
\]
where the maps $\eta_i$ are defined as follows. We have  
\[
	\eta_1((\gamma,\Split))=\sum_j (-1)^j\cdot \sign(\Split)\cdot \sign(\Split_{|\gamma_j})\cdot(\gamma_j,\Split_{|\gamma_j})
\]
where we treat $(\gamma_j,\Split_{|\gamma_j})$ as zero if $\gamma_j\notin \Gamma_3(b)$ or $\Split_{|\gamma_j}$ is not a non-standard splitting.
Likewise,
\[
\eta_2((v)\cdot (\gamma,\Split))=(\gamma,\Split).
\]

Under $\eta_3$, $(\gamma,\Split)$ maps to the class of the boundary of the simplicial chain $\sign(\Split)\cdot (e_1(\gamma),e_2(\gamma),e_3(\gamma),e_4(\gamma))$.
Likewise, under $\eta_4$, $(v)\cdot (\gamma,\Split)$ maps to the class of the boundary of the simplicial chain $\sign(\Split)\cdot (e_1(\gamma),e_2(\gamma),e_3(\gamma),v)$.

\begin{theorem}\label{thm:1111}
	Consider $\bb\in \ZZ_{\geq 0}^\cV$ of type $(1,1,1,1)$. Suppose that $\cK$ is a flag complex. 
If more than three edges of $\cK$ are subsets of $b$, then $T^4_{-\bb}(\cK)=0$. Otherwise,
\[
	T^4_{-\bb}(\cK)\cong \Hom_\KK(\coker\eta,\KK).
\]
\end{theorem}
\noindent We delay the proof until the end of the section, first providing an example illustrating the theorem.
\begin{example}\label{ex:hex1111}
	We return to Example \ref{ex:hex} with $\bb$ of type $(1,1,1,1)$.
	We will depict $\cK_b$ for various cases in Figure \ref{fig:1111}.
The complex $\cK_b$ is depicted in black, and the edges in $\cE_b$ are dashed gray lines. The sets $\cap \Split$ for various splittings will be depicted in blue, red, or orange, as we describe below.

	First consider the case $b=\{1,2,3,4\}$. 
	The complex $\cK_b$ consists of the vertices $5$ and $6$ with an edge between them. $\cE_b$ consists of the edges $e_1=\{1,3\}$, $e_2=\{1,4\}$, and $e_3=\{2,4\}$. The set $\Gamma_3(b)$ has a single element $\gamma$, with both a standard and a non-standard splitting. The sets $\Gamma_4(b)$, $\Gamma_5(b)$, and $\Gamma_6(b)$ are empty. 
	For the standard splitting $\Split$ of $\gamma$,
$\cap \Split$ is depicted in blue, and 
	\[\link(\cap\Split,\cK)\cap \cK_b=\{6\}.\]
	The complex $\widehat\cK_b^\std$ has maximal faces consisting of any $4$-element subset of $\{5,6,e_1,e_2,e_3\}$ except for $\{5,e_1,e_2,e_3\}$. If we include the face $\{5,e_1,e_2,e_3\}$, the complex is the boundary of a $4$-simplex, that is, a $3$-sphere. Removing the open face $\{5,e_1,e_2,e_3\}$ punctures the sphere, and we obtain something contractible. In particular, $H_2(\widehat\cK_b^\std,\KK)=0$.

	For the nonstandard splitting  $\Split'$ of $\gamma$, $\cap \Split'$ is depicted in red, and
	\[\link(\cap\Split',\cK)\cap \cK_b=\{5\}.\] We thus obtain that $\eta$ has the form
\[
	\begin{tikzcd}[ampersand replacement=\&,column sep=tiny]
	K_1:\& 0 \&\oplus\& \KK\\
	K_0:\&\KK \&\oplus\& 0\\
\arrow[from=1-2, to=2-2]
\arrow[from=1-2, to=2-4]
\arrow[from=1-4, to=2-2]
	\arrow[from=1-4, to=2-4]
\end{tikzcd}
\]
with the map $\KK\to \KK$ the identity. Thus, $T^4_{-\bb}(\cK)=0$ for this $\bb$.

We next consider the case $b=\{1,3,4,5\}$. The complex $\cK_b$ consists of the isolated vertices $2$ and $6$.
$\cE_b$ consists of the edges $e_1=\{1,3\}$, $e_2=\{1,4\}$, $e_3=\{1,5\}$, and $e_4=\{3,5\}$. The set $\Gamma_4(b)$ has a single element $\gamma$ with a single splitting $\Split=\{\{e_1,e_3,e_4\},\{e_2\}\}$, which is standard. The set $\Gamma_3(b)$ contains the elements $\gamma_1$, $\gamma_3$, and $\gamma_4$. The induced splittings $\Split_{\gamma_i}$ for $i=1,3,4$ are all standard; there are also non-standard splittings for $\gamma_1$ and $\gamma_3$. For $\Split$ (and the splittings induced by it), $\cap \Split$ is depicted in blue and 
\[\link(\cap\Split,\cK)\cap \cK_b=\cK_b.\]
We thus obtain that $\widehat\cK_b^\std=\Delta(\cE_b)*\cK_b$, hence $\widehat\cK_b^\std$ is contractible and $H_2(\widehat\cK_b^\std,\KK)=0$.

For the non-standard splittings $\Split'$ of $\gamma_1$ and $\gamma_3$, 
$\cap \Split'$ is depicted in red or orange, and in both cases
\[\link(\cap\Split',\cK)\cap \cK_b\]
consists of a single vertex ($\{6\}$ in the case of $\gamma_1$ and $\{2\}$ in the case of $\gamma_3$). 
We thus obtain that $\eta$ has the form
\[
	\begin{tikzcd}[ampersand replacement=\&,column sep=tiny]
	K_1:\& 0 \&\oplus\& \KK^2\\
	K_0:\&\KK^2 \&\oplus\& 0\\
\arrow[from=1-2, to=2-2]
\arrow[from=1-2, to=2-4]
\arrow[from=1-4, to=2-2]
	\arrow[from=1-4, to=2-4]
\end{tikzcd}
\]
with the map $\KK^2\to \KK^2$ the identity. Thus, $T^4_{-\bb}(\cK)=0$ for this $\bb$.

Finally, consider the case $b=\{1,2,4,5\}$. The complex $\cK_b$ consists of the isolated vertices $3$ and $6$.
$\cE_b$ consists of the edges $e_1=\{1,4\}$, $e_2=\{1,5\}$, $e_3=\{2,4\}$, and $e_4=\{2,5\}$. The set $\Gamma_4(b)$ has a single element $\gamma$ with both a standard splitting $\Split$ and a non-standard splitting $\Split'$.
The sets $\cap \Split$ and $\cap \Split'$ are depicted in blue and red, respectively.
The set $\Gamma_3(b)$ contains four elements $\gamma_1,\gamma_2,\gamma_3,\gamma_4$, each with a standard splitting $\Split_{|\gamma_i}$ and a non-standard splitting $\Split'_{|\gamma_i}$.
We have 
\[\link(\cap\Split_{|\gamma_i},\cK)\cap \cK_b=\begin{cases}
	\{3\}&i=1,2\\
	\{6\}&i=3,4\end{cases}
\]
and
\[\link(\cap\Split'_{|\gamma_i},\cK)\cap \cK_b=\begin{cases}
	\{3\}&i=2,4\\
	\{6\}&i=1,3\end{cases}.
\]

It follows that the complex $\widehat\cK_b^\std$ 
is the complex $X$ from Example \ref{ex:shelling}, and hence
$H_2(\widehat\cK_b^\std,\KK)\cong \KK^2$, with basis given by
the classes of the boundaries of $\{6,e_2,e_3,e_4\}$ and $\{3,e_1,e_2,e_3\}$.

We thus obtain that $\eta$ has the form
\[
	\begin{tikzcd}[ampersand replacement=\&,column sep=tiny]
	K_1:\& \KK \&\oplus\& \KK^4\\
	K_0:\&\KK^4 \&\oplus\& \KK^2\\
\arrow[from=1-2, to=2-2]
\arrow[from=1-2, to=2-4]
\arrow[from=1-4, to=2-2]
	\arrow[from=1-4, to=2-4]
\end{tikzcd}.
\]
The summands of the $\KK^4$ term in $K_1$ corresponding to $\gamma_1$ and $\gamma_4$ map under $\eta_4$ to the boundaries of $\{3,e_1,e_2,e_3\}$ and $\{6,e_2,e_3,e_4\}$, and thus surject onto the $\KK^2$ term in $K_0$. 
The remaining summands of the $\KK^4$ in $K_1$ map surjectively under $\eta_2$ to the $\gamma_2$ and $\gamma_3$ terms in the $\KK^4$ term of $K_0$. Likewise, the image under $\eta_1$ of the $\KK$ term in $K_1$  is a sum involving all $(\gamma_i,\Split'_{|\gamma_i})$. Hence, the map $\eta$ is injective, and we conclude $T^4_{-\bb}(\cK)\cong \KK$ for this $\bb$.

All weights $\bb$ of type $(1,1,1,1)$ are equivalent up to symmetry to one of the above three cases. Moreover, we have seen in Example \ref{ex:hex} and \ref{ex:hex211} that no other weights contribute to $T^4(\cK)$. Hence, $T^4(\cK)\cong \KK^3$, with generators in weights $-(1,1,0,1,1,0)$, $-(0,1,1,0,1,1)$, and $-(1,0,1,1,0,1)$.

\begin{figure}
\begin{subfigure}{.3\textwidth}
	\begin{tikzpicture}
		\draw[fill,lightgray] (0,0) circle [radius=0.05] node[below left] {$1$};
		\draw[fill,lightgray] (2,0) circle [radius=0.05] node[below right] {$4$};
		\draw[fill,lightgray] (.5,1) circle [radius=0.05] node[above] {$6$};
		\draw[fill,lightgray] (1.5,1) circle [radius=0.05] node[above] {$5$};
		\draw[fill,lightgray] (.5,-1) circle [radius=0.05] node[below] {$2$};
		\draw[fill,lightgray] (1.5,-1) circle [radius=0.05] node[below] {$3$};
		\draw[gray] (1,0) node[above] {\scriptsize$e_2$};
		\draw[gray] (.5,-.5) node {\scriptsize$e_1$};
		\draw[gray] (1.5,-.5) node {\scriptsize $e_3$};
		\draw[lightgray] (0,0) -- (.5,-1) -- (1.5,-1) -- (2,0) -- (1.5,1) -- (.5,1) -- (0,0);
		\draw[lightgray,dashed] (1.5,-1) -- (0,0) -- (2,0) -- (.5,-1);
		\draw[fill] (.5,1) circle [radius=0.05] node[above] {$6$};
		\draw[fill] (1.5,1) circle [radius=0.05] node[above] {$5$};
\draw (1.5,1) -- (.5,1);
		\draw[fill,blue] (0,0) circle [radius=0.07];
		\draw[fill,red] (2,0) circle [radius=0.07];
	\end{tikzpicture}
	\caption{${\cK}_{\{1,2,3,4\}}$ }\label{fig:1234}
\end{subfigure}
\begin{subfigure}{.3\textwidth}
	\begin{tikzpicture}
		\draw[fill,lightgray] (0,0) circle [radius=0.05] node[below left] {$1$};
		\draw[fill,lightgray] (2,0) circle [radius=0.05] node[below right] {$4$};
		\draw[fill,lightgray] (.5,1) circle [radius=0.05] node[above] {$6$};
		\draw[fill,lightgray] (1.5,1) circle [radius=0.05] node[above] {$5$};
		\draw[fill,lightgray] (.5,-1) circle [radius=0.05] node[below] {$2$};
		\draw[fill,lightgray] (1.5,-1) circle [radius=0.05] node[below] {$3$};
		\draw[lightgray] (0,0) -- (.5,-1) -- (1.5,-1) -- (2,0) -- (1.5,1) -- (.5,1) -- (0,0);
		\draw[lightgray,dashed] (1.5,-1) -- (0,0) -- (2,0);
		\draw[lightgray,dashed]	(0,0) -- (1.5,1) -- (1.5,-1);
		\draw[fill] (.5,-1) circle [radius=0.05] node[below] {$2$};
		\draw[fill] (.5,1) circle [radius=0.05] node[above] {$6$};
		\draw[fill,blue] (0,0) circle [radius=0.07];
		\draw[fill,red] (1.5,1) circle [radius=0.07];
		\draw[fill,orange] (1.5,-1) circle [radius=0.07];
		\draw[gray] (1,0) node[above] {\scriptsize$e_2$};
		\draw[gray] (.5,-.5) node {\scriptsize$e_1$};
		\draw[gray] (.5,.5) node {\scriptsize $e_3$};
		\draw[gray] (1.65,.4) node {\scriptsize $e_4$};
	\end{tikzpicture}
	\caption{${\cK}_{\{1,3,4,5\}}$ }\label{fig:1345}
\end{subfigure}
\begin{subfigure}{.3\textwidth}
	\begin{tikzpicture}
		\draw[fill,lightgray] (0,0) circle [radius=0.05] node[below left] {$1$};
		\draw[fill,lightgray] (2,0) circle [radius=0.05] node[below right] {$4$};
		\draw[fill,lightgray] (.5,1) circle [radius=0.05] node[above] {$6$};
		\draw[fill,lightgray] (1.5,1) circle [radius=0.05] node[above] {$5$};
		\draw[fill,lightgray] (.5,-1) circle [radius=0.05] node[below] {$2$};
		\draw[fill,lightgray] (1.5,-1) circle [radius=0.05] node[below] {$3$};
		\draw[lightgray] (0,0) -- (.5,-1) -- (1.5,-1) -- (2,0) -- (1.5,1) -- (.5,1) -- (0,0);
		\draw[lightgray,dashed] (1.5,1) -- (0,0) -- (2,0) -- (.5,-1) -- (1.5,1);
		\draw[fill] (.5,1) circle [radius=0.05] node[above] {$6$};
		\draw[fill] (1.5,-1) circle [radius=0.05] node[below] {$3$};
		\draw[fill,blue] (0,0) circle [radius=0.07];
		\draw[fill,red] (2,0) circle [radius=0.07];
		\draw[fill,blue] (.5,-1) circle [radius=0.07];
		\draw[fill,red] (1.5,1) circle [radius=0.07];
		\draw[gray] (1.5,.15) node {\scriptsize$e_1$};
		\draw[gray] (.5,.5) node {\scriptsize$e_2$};
		\draw[gray] (1.5,-.5) node {\scriptsize $e_3$};
		\draw[gray] (.65,-.3) node {\scriptsize $e_4$};
	\end{tikzpicture}
	\caption{${\cK}_{\{1,2,4,5\}}$ }\label{fig:1245}
\end{subfigure}
\caption{Complexes $\cK_b$ and $\cE_b$ in Example \ref{ex:hex1111}}\label{fig:1111}
\end{figure}
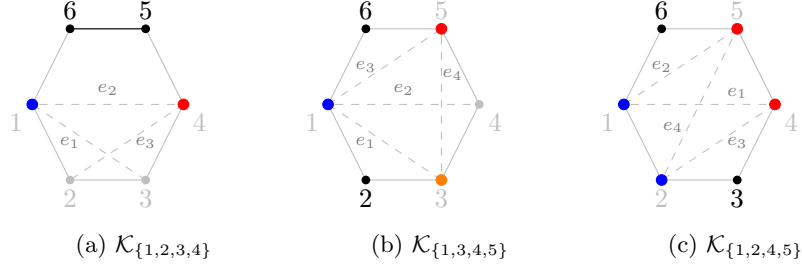
\end{example}

\begin{proof}[Proof of Theorem \ref{thm:1111}]
	As usual, by Theorem \ref{thm.computingTviahomology} we will compute $T_{-\bb}^4(\cK)$ as the dual of $H_4(W_\bullet^{-\bb})$.  

Given $\gamma\in \Gamma_i(b)$ with $i=3,4$, we let $y_\gamma$ be the generator
\[
	y_\gamma=\{f_1(\gamma)\cdots f_i(\gamma)\}
\]
where $f_j(\gamma)$ corresponds to the edge $e_j(\gamma)$.
Likewise, given a splitting $\Split$ of $\gamma$, we let $y_\Split$ be the generator
\[
y_\Split=-[\{f_1(\Split,1)\cdots f_\ell(\Split,1)\}\{f_1(\Split,2)\cdots f_m(\Split,2)\}]
\]
where $f_j(\Split,k)$ corresponds to the edge $e_j(\Split,k)$ of $\gamma$, and $e_j(\Split,k)$ is as in Definition~\ref{def.splittingsign}.
By definition of $\Gamma_i(b)$, both types of generators belong to $W_\bullet^{-\bb}$. This follows from Remark~\ref{rem:graphweight} for $y_\gamma$ and from Remark~\ref{rem:splittings} for $y_\Split$. Moreover, $y_\Split$ is a splitting of $\sign(\Split)y_\gamma$.

We next observe that any generator $y=\{f_1f_2f_3\}\in W_3^{-\bb}$ has a splitting by Remark~\ref{rem:splittings}. 
Now, if $G(y)$ is disconnected, then any two splittings differ by an element of $\partial_4^{-\bb}(W_4^{-\bb})$. Indeed, if there is more than one splitting, then $[\{f_1\}\{f_2\}\{f_3\}]\in W_4^{-\bb}$ and 
\[
	[[\{f_1\}\{f_2\}]\{f_3\}], [\{f_1\}[\{f_2\}\{f_3\}]], [[\{f_1\}\{f_3\}]\{f_3\}]
\]
all belong to $W_5^{-\bb}$, and their images under $\partial_5^{-\bb}$ generate all possible differences of splittings. For any $y$ with $G(y)$ disconnected, we call any splitting standard.

If instead $G(y)$ is connected, then $\gamma=G(y)$ is an element of $\Gamma_3(b)$, and $y_\gamma$ agrees with $y$ up to sign. Thus, for any splitting $\Split$ of $\gamma$, $y_\Split$ is a splitting (up to sign) of $y$. We thus obtain a bijection between splittings of $\gamma$ and splittings of $y$.
 We call a splitting of $y$ standard (respectively non-standard) if the corresponding splitting of $G(y)$ is standard (respectively non-standard).

 If $y=\{f_1f_2f_3\}$ has $G(y)$ a tree $\gamma$ with trivalent vertex, then similar
 to above $[\{f_1\}\{f_2\}\{f_3\}]\in W_4^{-\bb}$, and any two splittings of $y$ differ by a boundary. It follows that \emph{any} generator $y=\{f_1f_2f_3\}\in W_3^{-\bb}$ has a unique standard splitting, up to the image of $\partial_4^{-\bb}$.

 We now describe a subcomplex $W_\bullet^\std$ of $W_\bullet^{-\bb}$. For $i\leq 3$, $W_i^\std=W_i^{-\bb}$. For $i=4$ we take $W_4^\std$ to be generated by all primitive generators as well as all standard splittings. For $i\geq5$ we inductively take $W_i^\std$ to be the subspace of $W_i^{-\bb}$ generated by all generators mapping under $\partial^{-\bb}_i$ to $W_{i-1}^\std$.
We denote the differential of $W_\bullet^\std$ by $\partial^\std$.

For any non-standard splitting $\Split$ of some $\gamma\in \Gamma_3(b)$, we set 
\[
	\kappa_\Split=y_{\Split}-\sign(\Split)\sign(\Split')y_{\Split'},
\]
where $\Split'$ is the standard splitting of $\gamma$.
It is straightforward to check that $\partial_4^{-\bb}(\kappa_\gamma)=0$. Moreover, the kernel of $\partial_4^{-\bb}:W_4^{-\bb}\to W_3^{-\bb}$ decomposes as a direct sum
\[
	\ker \partial_4^{-\bb}=\langle \kappa_\Split \rangle_{\{\Split \textrm{ non-std}\}} \oplus \ker \partial_4^\std.
\]
Indeed, the only generators of $W_4^{-\bb}$ not in $W_4^\std$ are those of the form $y_\Split$ for $\Split$ non-standard, so any element of $W_4^{-\bb}$ may be written as a sum of something in $\langle \kappa_\Split\rangle $ and something in $W_4^\std$. The directness of the decomposition follows from the fact that the generators $y_\Split$ for $\Split$ non-standard are linearly independent from a basis of $W_4^\std$.

We next consider those generators $y$ of $W_5^{-\bb}$ that do not belong to $W_5^\std$. By definition, $y$ cannot be primitive. Moreover, a direct computation shows that $y$ cannot be of the form $[[\{f_1\}\{f_2\}]\{f_3\}]$, because all the summands appearing in the image of such a generator correspond to a standard splitting. Hence we can assume $y$ to be quasiprimitive, and that the restriction of $G(y)$ to $b$ must belong to $\Gamma_3(b)$ or $\Gamma_4(b)$.
If $G(y)$ restricted to $b$ belongs to $\Gamma_4(b)$, then in fact $G(y)$ already belongs to $\Gamma_4(b)$, and there is a splitting $\Split$ of $G(y)$ so that $y_\Split=\pm y$. It follows that $\Split$ must be a non-standard splitting. 
On the other hand, for $\gamma\in \Gamma_3(b)$, we let $W_\gamma$ be the span of all generators $y$ in $W_5^{-\bb}$ that do not belong to $W_5^\std$ and such that $G(y)$ restricted to $b$ is $\gamma$. 

With the notation from above, we have shown that $H_4(W_\bullet^{-\bb})$ is the cokernel of 
\begin{equation}\label{eqn:proofdiagram}
	\begin{tikzcd}[ampersand replacement=\&,column sep=tiny]
\&
\begin{array}{c} \displaystyle\bigoplus_{\substack{\gamma\in \Gamma_4(b)\\ \Split\ \textrm{non-std.}}} \KK\cdot y_\Split  \end{array} \& \oplus \&\begin{array}{c} \displaystyle\bigoplus_{\substack{\gamma\in \Gamma_3(b)\\ \phantom{\Split\ \textrm{non-std.}}}} W_\gamma \end{array}\\
 \&{\displaystyle\bigoplus_{\substack{\gamma\in\Gamma_3(b)\\ \Split\ \textrm{non-std.}}}\KK\cdot \kappa_\Split} \&\oplus \&	H_4(W_\bullet^\std)
\arrow[from=1-2, to=2-2]
\arrow[from=1-2, to=2-4]
	\arrow[from=1-4, to=2-2]
	\arrow[from=1-4, to=2-4]
\end{tikzcd}
\end{equation}
with all arrows induced by $\partial^{-\bb}$. It remains to relate this two-term complex to the complex $K_\bullet$ in the statement of the theorem.

Let $B_i$ be the subspace of $W_i^\std$ generated by primitive generators, and $A_{i-1}$ the subspace of $W_i^\std$ generated by generators that are not primitive. Let $f_i:A_i\to B_i$ be the map induced by $\partial^\std$ composed with the projection $W_i^\std \to B_i$. We notice that $f_2$ is injective (see the $i=3$ case of Proposition \ref{prop:applicable}) and the kernel of $f_3$ is contained in the image of $\partial_5^\std$, by construction of $W_i^\std$.
Let $\widetilde W_i^\std=\coker f_i$. Similar to \S\ref{sec:nonsplittingproof}, $\widetilde W_\bullet^\std$ is a complex, and Lemma~\ref{lemma:homological} implies that the natural map
\[ H_4(W_\bullet^\std)\to H_4(\widetilde W_\bullet^\std) \]
is an isomorphism. Thus, we may replace $H_4(W_\bullet^\std)$ by 
$H_4(\widetilde{W}_\bullet^\std)$
in \eqref{eqn:proofdiagram}.

For any primitive generator $y\in W_j^{-\bb}$,
every edge of $G(y)$ must be adjacent to an element of $b$, every element of $b$ must be a vertex of $G(y)$, and the vertices of $G(y)$ outside of $b$ must form a simplex in $\cK$. 
The map of complexes
\[
	\psi:C_\bullet(\Delta(\cE_b)*\cK_b')[-1]\to \widetilde W_\bullet^\std
\]
defined as in \eqref{eqn:psi} in \S\ref{sec:b3} is thus surjective.

As in \S\ref{sec:b3}, we consider the kernel of $\psi$.
For $y$ a primitive generator with $G(y)$ restricted to $b$ not in some $\Gamma_i(b)$, either $y\notin W_\bullet^{-\bb}$, or the image of $y$ in $\widetilde W_\bullet^\std$ vanishes. Hence, any oriented simplex $\sigma$ of 
\begin{equation}\label{eqn:1111proof1}
	\left(\Delta(\cE_b)\setminus \bigcup_{i=3}^6 \Gamma_i(b)\right)*\cK_b'\end{equation}
lies in the kernel of $\Psi$.

Any oriented simplex $\sigma$ of $\Delta(\cE_b)*\cK_b'$ not belonging to \eqref{eqn:1111proof1} corresponds to a primitive generator $y$ that necessarily belongs to $W_\bullet^{-\bb}$. It remains to see when the image of such a generator vanishes in $\widetilde W_\bullet^\std$. This is exactly the case when $\sigma$ belongs to 
	\[\Delta(\gamma)*\left(\link(\cap \Split,\cK)\cap \cK_b\right)
	\]
for some $\gamma\in \Gamma_i(b)$ ($i=3,4$) and $\Split$ a standard splitting of $\gamma$.

Letting ${\widehat\cK_b^{\std,\prime}}$ be the union of \eqref{eqn:1111proof1} with all the $\Delta(\gamma)*\left(\link(\cap \Split,\cK)\cap \cK_b\right)$ above, it follows that $\psi$ induces an isomorphism
\[
	C_\bullet(\Delta(\cE_b)*\cK_b',{\widehat\cK_b^{\std,\prime}})[-1]\to \widetilde W_\bullet^\std.
\]

If more than three edges of $\cK$ are subsets of $b$, then all $\Gamma_i$ are empty, and ${\widehat\cK_b^{\std,\prime}}=\Delta(\cE_b)*\cK_b'$. It follows that all terms in \eqref{eqn:proofdiagram} are zero, so $T^4_{-\bb}(\cK)=0$.
We thus suppose going forward that at most three edges of $\cK$ are subsets of $b$, in which case $\Gamma_3$ is necessarily non-empty.

Since $\Delta(\cE_b)*\cK_b'$ is contractible, by the long exact sequence of cohomology we may identify
$H_4(\widetilde W_\bullet^\std)$  with $H_2({\widehat\cK_b^{\std,\prime}},\KK)$. 
The map $\phi_b$ induces a surjective simplicial map $\widehat\phi_b^\std:\widehat \cK_b^{\std,\prime}\to \widehat \cK_b$. As in Lemma \ref{lemma:homotopic}, this again induces a homotopy equivalence of $\widehat \cK_b^{\std,\prime}$ and $\widehat \cK_b^\std$, because the preimage of any simplex under $\widehat\phi_b$ is again a simplex. 

Thus, we may replace \eqref{eqn:proofdiagram} by a diagram in which the $H_4(W_\bullet^\std)$ has been replaced by $H_2(\widehat \cK_b^\std,\KK)$. The $y_\Split$ and $\kappa_\Split$ from this diagram we may rename simply by $(\gamma,\Split)$ as in the description of $K_1,K_0$, and it is straightforward to check that the maps $\eta_1,\eta_3$ agree with the induced maps in \eqref{eqn:proofdiagram}.

It remains to deal with the $W_\gamma$ summands of \eqref{eqn:proofdiagram}. Consider any generator $y\in W_\gamma$, represented by a string $s$. Then there is an element $g\in I_\cK$ and a non-standard splitting $\Split$ of $\gamma$ such that the string $s\setminus g$ represents $y_\Split$, up to sign. Moreover, the vertices $\{u,v\}$ of the edge of $G(y)$ corresponding to $g$ must satisfy (after permuting $u,v$) that $u\in b$ and $v\in \link((\cap \Split,\cK)\cap\cK_b)$. The image of $y$ in both 
\[\bigoplus_{\substack{\gamma\in\Gamma_3(b)\\ \Split\ \textrm{non-std.}}}\KK\cdot (\gamma,\Split)\qquad\textrm{and}\qquad H_2(\widehat\cK_b^\std,\KK)
\]
is independent of $u$, and depends on the position of $g$ in $s$ only up to sign; in fact, it is straightforward to check that the images of $y$ agree up to sign with $\eta_2((v)\cdot (\gamma,\Split))$ and $\eta_4((v)\cdot (\gamma,\Split))$.
Conversely, for any $v \in \link((\cap \Split,\cK)\cap\cK_b)$, there is $u\in b$ such that $g=x_ux_v\in I_\cK$, and we obtain a generator of $W_\gamma$ by inserting $g$ into any primitive string contained in a string representing $y_\Split$.
Hence, we may replace $W_\gamma$ in \eqref{eqn:proofdiagram} by $C_0(\link(\cap \Split,\cK)\cap\cK_b)\cdot (\gamma,\Split)$ and the claim of the theorem follows.
\end{proof}

We also obtain the following result guaranteeing vanishing of $T^4_{-\bb}(\cK)$.
\begin{cor}\label{cor:b4vanishing}
	Let $\cK_b$ be a flag complex, and $\bb\in \ZZ_{\geq 0}^\cV$ be of type $(1,1,1,1)$. Suppose that 
	\[H^2\left(\left(\Delta(\cE_b)\setminus\bigcup_{i=3}^6 \Gamma_i(b)\right)*\cK_b,\KK\right)=0\]
	and for every $\gamma\in\Gamma_3(b)$ and splitting $\Split$ of $\gamma$,
	\begin{equation}\label{eqn:t4links}
\link(\cap\Split,\cK)\cap\cK_b\neq \emptyset.
\end{equation}
Then $T^4_{-\bb}(\cK)=0$. In particular, $T^4_{-\bb}(\cK)=0$ if $\cK_b$ is connected and \eqref{eqn:t4links} holds.
\end{cor}
\begin{proof}
	We consider the subcomplex 
	\[
X=(\Delta(\cE_b)\setminus\bigcup_{i=3}^6 \Gamma_i(b))*\cK_b
	\]
	of $\widehat\cK_b^\std$.
Similar to the proof of Lemma 
	\ref{lemma:b3seq}
	we have an isomorphism
\[
	\bigoplus_{\substack{\gamma\in \Gamma_3(b)\\ \Split\ \textrm{std.}}} \widetilde C_\bullet (\link(\cap\Split,\cK)\cap \cK_b)[-3]  \to C_\bullet (\widehat\cK_b^\std,
	X
	).
\]
Thus by the long exact sequence of homology and (the dual of) this isomorphism, under the hypotheses of the first claim we obtain
\[
H_2(\widehat\cK_b^\std,\KK)=0.
\]
But then \eqref{eqn:t4links} for the non-standard splittings implies that $\eta$ is surjective, so that $T^4_{-\bb}(\cK)=0$ by Theorem~\ref{thm:1111}.

The last part of the statement follows from the first claim. Indeed, assume that $\cK_b$ is connected. So is $\Delta(\cE_b)\setminus\bigcup_{i=3}^6 \Gamma_i(b))$, so then by 
\cite[6.10.9]{tomdieck}
we obtain the vanishing of the required second cohomology group.
\end{proof}

\section{$T^3$ vanishing and the associahedron}
\subsection{Criteria for $T^3$ vanishing}\label{sec:vanishing}
Putting our topological results together, we obtain the following sufficient criterion for $T^3$ vanishing.
\begin{cor}\label{cor:vanishing}
Let $\cK$ be a flag complex on vertex set $\cV$.
Suppose that the following hold:
\begin{enumerate}
	\item For every vertex $v$ of $\cK$, $T^3(\link(v,\cK))=0$;
	\item For every $b\subseteq \cV$ of cardinality $1$, $2$, or $3$ and $b$ containing at most $\#b-2$ edges of $\cK$,
		\[
			\widetilde H^{3-\#b}(\cK_b,\KK)=0;\ \textrm{and}
		\]
	\item For every $u,v,w\in \cV$ with $\{u,v\},\{u,w\}\notin\cK$, $\link(u,\cK)\cap \cK_b\neq \emptyset$.
\end{enumerate}
Then $T^3(\cK)=0$.
\end{cor}
\begin{proof}
	Since every proper link of $\cK$ arises by taking a link of $\link(v,\cK)$ for some vertex $v$, the localization formula (Theorem \ref{thm:localizationformula}) implies that it suffices to show that $T_{-\bb}^3(\cK)=0$ for every $\bb\in\ZZ_{\geq 0}^\cV$. Moreover, by Theorem \ref{thm:bound} and Proposition  \ref{prop:t3bound} it suffices to consider $\bb$ of type $(1)$, $(1,1)$, and $(1,1,1)$. The vanishing of $T_{-\bb}^3(\cK)=0$ then follows from Theorems \ref{thm:b1} and \ref{thm:b2} and Corollary \ref{cor:b3simplyconnected}.
\end{proof}

We now prove Corollary~\ref{cor:sphere} from \S\ref{sec:app}:
\begin{proof}[Proof of Corollary \ref{cor:sphere}]
To avoid confusion, in this proof (and this proof only) we will use $|X|$ to denote the topological realization of a simplicial complex $X$.
	We will show that the hypotheses of Corollary \ref{cor:vanishing} are fulfilled. The first and third hypotheses are fulfilled by direct assumption. For the second hypothesis, let $X$ be the intersection of $\cK$ with $\Delta(b)$ and let $Y=\bigcap_{v\in b} \link(v,\cK)$. Then $|\cK_b|$ is a deformation retract of $|\cK|\setminus |X*Y|$.

	If $\#b=1$ then $X$ is a single vertex, $|X*Y|$ is contractible, and so $|\cK|\setminus |X*Y|$ is contractible since $|\cK|$ is a sphere.
We now suppose $\#b\geq 2$. If $|Y|$ is contractible, then so is $|X*Y|$, and thus $|\cK|\setminus |X*Y|$ is contractible. 

Suppose instead that $Y$ is empty. If $\#b=2$, $X$ is two disconnected vertices and  $|\cK|\setminus |X*Y|$ is then a sphere with two points removed. This has trivial $H^1$ unless $\dim \cK=2$.

	If  instead $\#b=3$, $X$ is either three disconnected vertices or an edge and a disconnected vertex. Then $|\cK|\setminus \vert X*Y\vert$ is a sphere with two or three contractible sets removed; this is connected unless $\dim \cK =1$.

	The statement follows.
\end{proof}

\subsection{Triangulations of $S^1$ with $T^3=0$.}
%We consider triangulations of $S^1$ with $T^3=0$.
A triangulation of $S^1$ with $n$ vertices is simply the boundary of an $n$-gon.

\begin{example}\label{ex:pentagon}
	Let $\cK$ be the boundary of a $5$-gon, with vertices labeled consecutively $1,2,3,4,5$. 
We claim that $T^3(\cK)=0$.
	Similar to Example \ref{ex:hex}, proper links  of $\cK$ have no $T^3$, so we only need to show $T^3_{-\bb}(\cK)=0$ with $\bb\in \ZZ_{\geq 0}^\cV$. Moreover, by Proposition~\ref{prop:t3bound} and Theorem~\ref{thm:bound} we can assume $\bb$ of type $(1)$, $(1,1)$, or $(1,1,1)$. In the first two cases, $\cK_b$ is contractible so that $T_{-\bb}^i(\cK)$ vanishes when $i\geq 2$ (Theorems \ref{thm:b1} and \ref{thm:b2}). For $\bb$ of type $(1,1,1)$ and not containing two edges of $\cK$, $\cE_b$ is a $1$-simplex, $\cK_b$ consists of $2$ disjoint points, and $\widehat \cK_b$ is all of $\cE_b*\cK_b$, and hence contractible. Thus, $T_{-\bb}^i(\cK)=0$ in this case for $i=3,4$ (Theorem \ref{thm:b3}). It follows that $T^3(\cK)=0$.

	We claim that $T^4(\cK)$ also vanishes. To this end, by Proposition~\ref{prop:t4bound} and Theorem~\ref{thm:bound} we deduce that the only remaining cases to consider are $\bb$ of type $(1,1,1,1)$ or $(2,1,1)$.
    
    In the first case, the analysis is fairly similar to the case $b=\{1,2,3,4\}$ of Example~\ref{ex:hex1111}, except that $\cK_b^\std$ is just the join of $\Delta(\cE_b)$ with $\cK_b$, the latter of which just consists of a single vertex. In any case $\cK_b^\std$ is contractible, and $\eta:K_1\to K_0$ has the same form as in the $b=\{1,2,3,4\}$ case of Example \ref{ex:hex1111}. Hence, by Theorem~\ref{thm:1111}, $T_{-\bb}^4(\cK)=0$.

    Lastly, assume $\bb$ is of type $(2,1,1)$. By Theorem~\ref{thm:211}, without loss of generality we may take $\bb=(2,0,1,1,0)$ up to rotation. The complex $\cK_{1,\{1,3,4\}}$ is the simple path from $1$ to $4$, which is connected. Hence $T_{-\bb}^4(\cK)=0$ by Theorem \ref{thm:211}.
It follows that $T^4(\cK)=0$.
\end{example}

\begin{proposition}\label{prop:S1}
A triangulation $\cK$ of $S^1$ with $n$ vertices has $T^3(\cK)=0$ if and only if $n\leq 5$.
\end{proposition}
\begin{proof}
	We have shown in Example \ref{ex:pentagon} above that the triangulation $\cK$ with $5$ vertices has $T^3(\cK)=0$. Triangulations $\cK$ with $3$ or $4$ vertices also have $T^3(\cK)=0$, since the corresponding rings $S_\cK$ are complete intersection rings (either a single cubic, or two quadrics).

	On the other hand, we saw in Example \ref{ex:hex} that the triangulation $\cK$ with $6$ vertices has $T^3(\cK)\neq 0$. Since any triangulation $\cK'$ of $S^1$ arise from $\cK$ through iterated edge subdivision, $S_{\cK'}$ deforms to $S_{\cK}$ (\cite[Proposition 6.1]{ic2}) so semicontinuity of cotangent cohomology (cf.~\cite[\S1.1.7]{behnkechristophersen}) implies $T^3(\cK)\neq 0$ as well.
\end{proof}

\subsection{The associahedron}\label{sec:assoc}
Fix $n\geq 3$. Consider an  $n$-gon $\cR_n$, with vertices labeled $1,\ldots,n$ in counterclockwise order. Let $\cV$ be the set of ${n\choose 2}-n$ diagonals in this polygon. For $1\leq i<j\leq n$ we will refer to the diagonal with endpoints $i,j$ simply as $ij$.
Two diagonals are said to cross if they intersect in their relative interiors. We denote by $\cA_n$ the simplicial complex on the vertex set $\cV$ consisting of sets of diagonals that do not pairwise cross; it follows immediately from the definition that $\cA_n$ is a flag complex. Maximal faces of $\cA_n$ thus correspond to triangulations of $\cR_n$. The complex $\cA_3$ is the empty set, $\cA_4$ consists of two isolated vertices, $\cA_5$ is (coincidentally) the boundary of a pentagon, and $\cA_6$ is complex (i) from Figure \ref{fig:triangulations} in \S\ref{sec:spheres}. In general $\cA_n$ is an $(n-4)$-dimensional sphere, and is the boundary of the polytope dual to the associahedron on $n-1$ elements, see e.g.~\cite{assoc}.

We now prove Corollary \ref{cor:assoc}.
\begin{proof}[Proof of Corollary \ref{cor:assoc}]
	We show that for $n\geq 4$, $T^3(\cA_n)=0$ by induction on $n$. For $n=4$, $S_{\cA_4}$ is a hypersurface ring, so $T^i(\cA_4)$ vanishes for all $i\geq2$ (see Example~\ref{ex.onegenerator}). For $n=5$, the claim follows from Example~\ref{ex:pentagon}. We now assume that $n\geq 6$.
	
	We will show that the hypotheses of Corollary \ref{cor:sphere} are satisfied, adapting arguments from \cite[\S5]{ic1}.
First, consider any vertex $v$ of $\cA_n$; this corresponds to a diagonal $ij$ of $\cR_n$. In fact, $ij$ cuts $\cR_n$ into two smaller polygons, say an $n_1$-gon and an $n_2$-gon, and diagonals in $\link(v,\cA_n)$ are exactly the diagonals in either one of these two smaller polygons. Thus, 
\[
	\link(v,\cA_n)=\link(v,\cA_{n_1})*\link(v,\cA_{n_2})
\]
cf.~\cite[Lemma 5.4]{ic1}.
It follows by induction and \cite[Proposition 2.3]{ic1} that 
\[
T^3(\link(v,\cA_n))=0.
\]

Consider now vertices $v,w$ of $\cA_n$ with $\{v,w\}\notin \cA_n$. This corresponds to a pair of crossing diagonals of $\cR_n$. The convex hull of these diagonals is a quadrilateral inscribed in $\cR_n$; since $n\geq 6$ at least one of its edges must be a diagonal $u$ of $\cR_n$. The diagonal $u$ belongs to every maximal face of $\link(v,\cA_n)\cap \link(w,\cA_n)$, so this intersection of links is contractible.

Consider now vertices $u,v,w$ of $\cA_n$ with $\{u,v\}\notin \cA_n$ and $\{u,w\}\notin \cA_n$. This corresponds to three diagonals of $\cR_n$, at least of one of which crosses the other two (namely, the diagonal corresponding to the vertex $u$). The convex hull $Q$ of these diagonals is either a $5$-gon or $6$-gon inscribed in $\cR_n$. If $Q$ is all of $\cR_n$, then the intersection
\[\link(u,\cA_n)\cap \link(v,\cA_n)\cap \link(w,\cA_n)\]
is empty.
Otherwise, $Q$ has an edge $t$ which is also a diagonal of $\cR_n$, and every maximal face of the above intersection of links contains $t$; it follows that this intersection is contractible. Either way, the hypothesis \eqref{hyp:a} of the corollary is satisfied, since $\dim\cA_n\geq 2$ for $n\geq 6$.

To find a vertex $u'\in \link(u,\cA_n)$ as needed for hypothesis \eqref{hyp:b}, choose one end $x$ of the diagonal $u$. Between $x$ and the endpoints of $v$ and $w$, there must be a diagonal $u'$ that crosses $v$ or $w$.

We see that the hypotheses for Corollary \ref{cor:sphere} are satisfied, and so we conclude that $T^3(\cA_n)=0$.
\end{proof}

\begin{remark}\label{rem:assoct4}
	By Example \ref{ex:pentagon}, we see that $T^4(\cA_5)=0$. However, for $n>5$, $T^4(\cA_n)\neq 0$. It suffices to show this for $n=6$, since $\cA_6$ appears as a link in $\cA_n$ with $n\geq 6$.

	Let $\cK=\cA_6$, and consider $\bb$ of type $(1,1,1)$ with $b$ containing the three unique vertices of $\cK$ of valency $4$; the elements of $b$ are exactly the diagonals $14$, $25$, and $36$. The complex $\cK_b$ looks like the boundary of a prism over a triangle with the three quadrilateral faces removed, see Figure \ref{fig:assoct4}. The set $\cE_b$ consists of three elements, and for each edge $f$ of $\Delta(\cE_b)$, $\link(w_f,\cK)\cap\cK_b$ is the boundary of a quadrilateral.

	The complex $\widehat \cK_b$ is not shellable, so to show that $T_{-\bb}^4(\cK)\cong H_2(\widehat\cK_b,\KK)\neq 0$, we will count the number of faces of $\widehat\cK_b$ in each dimension.
	For any simplicial complex $X$, consider the polynomial $p_X(t)=1+\sum_j q_j t^{j+1}$ where $q_j$ is the number of $j$-dimensional faces. Then 
	\begin{align*}
	p_{\cK_b}(t)=1+6t+9t^2+2t^3\\
	p_{\Delta(\cE_b)^{(0)}}(t)=1+3t
\end{align*}
and
\begin{align*}
	p_{\link(w_f,\cK)\cap\cK_b}(t)=1+4t+4t^2
\end{align*}
for any edge $f$ of $\Delta(\cE_b)$.
It follows by construction of $\widehat\cK_b$ that
\begin{align*}
	p_{\widehat\cK_b}(t)=p_{\cK_b}(t)\cdot p_{\Delta(\cE_b)^{(0)}}(t)+3t^2p_{\link(f,\cK)\cap\cK_b}(t)\\
	=1+9t+30t^2+41t^3+18t^4.
\end{align*}
In other words, $\widehat\cK_b$ has $9$ vertices, $30$ edges, $41$ $2$-faces, and $18$ $3$-faces.
Now, we know that $\widehat\cK_b$ is connected, and $T^3_{-\bb}(\cK)=0$ by Corollary~\ref{cor:assoc}, so $\widetilde H^0(\widehat\cK_b,\KK)=H^1(\widehat\cK_b,\KK)=0$ by Theorem~\ref{thm:b3}. It follows that 
\[
\dim H^2(\widehat\cK_b,\KK)\geq 41-18-(30-9+1)=1 \;
\]
and thus $T_{-\bb}^4(\cK)\neq 0$ by Theorem~\ref{thm:b3}. More precisely, it is straightforward to check that the differential $C_3(\widehat\cK_b)\to C_2(\widehat\cK_b)$ is injective, so $T_{-\bb}^4(\cK)\cong \KK$.
\end{remark}

\begin{figure}
	\begin{tikzpicture}[scale=1]
		\draw[fill,lightgray] (0,0)--(1,0)--(0,1)--(0,0);
		\draw[fill,lightgray] (1.5,1)--(2.5,1)--(1.5,2)--(1.5,1);
		\draw (0,0)--(1,0)--(0,1)--(0,0);
		\draw (1.5,1)--(2.5,1)--(1.5,2)--(1.5,1);
		\draw (0,0) -- (1.5,1);
		\draw (1,0) -- (2.5,1);
		\draw (0,1) -- (1.5,2);
		\draw[fill] (0,0) circle [radius=0.05] node[below left] {$13$};
		\draw[fill] (1,0) circle [radius=0.05] node[below right] {$15$};
			\draw[fill] (0,1) circle [radius=0.05] node[above left] {$35$};
			\draw[fill] (1.5,1) circle [radius=0.05] node[above left] {$46$};
			\draw[fill] (2.5,1) circle [radius=0.05] node[right] {$24$};
			\draw[fill] (1.5,2) circle [radius=0.05] node[above right] {$26$};
	\end{tikzpicture}
	\caption{The complex $\cK_b$ from Remark \ref{rem:assoct4}}\label{fig:assoct4}
\end{figure}
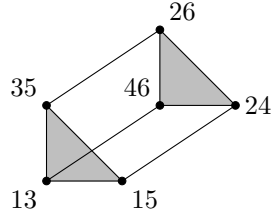

\subsection{Triangulations of $S^2$ with $T^3=0$.}\label{sec:spheres}
\begin{figure}
\begin{subfigure}{.2\textwidth}
	\centering
	\begin{tikzpicture}[scale=.8]
		\draw[fill] (-.5,1) circle [radius=0.05]; 
		\draw[fill] (.5,1) circle [radius=0.05]; 
		\draw[fill] (0,0) circle [radius=0.05]; 
	\draw (-.5,1) -- (.5,1) -- (0,0) -- (-.5,1);
	\draw (-.5,1) -- (-1,1.5);
	\draw (.5,1) -- (1,1.5);
	\draw (0,0) -- (0,-.5);
	\end{tikzpicture}
	\caption{}
\end{subfigure}
\begin{subfigure}{.2\textwidth}
	\centering
	\begin{tikzpicture}[scale=.8]
		\draw[fill] (0,.5) circle [radius=0.05]; 
		\draw[fill] (0,-.5) circle [radius=0.05]; 
		\draw[fill] (.5,0) circle [radius=0.05]; 
		\draw[fill] (-.5,0) circle [radius=0.05]; 
	\draw (-.5,0) -- (0,.5) -- (.5,0) -- (0,-.5)--(-.5,0);
	\draw (-1,0) -- (1,0);
	\draw (0,.5) -- (0,1);
	\draw (0,-.5) -- (0,-1);
	\end{tikzpicture}
	\caption{}
\end{subfigure}
\begin{subfigure}{.2\textwidth}
	\centering
	\begin{tikzpicture}[scale=.8]
		\draw[fill] (0,.5) circle [radius=0.05]; 
		\draw[fill] (0,-.5) circle [radius=0.05]; 
		\draw[fill] (.5,0) circle [radius=0.05]; 
		\draw[fill] (-.5,0) circle [radius=0.05]; 
		\draw[fill] (0,0) circle [radius=0.05]; 
	\draw (-.5,0) -- (0,.5) -- (.5,0) -- (0,-.5)--(-.5,0);
	\draw (-1,0) -- (1,0);
	\draw (0,1) -- (0,-1);
	\end{tikzpicture}
	
	\caption{}
\end{subfigure}
\begin{subfigure}{.2\textwidth}
	\centering
	\begin{tikzpicture}[scale=.8]
		\draw[fill] (-.5,1) circle [radius=0.05]; 
		\draw[fill] (.5,1) circle [radius=0.05]; 
		\draw[fill] (0,.5) circle [radius=0.05]; 
		\draw[fill] (-1,0) circle [radius=0.05]; 
		\draw[fill] (1,0) circle [radius=0.05]; 
		\draw (-.5,1) -- (.5,1) -- (0,.5) -- (-.5,1);
		\draw (-.5,1) -- (-1,0) -- (0,.5) -- (1,0) -- (.5,1);
	\draw (-.5,1) -- (-1,1.5);
	\draw (.5,1) -- (1,1.5);
	\draw (0,.5) -- (0,-.5);
	\draw (-1,0) -- (-1.5,-.6);
	\draw (1,0) -- (1.5,-.6);
	\end{tikzpicture}
	\caption{}
\end{subfigure}
\begin{subfigure}{.2\textwidth}
	\centering
	\begin{tikzpicture}[scale=.8]
		\draw[fill] (-.5,1) circle [radius=0.05]; 
		\draw[fill] (.5,1) circle [radius=0.05]; 
		\draw[fill] (0,.5) circle [radius=0.05]; 
		\draw[fill] (-1,0) circle [radius=0.05]; 
		\draw[fill] (1,0) circle [radius=0.05]; 
		\draw[fill] (0,-.5) circle [radius=0.05]; 
		\draw (-.5,1) -- (.5,1) -- (0,.5) -- (-.5,1);
		\draw (-.5,1) -- (-1,0) -- (0,.5) -- (1,0) -- (.5,1);
	\draw (-.5,1) -- (-1,1.5);
	\draw (.5,1) -- (1,1.5);
	\draw (0,.5) -- (0,-.5);
	\draw (-1,0) -- (-1.5,-.6);
	\draw (1,0) -- (1.5,-.6);
	\draw (-1,0) -- (0,-.5) -- (1,0);
	\draw (0,-.5) -- (0,-1);
	\end{tikzpicture}
	\caption{}
\end{subfigure}
\begin{subfigure}{.2\textwidth}
	\centering
	\begin{tikzpicture}[scale=.8]
		\draw[fill] (-.5,1) circle [radius=0.05]; 
		\draw[fill] (.5,1) circle [radius=0.05]; 
		\draw[fill] (0,.5) circle [radius=0.05]; 
		\draw[fill] (-1,0) circle [radius=0.05]; 
		\draw[fill] (1,0) circle [radius=0.05]; 
		\draw[fill] (0,-.5) circle [radius=0.05]; 
		\draw (-.5,1) -- (.5,1) -- (0,.5) -- (-.5,1);
		\draw (-.5,1) -- (-1,0) -- (0,.5) -- (1,0) -- (.5,1);
	\draw (-.5,1) -- (-1,1.5);
	\draw (.5,1) -- (1,1.5);
	\draw (-1,0) -- (1,0);
	\draw (-1,0) -- (-1.5,-.6);
	\draw (1,0) -- (1.5,-.6);
	\draw (-1,0) -- (0,-.5) -- (1,0);
	\draw (0,-.5) -- (0,-1);
	\end{tikzpicture}
	\caption{}
\end{subfigure}
\begin{subfigure}{.2\textwidth}
	\centering
	\begin{tikzpicture}[scale=.8]
	\draw (-.5,1) -- (-1,1.5);
	\draw (.5,1) -- (1,1.5);
	\draw (-1,0) -- (1,0);
	\draw (0,.5) -- (0,-.5);
	\draw (-1,0) -- (-1.5,-.6);
	\draw (1,0) -- (1.5,-.6);
	\draw (-1,0) -- (0,-.5) -- (1,0);
	\draw (0,-.5) -- (0,-1);
		\draw (-.5,1) -- (.5,1) -- (0,.5) -- (-.5,1);
		\draw (-.5,1) -- (-1,0) -- (0,.5) -- (1,0) -- (.5,1);
		\draw[fill] (-.5,1) circle [radius=0.05]; 
		\draw[fill] (.5,1) circle [radius=0.05]; 
		\draw[fill] (0,.5) circle [radius=0.05]; 
		\draw[fill] (-1,0) circle [radius=0.05]; 
		\draw[fill] (1,0) circle [radius=0.05]; 
		\draw[fill] (0,-.5) circle [radius=0.05]; 
		\draw[fill] (0,0) circle [radius=0.05]; 
	\end{tikzpicture}
	\caption{}
\end{subfigure}
\begin{subfigure}{.2\textwidth}
	\centering
	\begin{tikzpicture}[scale=.8]
		\draw[fill] (-.5,1) circle [radius=0.05]; 
		\draw[fill] (.5,1) circle [radius=0.05]; 
		\draw[fill] (0,.4) circle [radius=0.05]; 
		\draw[fill] (-1,0) circle [radius=0.05]; 
		\draw[fill] (1,0) circle [radius=0.05]; 
		\draw (-.5,1) -- (.5,1) -- (0,.4) -- (-.5,1);
		\draw (-.5,1) -- (-1,0) -- (0,.4) -- (1,0) -- (.5,1);
	\draw (-.5,1) -- (-1,1.5);
	\draw (.5,1) -- (1,1.5);
	\draw (-1,0) -- (1,0);
	\draw (-1,0) -- (-1.5,-.6);
	\draw (1,0) -- (1.5,-.6);
	\draw (-1,0) -- (0,-.5) -- (1,0);
	\draw (0,-.5) -- (0,-1);
	\draw (0,.75) -- (-.5,1);
	\draw (0,.75) -- (.5,1);
	\draw (0,.75) -- (0,.4);
		\draw[fill,red] (0,-.5) circle [radius=0.05]; 
		\draw[fill,red] (0,.75) circle [radius=0.05]; 
	\end{tikzpicture}
	\caption{}
\end{subfigure}
\begin{subfigure}{.2\textwidth}
	\centering
	\begin{tikzpicture}[scale=.8]
		\draw[fill] (-.5,1) circle [radius=0.05]; 
		\draw[fill] (.5,1) circle [radius=0.05]; 
		\draw[fill] (0,.6) circle [radius=0.05]; 
		\draw[fill] (-.4,.3) circle [radius=0.05]; 
		\draw[fill] (.4,.3) circle [radius=0.05]; 
		\draw[fill] (-1,0) circle [radius=0.05]; 
		\draw[fill] (1,0) circle [radius=0.05]; 
		\draw[fill] (0,-.5) circle [radius=0.05]; 
		\draw (-.5,1) -- (.5,1) -- (0,.6) -- (-.5,1);
		\draw (-.5,1) -- (-1,0);
		\draw (1,0) -- (.5,1);
	\draw (-.5,1) -- (-1,1.5);
	\draw (.5,1) -- (1,1.5);
	\draw (-1,0) -- (-1.5,-.6);
	\draw (1,0) -- (1.5,-.6);
	\draw (-1,0) -- (0,-.5) -- (1,0);
	\draw (0,-.5) -- (0,-1);
	\draw (-.5,1) -- (-.4,.3) -- (0,.6) -- (.4,.3) -- (.5,1);
	\draw (-1,0) -- (-.4,.3) -- (.4,.3) -- (1,0);
	\draw (-.4,.3) -- (0,-.5) -- (.4,.3);
	\end{tikzpicture}
	\caption{}
\end{subfigure}
\begin{subfigure}{.2\textwidth}
	\centering
	\begin{tikzpicture}[scale=.8]
		\draw[fill] (-.5,1) circle [radius=0.05]; 
		\draw[fill] (.5,1) circle [radius=0.05]; 
		\draw[fill] (-.4,.3) circle [radius=0.05]; 
		\draw[fill] (.4,.3) circle [radius=0.05]; 
		\draw[fill] (-1,0) circle [radius=0.05]; 
		\draw[fill] (1,0) circle [radius=0.05]; 
		\draw[fill] (0,-.5) circle [radius=0.05]; 
		\draw (-.5,1) -- (.5,1) -- (0,.6) -- (-.5,1);
		\draw (-.5,1) -- (-1,0);
		\draw (1,0) -- (.5,1);
	\draw (-.5,1) -- (-1,1.5);
	\draw (.5,1) -- (1,1.5);
	\draw (-1,0) -- (-1.5,-.6);
	\draw (1,0) -- (1.5,-.6);
	\draw (-1,0) -- (0,-.5) -- (1,0);
	\draw (0,-.5) -- (0,-1.5);
	\draw (-.5,1) -- (-.4,.3) -- (0,.6) -- (.4,.3) -- (.5,1);
	\draw (-1,0) -- (-.4,.3) -- (.4,.3) -- (1,0);
	\draw (-.4,.3) -- (0,-.5) -- (.4,.3);
	\draw (-1,0)--(0,-1)--(1,0);
		\draw[fill,red] (0,-1) circle [radius=0.05]; 
		\draw[fill,red] (0,.6) circle [radius=0.05]; 
	\end{tikzpicture}
	\caption{}
\end{subfigure}
\begin{subfigure}{.2\textwidth}
	\centering
	\begin{tikzpicture}[scale=.8]
		\draw[fill] (-.5,1) circle [radius=0.05]; 
		\draw[fill] (.5,1) circle [radius=0.05]; 
		\draw[fill] (1,0) circle [radius=0.05]; 
		\draw[fill] (0,.2) circle [radius=0.05]; 
		\draw[fill] (.3,-.2) circle [radius=0.05]; 
		\draw[fill] (-1,0) circle [radius=0.05]; 
		\draw[fill] (-.3,-.2) circle [radius=0.05]; 
		\draw[fill] (-.4,.3) circle [radius=0.05]; 
		\draw[fill] (.4,.3) circle [radius=0.05]; 
		\draw (-.5,1) -- (.5,1); 
		\draw (-.5,1) -- (-1,0);
		\draw (1,0) -- (.5,1);
	\draw (-.5,1) -- (-1,1.5);
	\draw (.5,1) -- (1,1.5);
	\draw (-1,0) -- (-1.5,-.6);
	\draw (1,0) -- (1.5,-.6);
	\draw (-1,0) -- (0,-1) -- (1,0);
	\draw (0,-1.5) -- (0,-1);
	\draw (0,.6) -- (.4,.3) -- (.3,-.2)-- (-.3,-.2)--(-.4,.3)--(0,.6);
	\draw (-.5,1) --(0,.6) --(.5,1);
	\draw (-.5,1) -- (-.4,.3) -- (-1,0) -- (-.3,-.2) -- (0,-1);
	\draw (.5,1) -- (.4,.3) -- (1,0) -- (.3,-.2) -- (0,-1);
	\draw (-.3,-.2) -- (0,.2) -- (-.4,.3);
	\draw (.3,-.2) -- (0,.2) -- (.4,.3);
	\draw (0,.2) -- (0,.6);
		\draw[fill,red] (0,-1) circle [radius=0.05]; 
		\draw[fill,red] (0,.6) circle [radius=0.05]; 
	\end{tikzpicture}
	\caption{}
\end{subfigure}

	\caption{Triangulations of $S^2$ with vertex valencies at most $5$ (point at infinity not pictured). $T^3=0$ for (a)--(g) and (i).}\label{fig:triangulations}
\end{figure}
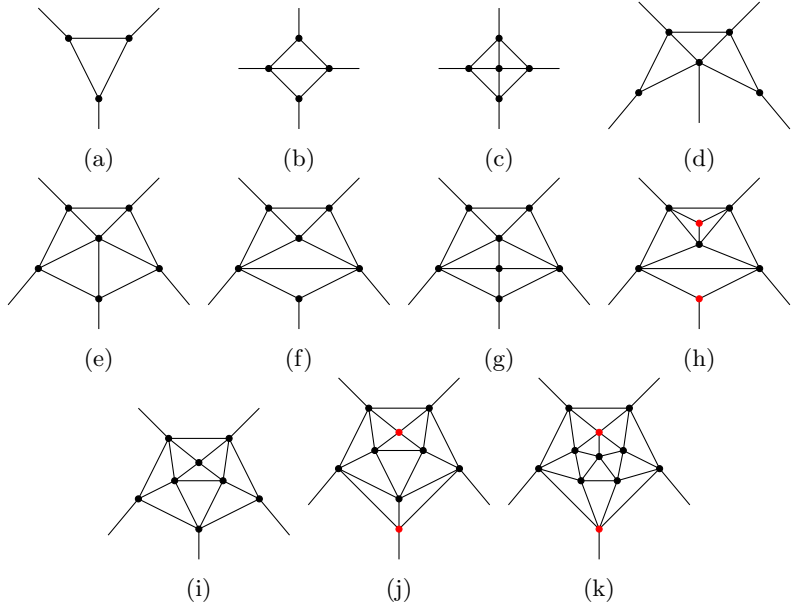

We now prove Corollary \ref{cor:spheret3}. By Proposition \ref{prop:S1} and Theorem \ref{thm:localizationformula}, we only need to consider those triangulations of $S^2$ where each vertex has valency at most $5$. These are pictured in Figure \ref{fig:triangulations}; see \cite[Figure 1]{ishida}. There is a vertex at infinity not pictured.

The triangulation (i) is just the complex $\cA_6$, so has $T^3=0$ by Corollary \ref{cor:assoc}. Starting with any of the triangulations (a)--(g), one may obtain (i) via a sequence of edge subdivisions, so the $T^3$ vanishing for (i) implies $T^3$ vanishing for these complexes as well by \cite[Proposition 6.1]{ic2} and semicontinuity of cotangent cohomology (cf.~\cite[\S1.1.7]{behnkechristophersen}).

It remains to see that $T^3$ is non-zero for (h), (j), and (k). For each of these complexes $\cK$, we consider $T^3_{-\bb}(\cK)$ with $\bb$ of type $(1,1)$ and $b$ given by the two red vertices in Figure \ref{fig:triangulations}. In each case, these two vertices have disjoint links, so $\cK_b$ is homotopy equivalent to $S^2$ with $2$ points removed. It follows by Theorem \ref{thm:b2} that $T^3_{-\bb}(\cK)\cong H^1(\cK_b,\KK)\cong \KK$, so in particular $T^3(\cK)\neq 0$. Note that Theorem \ref{thm:b2} applies on the nose in cases (j) and (k), as these are flag complexes. Case (h) is not a flag complex, but by Remark \ref{rem:nonflagb2} we may still apply the conclusion of Theorem \ref{thm:b2}. 

We know by Remark \ref{rem:assoct4} that $T^4\neq 0$ for the triangulation (i). 
We will see below in Example \ref{ex:G} that for (g), $T^4=0$. The vanishing of $T^4$ for the triangulations (a)--(f) then again follows by \cite[Proposition 6.1]{ic2} and semicontinuity of cotangent cohomology (cf.~\cite[\S1.1.7]{behnkechristophersen}).
\qed

\begin{remark}\label{rem:spheret4}
	Let $\cK$ be any flag complex that is a triangulation of $S^2$. Then for any $\bb\in\ZZ_{\geq 0}^\cV$ of type $(1)$, $(1,1)$, or $(2,1,1)$, $T^4_{-\bb}(\cK)=0$. Indeed, $H^i(\cK_b,\KK)=0$ for $i\geq 2$ since $\cK_b$ will be a proper subcomplex of $\cK$. This shows the claim if $\bb$ is of type $(1)$ or $(1,1)$ by Theorems \ref{thm:b1} and \ref{thm:b2}. For $\bb$ of type $(2,1,1)$, the complex $\cK_{u,b}$ in Theorem \ref{thm:211} is a retract of the complement in $\cK\cong S^2$ of $\{u\}*\bigcap_{v\in b} \link(v,\cK)$. This latter complex is contractible, so its complement is as well, and we obtain $T^4_{-\bb}(\cK)=0$ if $\bb$ is of type $(2,1,1)$ by Theorem \ref{thm:211}.
\end{remark}

\begin{example}\label{ex:G}
	We consider the simplicial complex $\cK$ given by the triangulation (g) of $S^2$ from Figure \ref{fig:triangulations}, and show that $T^4(\cK)=0$. 
	By Example \ref{ex:pentagon}, we know that for any proper link of $\cK$ we have $T^4=0$, so by Theorem \ref{thm:localizationformula} it suffices to show that $T^4_{-\bb}(\cK)=0$ for $\bb\in\ZZ_{\geq 0}^\cV$. By Proposition \ref{prop:t4bound} coupled with Remark \ref{rem:spheret4} above, we only need to additionally consider $\bb$ of type $(1,1,1)$ and $(1,1,1,1)$. 
Throughout, we will let
\[L_b=\bigcap_{v\in b} \link(v,\cK)\]
and 
\[
	\cK_{|b}=\cK\cap \Delta(b).
\]

Consider $\bb$ of type $(1,1,1)$. By Theorem \ref{thm:b3}, we may assume that at most one edge of $\cK$ is a subset of $b=\{u,v,w\}$. 
Since each vertex has valency $4$ or $5$ and $\cK$ has $8$ vertices, two of the elements of $b$ must form an edge in $\cK$.
Thus, without loss of generality $\{v,w\}\in\cK$.
We next claim that $L_b$  must have either one or two vertices. Indeed, since $v,w$ are connected by an edge, the intersection of their links consists of two vertices, say $z_1,z_2$. The link $\link(u,\cK)$ cannot contain $u$, $v$, or $w$, and since it has at least four elements, must also contain $z_1$ or $z_2$. If $L$ is a single vertex, then $\cK_{|b}*L_b$ is contractible, and it follows that $\cK_b$ is contractible since it is a retract of the complement of $\cK_{|b}*L_b$ in $\cK$. If instead $L_b$ has two vertices, then $\cK_b$ has three vertices and at most two connected components, each of which is contractible. It follows that $\Delta(\cE_b)^{(0)}*\cK_b$ is either contractible or homotopy equivalent to $S^1$. In all cases, we have $H^2(\cE_b)^{(0)}*\cK_b,\KK)=0$. Moreover, in all cases, for any $z\in b$, $H^1(\link(z,\cK_b),\cK)=0$. It now follows by Lemma \ref{lemma:b3seq} that $T^4_{-\bb}=0$. 

We now consider instead $\bb$ of type $(1,1,1,1)$. By Theorem \ref{thm:1111} we may assume that at most three edges of $\cK$ are subsets of $b$. Similar to in type $(1,1,1)$, we see that $L_b$ can have at most two vertices (although it might now also be empty). We now note that for any $\gamma\in \Gamma_3(b)$ and splitting $\Split$, 
\[
\link(\cap \Split,\cK)\cap \cK_b
\]
will be non-empty. Indeed, 
$\link(\cap \Split,\cK)$ contains at least four vertices, at least $3$ of which must be outside of $b$. Since $\cK$ has eight vertices and $L_b$ has at most two, we see that $\link(\cap \Split,\cK)$ must contain a vertex of $\cK_b$. 

Suppose that $L_b$ is non-empty. 
If $L_b$ has a single vertex, then $L_b$ is clearly contractible. Suppose instead that  $L_b$ has two vertices, say $u,v$. Since the link of $u$ is a $4$- or $5$-cycle, $\cK_b$ must be either a simple path of length $4$ or a $4$-cycle. But if it were the latter, $\cK_{b}*L_b$ would be the boundary of an octahedron, which cannot be contained in $\cK$. Hence, $\cK_{|b}$ must be contractible.
Either way, as long as $L_b$ is non-empty, we may conclude that $\cK_{|b}*L_b$ is contractible, which similar to above, implies that $\cK_b$ is also contractible. Applying Corollary \ref{cor:b4vanishing} we obtain $T^4_{-\bb}(\cK)=0$.

It remains to consider the case that $L_b=\emptyset$. Note that now  $\cK_b$ has $4$ vertices. Then $\cK_b$ cannot be disconnected, since this would imply that $\cK_{|b}$ is a $4$-cycle, but we assume that at most three edges of $\cK$ are subsets of $b$.
Hence $\cK_b$ is connected, and applying Corollary \ref{cor:b4vanishing} we obtain that $T^4_{-\bb}(\cK)=0$.
\end{example}
\subsection{The dual quotient bundle}\label{sec:grass}
We now give a sample geometric application of the vanishing of higher cotangent cohomology.
For any $n\in \ZZ$, $n\geq 3$, let $G(2,n)$ be the Grassmannian parametrizing $2$-dimensional subspaces of $\KK^n$. The \emph{quotient bundle} $\cQ$ is the rank $n-2$ vector bundle whose fiber at $x\in G(2,n)$ is $\KK^n/L_x$, where $L_x\subseteq \KK^n$ is the linear space corresponding to the point $x$. We thus have a surjection of vector bundles
\[
G(2,n)\times \KK^n \to \cQ.
\]
Let $\cQ^*$ be the dual vector bundle, i.e.~ the \emph{dual quotient bundle}. The above surjection dualizes to an inclusion
\[
\cQ^*\hookrightarrow G(2,n)\times (\KK^n)^*.
\]

The Grassmannian $G(2,n)$ has a natural embedding in $\PP^{ {n\choose 2}-1}$ sending a point $x$ to the $2\times 2$ minors of any matrix whose rowspan is $L_x$; this is the \emph{Pl\"ucker embedding}.  This induces an embedding of $\cQ^*$ in $\PP^{ {n\choose 2}-1}\times \AA^n$. Taking the projectivization\footnote{We use the convention that the projectivization of a vector bundle means taking all lines in each fiber (as opposed to codimension-one quotients).} 
 $X$ of $\cQ^*$, we obtain an embedding of $X$ in $\PP^{ {n\choose 2}-1}\times \PP^{n-1}$.

In \cite[\S2.2]{iltenturo}, Turo and the first author describe a certain simplicial complex $\cK_n$ for $n\geq 5$ obtained by an edge subdivision of $\cA_n*\cA_2$. The complex $\cK_n$ is related to the dual quotient bundle $\cQ^*$ by the following:

\begin{theorem}[{\cite[Proposition 2.5 \& Theorem 3.1]{iltenturo}}]\label{thm:Qdegen}
The 
Stanley-Reisner ideal of the join of $\cK_n$ with a $(2n-4)$-dimensional simplex is an initial ideal of the ideal of $\cQ^*$ in $\PP^{ {n\choose 2}-1}\times \AA^n$.
\end{theorem}

\begin{lemma}\label{lemma:t3k}
	For $n\geq 5$, the simplicial complex $\cK_n$ of \cite[\S2.2]{iltenturo} satisfies $T^3(\cK_n)=0$.
\end{lemma}
\begin{proof}
	By \cite[Theorem 6.2]{ic2}
	a sequence of edge subdivisions transforms $\cK_n$ into $\cA_{n+1}$. By \cite[Theorem 6.1]{ic2} and semicontinuity of cotangent cohomology (cf.~\cite[\S1.1.7]{behnkechristophersen}), the vanishing of $T^3(\cA_{n+2})$ from Corollary \ref{cor:assoc} implies the vanishing of $T^3(\cK_n)$.
\end{proof}

We may now extend \cite[Theorem 6.5]{iltenturo} from the case of hypersurfaces to codimension-two complete intersections:
\begin{cor}\label{cor:codim2}For $n\geq 6$, let $X\subsetneq \PP^{{n \choose 2}-1}\times \PP^{n-1}$ be the projectivization of the  dual of the quotient bundle on the Grassmannian $G(2,n)$, with embedding induced by the Pl\"ucker embedding. Let $Y$ be a codimension-two complete intersection in $X$ of hypersurfaces in  $\PP^{{n \choose 2}-1}\times \PP^{n-1}$.
Then any abstract small deformation of $Y$ is again a complete intersection in $X$.
\end{cor}
\begin{proof}
	By Lemma \ref{lemma:t3k} and \cite[Proposition 2.3]{ic1}, $T^3$ vanishes for the Stanley-Reisner ring of the join of $\cK_n$ and a $(2n-4)$-dimensional simplex. The same is also true for $T^2$ (cf.~\cite[Theorem 2.4]{iltenturo}).
	Let $J$ be the ideal of $\cQ^*$ in $\PP^{ {n\choose 2}-1}\times \AA^n$. By Theorem \ref{thm:Qdegen} and semicontinuity of cotangent cohomology (cf.~\cite[\S1.1.7]{behnkechristophersen}), $T^2$ and $T^3$ vanishes for the homogeneous coordinate ring of $X$.
	Moreover, $T^1$ vanishes as well by \cite[Theorem 5.1]{iltenturo}.
	The claim of our corollary now follows from  \cite[Lemma 6.6]{iltenturo}; the hypotheses  (1) and (2) from loc.~cit. hold by the same arguments used in the proof of \cite[Theorem 6.5]{iltenturo}. 
\end{proof}
\begin{remark}
	As was pointed out to us by Enrico Fatighenti, the dual quotient bundle $\cQ^*$ is isomorphic to the partial flag variety $\Flag(1,n-2,n)$ parametrizing partial flags consisting of nested $1$- and $(n-2)$-dimensional subspaces of $\KK^n$. We believe that the Borel-Bott-Weil theorem can be used to give an alternate proof of Corollary~\ref{cor:codim2} and extend beyond the codimension two case.
\end{remark}
\bibliographystyle{alpha}
\bibliography{paper}

\end{document}